\documentclass[oneside,a4paper,english, 11pt]{amsart}

\usepackage{adjustbox}
\usepackage[latin9]{inputenc}
\usepackage{hyperref}
\usepackage{lscape}
\usepackage{slashbox}
\usepackage{mathrsfs}
\usepackage[title]{appendix}
\usepackage{enumerate}
\usepackage{enumitem}
\usepackage{xcolor}
\usepackage{bm}
\usepackage[sort=def]{glossaries}
\usepackage{appendix}

\makeatletter

 \theoremstyle{plain}
\newtheorem{thm}[subsubsection]{Theorem}
 \theoremstyle{plain}

\theoremstyle{plain}

\theoremstyle{plain}

\theoremstyle{plain}

\theoremstyle{plain}

\theoremstyle{plain}
  \newtheorem{prop}[subsubsection]{Proposition}
\theoremstyle{plain}
 \newtheorem{lemma}[subsubsection]{Lemma}
\theoremstyle{plain}

\theoremstyle{plain}

\theoremstyle{plain}

\theoremstyle{definition}
  \newtheorem{defn}[subsubsection]{Definition}
\theoremstyle{definition}
  
 \theoremstyle{definition}

\theoremstyle{remark}
\newtheorem{rmk}[subsubsection]{Remark}
\usepackage{amsmath}
\usepackage{amssymb}
\usepackage{amscd}
\usepackage{amsthm}
\usepackage{array,multirow}
\usepackage[arrow,matrix,tips,all,cmtip,poly,color]{xy}

\usepackage{stmaryrd}
\usepackage{url}
\usepackage{tikz-cd}
\usepackage{color}
\usepackage{caption}
\usepackage{float}

\pdfpageheight\paperheight
\pdfpagewidth\paperwidth
\newcommand{\Z}{\mathbb{Z}}

\newcommand{\Q}{\mathbb{Q}}

\newcommand{\Qp}{\mathbb{Q}_p}

\newcommand{\R}{\mathbb{R}}
\newcommand{\C}{\mathbb{C}}

\newcommand{\F}{\mathbb{F}}

\newcommand{\N}{\mathbb{N}}

\newcommand{\cA}{\mathcal{A}}

\newcommand{\cC}{\mathcal{C}}

\newcommand{\cG}{\mathcal{G}}

\newcommand{\cI}{\mathcal{I}}

\newcommand{\cJ}{\mathcal{J}}

\newcommand{\cM}{\mathcal{M}}

\newcommand{\cO}{\mathcal{O}}

\newcommand{\cT}{\mathcal{T}}
\newcommand{\cU}{\mathcal{U}}

\newcommand{\cX}{\mathcal{X}}

\newcommand{\eps}{\varepsilon}

\newcommand{\phz}{\varphi}

\newcommand{\Zp}{\mathbb{Z}_p}

\newcommand{\Gal}{\mathrm{Gal}}
\newcommand{\Hom}{\mathrm{Hom}}

\newcommand{\Res}{\mathrm{Res}}

\newcommand{\GL}{\mathrm{GL}}

\newcommand{\Spec}{\mathrm{Spec}\ }

\newcommand{\Fp}{\F_p}

\newcommand{\un}[1]{\underline{#1}}
\renewcommand{\bf}[1]{\mathbf{#1}}
\newcommand{\Rep}{\mathrm{Rep}}
\newcommand{\Diag}{\mathrm{Diag}}

\newcommand{\tld}[1]{\widetilde{#1}}

\newcommand{\JH}{\mathrm{JH}}

\newcommand{\rbar}{\overline{r}}
\newcommand{\rhobar}{\overline{\rho}}

\newcommand{\Spf}{\mathrm{Spf}}
\newcommand{\semis}{\mathrm{ss}}
\newcommand{\speci}{\mathrm{sp}}

\newcommand{\Adm}{\mathrm{Adm}}
\newcommand{\orient}{\textrm{or}}

\newcommand{\rG}{\mathrm{G}}

\newcommand{\defeq}{\stackrel{\textrm{\tiny{def}}}{=}}
\newcommand{\s}{^\times}

\newcommand{\ovl}[1]{\overline{#1}}

\newif\iffinalrun
\iffinalrun
  \newcommand{\mar}[1]{}
\else
  \newcommand{\mar}[1]{\marginpar{\raggedright\tiny #1}}

\DeclareMathOperator{\Mod}{Mod}
\DeclareMathOperator{\Coh}{Mod}

\DeclareMathOperator{\Mat}{Mat}

\newcommand{\ra}{\rightarrow}

\newcommand{\iarrow}{\hookrightarrow}
\newcommand{\into}{\hookrightarrow}

\newcommand{\risom}{\buildrel\sim\over\rightarrow}

\title{Colength one deformation rings}

\author{Daniel Le}
\address{
Department of Mathematics,
Purdue University,
150 N. University Street,
West Lafayette, Indiana 47907, USA
}
\email{ledt@purdue.edu}

\author{Bao V.~Le Hung}
\address{Department of Mathematics,
Northwestern University,
2033 Sheridan Road,
Evanston, Illinois 60208, USA}
\email{lhvietbao@googlemail.com}

\author{Stefano Morra}
\address{Universit\'e Paris 8, Laboratoire d'Analyse, G\'eom\'etrie et Applications,  LAGA, Universit\'e Sorbonne Paris Nord, CNRS, UMR 7539,  F-93430, Villetaneuse, France}
\email{morra@math.univ-paris13.fr}

\author{Chol Park}
\address{Department of Mathematical Sciences, Ulsan National Institute of Science and Technology,
UNIST-gil 50, Ulsan 44919, South Korea}
\email{cholpark@unist.ac.kr}

\author{Zicheng Qian}
\address{Morningside Center of Mathematics, No.55, Zhongguancun East Road, Beijing, 100190, China }
\email{qianzicheng@amss.ac.cn}

\begin{document}

\begin{abstract}
Let $K/\Qp$ be a finite unramified extension, $\rhobar:\Gal(\ovl{\Q}_p/K)\ra\GL_n(\ovl{\F}_p)$ a continuous representation, and $\tau$ a tame inertial type of dimension $n$.
We explicitly determine, under mild regularity conditions on $\tau$, the potentially crystalline deformation ring $R^{\eta,\tau}_{\rhobar}$ in parallel Hodge--Tate weights $\eta=(n-1,\cdots,1,0)$ and inertial type $\tau$ when the \emph{shape} of $\rhobar$ with respect to $\tau$ has colength at most one.
This has application to the modularity of a class of shadow weights in the weight part of Serre's conjecture.
Along the way we make unconditional the local-global compatibility results of \cite{PQ}.
\end{abstract}
\maketitle

\section{Introduction}

In recent years, calculations of various potentially crystalline deformation spaces have seen a number of applications to questions of local-global compatibility in the mod $p$ and $p$-adic Langlands program.
This includes the weight part of Serre's conjecture, the determination of mod $p$ multiplicities, conjectures of Breuil on integral structures in $K$-types, and generalizations of Colmez's functor (see e.g.~\cite{EGS,LLLM,LLLM2,DL,MLM,LGC,GL3Wild,OBW,HW,BHHMS,BHHMS2}).
Under a somewhat exotic genericity condition, \cite{MLM} shows that tamely potentially crystalline deformation spaces are equisingular to certain closed subvarieties of Pappas--Zhu local models.
Not much is known about the geometry of these local models for Galois deformation spaces in general.
Moreover, it is difficult in practice to make the genericity condition explicit or to work with their natural presentations.

\subsection{The main result}

The local model has a stratification indexed by \emph{admissible} elements of the extended affine Weyl group called \emph{shapes} and the complexity of the geometry increases as the length of the shape decreases.
\cite{OBW} shows that when the shape is \emph{extremal} i.e.~has maximal length ($\binom{n+1}{3}$ for $n$-dimensional representations of $\Gal(\ovl{\Q}_p/\Qp)$), then the corresponding tamely potentially crystalline deformation ring is formally smooth.
The main result of this paper, which we state only for representations of $\Gal(\ovl{\Q}_p/\Qp)$ in the introduction, is the following:

\begin{thm}[Theorem \ref{thm: main deformation ring}]\label{thm:intro:main}
Let $E$ be a finite extension of $\Qp$, with ring of integers $\cO$ and residue field $\F$.
Let $\rhobar:\Gal(\ovl{\Q}_p/\Qp) \ra \GL_n(\F)$ a continuous Galois representation, $\tau$ a $3n-7$-generic tame inertial type, $\eta=(n-1,\cdots,1,0)\in\Z^n$, and $R^\tau_{\rhobar}$ the lifting ring for potentially crystalline representations of type $(\eta,\tau)$.
If $R^{\eta,\tau}_{\rhobar}$ is nonzero and the length of the shape $\tld{w}(\rhobar,\tau)$ is at least $\binom{n+1}{3}-1$ $($i.e.~the colength of the shape is at most one$)$, then $R^{\eta,\tau}_{\rhobar}$ is formally smooth over $\cO$ or $\cO[\![X,Y]\!]/(XY-p)$.
\end{thm}
\begin{rmk}
\label{rmk:intro}
\begin{enumerate}
\item
\label{it:rmk:intro:1}
Replacing $\Qp$ by a finite unramified extension $K$ and requiring that the shape have colength one at each embedding $K\into \ovl{\Q}_p$, we show that the deformation ring is formally smooth over a completed tensor product of rings of the form $\cO[\![X,Y]\!]/(XY-p)$.
The number of such factors in the tensor product can furthermore be explicitly computed.
\item The colength one deformation spaces that have a parabolic structure were computed in \cite{OBW}. In general, colength one deformation spaces do not have a parabolic structure, making their computation far more difficult.
\item
\label{it:rmk:intro:3}
When $n=3$, the tame inertial types with colength one shape are sufficient to prove the Serre weight conjecture for $\GL_3$ \cite{LLLM,LLLM2,GL3Wild}.
We generalize these ideas to prove the modularity of Serre weights of \emph{defect at most one} under the assumptions of \cite{OBW}, in particular under an explicit combinatorial genericity condition.
\item Using standard Taylor--Wiles techniques, Theorem \ref{thm:intro:main} gives modularity lifting results similar to \cite[Theorem 9.2.1]{MLM} (improving the polynomial genericity and the tameness condition at $p$ in \emph{loc.~cit.}, but imposing specific conditions on the shape with respect to the tame inertial types).
\end{enumerate}
\end{rmk}

While Theorem \ref{thm:intro:main} generalizes some previous results, its proof is perhaps surprisingly subtle despite the shape having close to maximal length.
We do not expect our methods to extend to shapes of smaller length.
This suggests that local models for Galois deformation spaces are genuinely complicated geometric objects and that simple explicit descriptions are hard to come by.

Our principal motivation in writing this paper was to apply Theorem \ref{thm:intro:main} to prove the weight elimination and mod $p$ multiplicity one results necessary to make unconditional the local-global compatibility result of \cite{PQ} which states roughly that the local mod $p$ Galois representation at~$p$ can be recovered from the $\GL_n(\Q_p)$-action on the Hecke isotypic part of the mod $p$ completed cohomology of a definite unitary group.
While the results of \cite{PQ} were superseded by those of \cite{LGC}, the method of \cite{LGC} using only extremal shapes does not work for $\mathrm{GSp}_4(\Q_p)$ while it should be possible to adapt the methods of \cite{PQ} (which builds on \cite{BD,HLM,LMP,MP} in small rank) to many $p$-adic reductive groups over $\Q_p$.
Indeed, this has been carried out for $\mathrm{GSp}_4(\Q_p)$ \cite{EL}.
For generalizations of \cite{LGC}, an analogue of Theorem \ref{thm:intro:main} should prove useful.
We hope to return to this in future work.

\subsection{Global and local applications}
As mentioned in Remark \ref{rmk:intro}\eqref{it:rmk:intro:3}, as a more immediate global application of Theorem \ref{thm:intro:main} we obtain the modularity of weights of defect at most one.
The notion of \emph{defect} of a Serre weight $\sigma$ for a \emph{tame} Galois representation $\rhobar$ was first introduced in \cite[\S 8.6]{MLM}.
This notion is purely combinatorial, and encodes the maximal length for the shapes $\tld{w}(\rhobar,\tau)$ such that $\sigma\in W^?(\rhobar)\cap \JH(\ovl{\sigma(\tau)})$.
In this paper we generalize the notion of defect for any $\rhobar$ in terms of specializations (as done in \cite{OBW} for extremal weights), and prove their modularity when the defect is at most one, conditional to the existence of a modular obvious weight.
The result is the following, and we refer the reader to the bulk of the paper of any undefined notion:
\begin{thm}[Theorem \ref{thm:globalobv}]
\label{thm:main:GA:intro}
Let $F/F^+$ be a CM field.
Assume that $F^+\neq \Q$, that all places of $F^+$ above $p$ are unramified over $\Q_p$ and  totally split in $F$.
Let $\rbar: G_{F^+} \ra \cG(\F)$ be a continuous representation which is automorphic in the sense of \cite[Definition 5.5.1]{OBW}, with set of modular weights $W(\rbar)$.
Let $\rbar_p$ be the $L$-homomorphism attached to the collection $\{\rbar|_{G_{F^+_v}}\}_{v|p}$ and write $W_{\leq \un{0}}^g(\rbar_p)$ and $W_{\leq \un{1}}^g(\rbar_p)$ for the set of extremal weights and for the set of weights of defect at most one, respectively, for $\rbar_p$.
Asssume further that:
\begin{itemize}
\item $\rbar(G_{F(\zeta_p)})\subseteq \GL_n(\F)$ is adequate; and
\item $\rbar_p$ is $6(n-1)$-generic.
\end{itemize}
Then the following are equivalent:
\begin{enumerate}
\item $W_{\leq \un{0}}^g(\rbar_p) \cap W(\rbar) \neq \emptyset$; and
\item $W^g_{\leq \un{1}}(\rbar_p) \subset W(\rbar)$.
\end{enumerate}
\end{thm}
Compared to \cite[Theorem 5.5.5]{OBW}, Theorem \ref{thm:main:GA:intro} assumes that $F^+$ is unramified above $p$, but it gives the modularity of defect one weights.
(For $\GL_3$, this is sufficient to prove the generic Serre weight conjecture.)

\subsection{Acknowledgements}
Part of the work was carried out during a visit at the Universit\`a degli Studi di Padova (2019), which we would like to heartily thank for the excellent working conditions which provided to us.

D.L. was supported by the National Science Foundation under agreements Nos.~DMS-1128155 and DMS-1703182 and an AMS-Simons travel grant.
B.LH. acknowledges support from the National Science Foundation under grant Nos.~DMS-1128155, DMS-1802037 and the Alfred P. Sloan Foundation.
S.M. was supported by the ANR-18-CE40-0026 (CLap CLap) and the Institut Universitaire de France.
C.P. was supported by Samsung Science and Technology Foundation under Project Number~SSTF-BA2001-02.

\subsection{Notation}

For a field $K$, we denote by $\ovl{K}$ a fixed separable closure of $K$ and let $G_K \defeq \Gal(\ovl{K}/K)$.
If $K$ is defined as a subfield of an algebraically closed field, then we set $\ovl{K}$ to be this field.

If $K$ is a nonarchimedean local field, we let $I_K \subset G_K$ denote the inertial subgroup and $W_K \subset G_K$ denote the Weil group.
We fix a prime $p\in\Z_{>0}$.
Let $E \subset \ovl{\Q}_p$ be a subfield which is finite-dimensional over $\Q_p$.
We write $\cO$ to denote its ring of integers, fix an uniformizer $\varpi\in \cO$ and let $\F$ denote the residue field of $E$.
We will assume throughout that $E$ is sufficiently large.

\subsubsection{Reductive groups}
\label{sec:not:RG}
Let $G$ denote a split connected reductive group (over some ring) together with a Borel $B$, a maximal split torus $T \subset B$, and $Z \subset T$ the center of $G$.
Let $d = \dim G - \dim B$.
When $G$ is a product of copies of $\GL_n$, we will take $B$ to be upper triangular Borel and $T$ the diagonal torus.
Let $\Phi^{+} \subset \Phi$ (resp. $\Phi^{\vee, +} \subset \Phi^{\vee}$) denote the subset of positive roots (resp.~positive coroots) in the set of roots (resp.~coroots) for $(G, B, T)$.
We use the notation $\alpha > 0$ (resp.~$\alpha < 0$) for a positive (resp.~negative) root $\alpha\in \Phi$.
Let $\Delta$ (resp.~$\Delta^{\vee}$) be the set of simple roots (resp.~coroots).
Let $X^*(T)$ be the group of characters of $T$, and set $X^0(T)$ to be the subgroup consisting of characters $\lambda\in X^*(T)$ such that $\langle\lambda,\alpha^\vee\rangle=0$ for all $\alpha^\vee\in \Delta^{\vee}$.
Let  $W(G)$ denote the Weyl group of $(G,T)$.  Let $w_0$ denote the longest element of $W(G)$.
We sometimes write $W$ for $W(G)$ when there is no chance for confusion.
Let $W_a$ (resp.~$\tld{W}$) denote the affine Weyl group and extended affine Weyl group
\[
W_a = \Lambda_R \rtimes W(G)\quad\mbox{and} \quad \tld{W} = X^*(T) \rtimes W(G)
\]
for $G$, respectively.
We use $t_{\nu} \in \tld{W}$ to denote the image of $\nu \in X^*(T)$.

The Weyl groups $W(G)$, $\tld{W}$, and $W_a$ act naturally on $X^*(T)$.
If $A$ is any ring, then the above Weyl groups act naturally on $X^*(T)\otimes_{\Z} A$ by extension of scalars. Let $M$ be a free $\Z$-module of finite rank (e.g. $M=X^*(T)$).
The duality pairing between $M$ and its $\Z$-linear dual $M^*$ will be denoted by $\langle \ ,\,\rangle$.
If $A$ is any ring, the pairing $\langle \ ,\,\rangle$ extends by $A$-linearity to a pairing between $M\otimes_{\Z}A$ and $M^*\otimes_{\Z}A$, and by an abuse of notation it will be denoted with the same symbol $\langle \ ,\,\rangle$.

We write $G^\vee = G^\vee_{/\Z}$ for the split connected reductive group over $\Z$ determined by the root datum $(X_*(T),X^*(T), \Phi^\vee,\Phi)$.
This defines a maximal split torus $T^\vee\subseteq G^\vee$ such that we have canonical identifications $X^*(T^\vee)\cong X_*(T)$ and $X_*(T^\vee)\cong X^*(T)$.

Let $V\defeq X^*(T)\otimes_{\Z}\R$.
For $(\alpha,k)\in \Phi \times \Z$, we have the root hyperplane $H_{\alpha,k}\defeq \{x \in V \mid \langle\lambda,\alpha^\vee\rangle=k\}$.
An \emph{alcove} is a connected component of $V \setminus\ \big(\bigcup_{(\alpha,n)}H_{\alpha,n}\big)$, and we denote by $\cA$ the set of alcoves. We say that an alcove $A$ is \emph{restricted} if $0<\langle\lambda,\alpha^\vee\rangle<1$ for all $\alpha\in \Delta$ and $\lambda\in A$.
We let $A_0$ denote the (dominant) base alcove, i.e.~the set of $\lambda\in X^*(T)\otimes_{\Z}\R$ such that $0<\langle\lambda,\alpha^\vee\rangle<1$ for all $\alpha\in \Phi^+$.
Recall that $\tld{W}$ acts transitively on the set of alcoves, and $\tld{W}\cong\tld{W}_a\rtimes \Omega$ where $\Omega$ is the stabilizer of $A_0$.
We define
$$
\tld{W}^+\defeq\{\tld{w}\in \tld{W}\mid \tld{w}(A_0) \textrm{ is dominant}\}\quad \mbox{and}\quad
\tld{W}^+_1\defeq\{\tld{w}\in \tld{W}^+\mid \tld{w}(A_0) \textrm{ is restricted}\}.
$$
We fix an element $\eta\in X^*(T)$ such that $\langle \eta,\alpha^\vee\rangle = 1$ for all positive simple roots $\alpha$ and let $\tld{w}_h$ be $w_0 t_{-\eta}\in \tld{W}^+_1$.

When $G = \GL_n$, we fix an isomorphism $X^*(T) \cong \Z^n$ in the standard way, where the standard $i$-th basis element $\eps_i\defeq(0,\ldots, 1,\ldots, 0)$ (with the $1$ in the $i$-th position) of the right-hand side corresponds to extracting the $i$-th diagonal entry of a diagonal matrix.
In particular, we can write any root $\beta\in \Phi$ as $\beta=\eps_i-\eps_j$ for uniquely chosen $1\leq i,j\leq n$, $i\neq j$.

Given a finite set $\cJ$ and an isomorphism $G\stackrel{\sim}{\ra}\GL_n^{\cJ}$ we use superscripts in the notations above, e.g.~${\Phi}^{+,\cJ}\subset{\Phi}^{\cJ}$, ${\Delta}^{\cJ}$, $X^*({T})^{\cJ}$, ${W}^{\cJ}$, etc., where now $\Phi^+, \Phi, \Delta, X^*(T), {W}$, etc.~are relative to $\GL_n$.
In order not to overload notations, we do not use underlined notations for the elements of ${\Phi}^{\cJ}$, $X^*({T})^{\cJ}$, ${W}^{\cJ}$, etc., so that for instance a root $\alpha\in{\Phi}^{\cJ}$ is in fact a collection of roots $(\alpha^{(j)})_{j\in\cJ}$ where each $\alpha^{(j)}$ is a root of $\GL_n$.
Finally, we take $\eta \in X^*({T})^{\cJ}$ to correspond to the element $(n-1, n-2, \ldots, 0)_{j\in\cJ} \in (\Z^n)^{\cJ}$ in the identification above.
When an element $j\in\cJ$ is fixed, we will abuse notation and will use the same symbol $\eta$ to denote the element which corresponds to the tuple $(n-1,\dots,1,0)$ at $j$.

We let $F^+_p$ be a finite unramified \'etale $\Q_p$-algebra so that $F^+_p$ is isomorphic to a product $\prod_{S_p} F^+_{v}$ over a finite set $S_p$ where, for each $v\in S_p$, $F_{v}^+$ is a finite unramified extension of $\Q_p$.
For each $v\in S_p$ let $\cO_{F^+_{v}} $ be the ring of integers of $ F^+_{v}$, $k_{v}$ the residue field and let $\cO_p$ (resp.~$k_p$) be the product $\prod_{v\in S_p} \cO_{F_v^+}$ (resp.~$\prod_{v\in S_p} k_{v}$).
(This will be used in global applications, where $S_p$ will be a finite set of places dividing $p$ of a number field $F^+$.)

If $G$ is a split connected reductive group over $\F_p$, with Borel $B$ and maximal torus $T$, we let $G_0 \defeq \Res_{k_p/\F_p} G_{/k_p}$, and similarly define $B_0$, $T_0$. %
We will always assume that $\F$ contains the image of any ring homomorphism $k_p \ra \ovl{\F}_p$ so that we can and do fix an isomorphism $\big(G_0\times_{\Spec \Fp}\Spec \F\big)\stackrel{\sim}{\rightarrow}(G\times_{\Spec \Fp}\Spec\F)^{\cJ_p}$ where $\cJ_p$ denotes the set of ring homomorphisms $k_p \ra \F$.
For notational convenience, we will write $\un{G}\defeq G_0\times_{\Spec \Fp}\Spec \F$, and similarly for $\un{B}$, $\un{T}$.
The notations $\un{W}$, $\un{\tld{W}}$, $\un{\tld{W}}^+$, $\un{\tld{W}}^{+}_1$, etc.~as well as the identifications $\un{W}\stackrel{\sim}{\ra}W^{\cJ_p}$, $\un{\tld{W}}\stackrel{\sim}{\ra}\tld{W}^{\cJ_p}$, $\un{\tld{W}}^+\stackrel{\sim}{\ra}(\tld{W}^{+})^{\cJ_p}$, $\un{\tld{W}}^{+}_1\stackrel{\sim}{\ra}(\tld{W}^{+}_1)^{\cJ_p}$, etc.~should be clear.

\subsubsection{Galois theory}
\label{sec:not:GT}
Let $K$ be a finite extension of $\Qp$, with residue field $k$ of degree $f$ over $\Fp$.
We assume that $K/\Qp$ is unramified and write $W(k)$ for the ring of Witt vectors, which is also the ring of integers of $K$.
The arithmetic Frobenius automorphism on $W(k)$, acting as raising to $p$-th power on the residue field will be denoted by $\phz$.
We fix an embedding $\sigma_0$ of $K$ into $E$ (equivalently an embedding $k$ into $\F$) and define $\sigma_j = \sigma_0 \circ \phz^{-j}$.
This gives an identification between $\cJ=\Hom(k,\F)$ and $\Z/f\Z$.

We normalize Artin's reciprocity map $\mathrm{Art}_{K}: K\s\ra W_{K}^{\mathrm{ab}}$ so that uniformizers are sent to geometric Frobenius elements. We fix once and for all a sequence $ (p_{m})_{m\in \N}\in \ovl{K}^{\N}$ satisfying $p_{m+1}^{p}=p_{m}$, $p_{0}\defeq -p\in K$ and let $K_\infty$ be $\underset{m\in\N}{\bigcup}K(p_{m})$.

Given an element $\pi_1 \defeq (-p)^{\frac{1}{p^{f}-1}}\in \overline{K}$ we have a  character $\omega_{K}:I_K \ra W(k)^{\times}$ defined by the condition $g(\pi_1)=\omega_{K}(g)\pi_1$.
Using our choice of embedding $\sigma_0$ this gives a fundamental character of niveau $f$
\[
\omega_{f}:= \sigma_0 \circ \omega_{K}:I_K \ra \cO^{\times}.
\]

Let $\rho: G_K\rightarrow \GL_n(E)$ be a $p$-adic, de Rham Galois representation.
For $\sigma: K\iarrow E$, we define $\mathrm{HT}_\sigma(\rho)$ to be the multiset of $\sigma$-labeled Hodge-Tate weights of $\rho$, i.e.~the set of integers $i$ such that $\dim_E\big(\rho\otimes_{\sigma,K}\C_p(-i)\big)^{G_K}\neq 0$ (with the usual notation for Tate twists).
In particular, the cyclotomic character $\eps$ has Hodge--Tate weights 1 for all embedding $\sigma:K\iarrow E$.
For $\mu=(\mu^{(j)})_j\in X^*(\un{T})$ we say that $\rho$ has Hodge--Tate weighs $\mu$ if for all $j\in\cJ$
\[
\mathrm{HT}_{\sigma_j}(\rho)=\{\mu_{1}^{(j)},\mu_{2}^{(j)},\dots,\mu_{n}^{(j)}\}.
\]
The \emph{inertial type} of $\rho$ is the isomorphism class of $\mathrm{WD}(\rho)|_{I_K}$, where $\mathrm{WD}(\rho)$ is the Weil--Deligne representation attached to $\rho$ as in \cite{CDT}, Appendix B.1 (in particular, $\rho\mapsto\mathrm{WD}(\rho)$ is \emph{covariant}).
An inertial type is a morphism $\tau: I_K\ra \GL_n(E)$ with open kernel and which extends to the Weil group $W_K$ of $G_K$.
We say that $\rho$ has type $(\mu,\tau)$ if $\rho$ has Hodge--Tate weights $\mu$ and inertial type given by (the isomorphism class of) $\tau$.

\subsubsection{Miscellaneous}
Finally, $\delta_P$ denotes the Kronecker delta function on the condition $P$. We also use $\delta$ for the defect function in \S\ref{subsec: App to patching}. This shall cause no confusion.

\section{Preliminaries}

\subsection{Affine Weyl groups, tame inertial types, Serre weights}

\subsubsection{Affine Weyl group}

We collect here the necessary background to give a classification of colength one elements in the admissible set (Proposition \ref{prop:classification, colength 1 dual shapes}).

Recall from \S \ref{sec:not:RG} that $G$ is a split reductive group with split maximal torus $T$.
We write $W$ for the Weyl group associated to $(G,T)$ and $V\defeq X^*(T) \otimes \R \cong X_*(T^{\vee}) \otimes \R$ for the apartment of $(G, T)$ on which $\tld{W}\defeq X^*(T) \rtimes W$ acts.
We write $\cC_0$ for the dominant Weyl chamber in $V$.

Recall that $\cA$ denotes the set of alcoves of $X^*(T) \otimes \R$ and that $A_0 \in \cA$ denotes the dominant base alcove.
We let $\uparrow$ denote the upper arrow ordering on alcoves as defined in \cite[\S II.6.5]{RAGS} which induces the ordering $\uparrow$ on $W_a$ via the bijection $W_a \risom \cA$ given by $\tld{w} \mapsto \tld{w}(A_0)$.
The dominant base alcove $A_0$ defines a set of simple reflections in $W_a$ and thus a Coxeter length function on $W_a$ denoted $\ell(-)$ and a Bruhat order on $W_a$ denoted by $\leq$.
Given $\lambda \in X^*(T)$ we consider the set of $\lambda$-admissible elements of $\tld{W}$:
\begin{equation}
\label{eq:adm:lbd}
\Adm(\lambda) \defeq  \left\{ \tld{w} \in \tld{W}\mid \tld{w} \leq t_{w(\lambda)} \text{ for some } w \in W \right\}.
\end{equation}

If $\Omega \subset \tld{W}$ is the stabilizer of the base alcove, then $\tld{W} = W_a \rtimes \Omega$ and so $\tld{W}$ inherits a Bruhat order in the standard way: For $\tld{w}_1, \tld{w}_2\in W_a$ and $\delta\in \Omega$, $\tld{w}_1\delta\leq \tld{w}_2\delta$ if and only if $\tld{w}_1\leq \tld{w}_2$ , and elements in different right $W_a$-cosets are incomparable.
We extend $\ell(-)$ to $\tld{W}$ by letting $\ell(\tld{w}\delta)\defeq \ell(\tld{w})$ for any $\tld{w}\in W_a$, $\delta\in \Omega$.

Let $(\tld{W}^\vee,\leq)$ be the following partially ordered group: $\tld{W}^\vee$ is identified with $\tld{W}$ as a group, and $\ell(-)$ and $\leq$ are defined with respect to the \emph{antidominant} base alcove.
If $\tld{w}=t_\nu w\in \tld{W}$ with $w\in W$ and $\nu\in X^*(T)$ we define $\tld{w}^*\defeq w^{-1}t_{\nu}\in \tld{W}^\vee$.
(The assignement $\tld{w}\mapsto\tld{w}^*$ defines a bijection which preserves length and Bruhat order, see \cite[Lemma 2.1.3]{LLL}; we also denote the inverse bijection by the same symbol $\tld{w}\mapsto\tld{w}^*$.)
Given $\lambda \in X^*(T)$ we define the set $\Adm^\vee(\lambda)$ by replacing $\tld{W}$ by $\tld{W}^\vee$ in the right hand side of \eqref{eq:adm:lbd}.

We also recall that, given $m\in\Z$ and $\alpha\in\Phi$, the $m$-th $\alpha$-strip is the subset of $V$ defined by
\[
\{x\in V \mid m<\langle x,\alpha^\vee\rangle<m+1\}.
\]
Finally, we say that $\tld{w}\in\tld{W}$ is \emph{regular} if it is in the sense of \cite[Definition~2.1.3]{MLM}.

\begin{prop}
\label{prop:classification, colength 1 dual shapes} %
Suppose $\tld{w}\in \Adm(\eta)$ such that $\ell(\tld{w})=\ell(t_{\eta})-1$. Then one of the following holds:
\begin{enumerate}
\item
\label{it:colength:1:1}
$\tld{w}=w t_{\eta}s_\alpha w^{-1}$ where $w\in W$ and $\alpha>0$ a positive root such that $w(\alpha)>0$, and there are \emph{no} decompositions $\alpha=\beta_1+\beta_2$ with $\beta_i>0$ such that $w(\beta_i)>0$;
\item
\label{it:colength:1:2}
$\tld{w}=w t_{\eta-\alpha}s_\alpha w^{-1}$ where $w\in W$, $\alpha\in \Phi^+\setminus \Delta$ such that $w(\alpha)<0$, and for any decomposition $\alpha=\beta_1+\beta_2$ with $\beta_i>0$, we have $w(\beta_i)<0$.
\end{enumerate}
\end{prop}
We say that a colength one element $\tld{w}\in \Adm(\eta)$ is \emph{of the first form} (resp.~of the \emph{second form}) is it is as in item \eqref{it:colength:1:1} (resp.~as in item \eqref{it:colength:1:2}) of Proposition \ref{prop:classification, colength 1 dual shapes}.
Recalling \cite[Definition 2.1.3]{MLM} we note that a colength one element $\tld{w}\in \Adm(\eta)$ is irregular exactly when it is of the first form and moreover the root $\alpha$ appearing in \eqref{it:colength:1:1} is a simple root.

\begin{rmk}
\label{rmk:expl:w}
If $\alpha=\eps_i-\eps_j\in\Phi^+\setminus\Delta$ then
\begin{enumerate}
\item
\label{it:rmk:expl:w:1}
the condition in item \eqref{it:colength:1:1} means that $w$ preserves the order of $i$ and $j$, and $w$ maps no element $k\in (i,j)$ into an element in $(w(i),w(j))$;
\item
\label{it:rmk:expl:w:2}
the condition in item \eqref{it:colength:1:2} means that $w$ reverses the order of $i$ and $j$, and maps any $k\in (i,j)$ to an element in $(w(j),w(i))$.
\end{enumerate}
\end{rmk}

\begin{proof}
We assume $\tld{w}(A_0)$ is in chamber $w(\cC_0)$. Then there is a gallery from $A_0$ to $\tld{w}(A_0)$ in the $w$-positive direction (cf.~\cite[\S 2, Definition 5.2]{HC}), hence \cite[Corollary 4.4]{HH} shows that $\tld{w}<wt_{\eta}w^{-1}$. Thus $\tld{u}\defeq w^{-1}\tld{w}$ is dominant so that  $w\tld{u}$ is a reduced expression for $\tld{w}$ by \cite[Lemma 2.2.1]{OBW} (cf.~Definition 2.1.2 in \emph{loc.~cit}.~for the notion of reduced expression).
We conclude that $\tld{u}\leq t_{\eta}w^{-1}$ and $\ell(\tld{u})=\ell(t_{\eta}w^{-1})-1$. Since $\tld{u}$ and $t_{\eta}w^{-1}$ are both dominant and their lengths differ by $1$ they must differ by an affine reflection in direction $\alpha$ for some positive root $\alpha>0$ by \cite[Corollary A.1.2]{GHS}. We claim that the $\alpha$-strip containing $\tld{u}$ is the lower neighbor of the $\alpha$-strip containing $t_{\eta}w^{-1}$.

Suppose the contrary. Then for each vertex $v$ of $A_0$, the line segment joining $\tld{u}(v)$ and $ t_{\eta}w^{-1}(v)$ lies in the dominant Weyl chamber and contains $t_\alpha\tld{u}(v)$.
This shows that $t_{\alpha}\tld{u}(A_0)$ is dominant and $\tld{u} \uparrow t_\alpha\tld{u}\uparrow t_{\eta}w^{-1}$. \cite[Theorem 4.3]{wang} implies that $\tld{u}<t_{\alpha}\tld{u}<t_\eta w^{-1}$ which contradicts $\ell(\tld{u})=\ell(t_{\eta}w^{-1})-1$.

There are two cases:
\begin{itemize}
\item
$\tld{u}(A_0)$ and $t_\eta w^{-1}(A_0)$ share the vertex $\eta$.
In this case $\tld{u}=t_\eta s_{\alpha}w^{-1}$.
The conditions $\tld{u}\leq t_\eta w^{-1}$ and $\ell(\tld{u})=\ell(t_\eta w^{-1})-1$ are equivalent to $\ell(s_\alpha w^{-1})=\ell(w^{-1})+1$, equivalently $\ell(ws_\alpha)=\ell(w)+1$. But this condition translates exactly to the condition in item \eqref{it:colength:1:1}.
\item $\tld{u}(A_0)$ has $\eta-\alpha$ as a vertex. In this case we must have $\tld{u}=t_{\eta-\alpha}s_\alpha w^{-1}$.
To see the condition on $w$, we note that for each positive root $\beta$ the difference between the $\beta$-heights of $t_{\eta-\alpha}s_\alpha w^{-1}$ and $t_\eta w^{-1}$ is
\[ \langle \eta, \beta^\vee \rangle +\delta_{w(\beta)>0}-1 - (\langle \eta, \beta^\vee \rangle-\langle \alpha,\beta^\vee\rangle +\delta_{ws_\alpha(\beta)>0}-1)=\langle \alpha, \beta^\vee \rangle +\delta_{w(\beta)>0}-\delta_{ws_\alpha(\beta)>0}.\]
Thus the colength $1$ condition becomes
\[\sum _{\beta \neq \alpha, \beta>0} -\delta_{w(\beta)>0}+\delta_{ws_\alpha(\beta)>0}=2\langle \eta,\alpha \rangle -2\]
which gives the condition in item \eqref{it:colength:1:2}.%
\end{itemize}
This completes the proof.
\end{proof}

\begin{lemma}
\label{lem:fact:reg}
Let $\tld{w}\in\Adm^{\mathrm{reg}}(\eta)$ satisfy $\ell(\tld{w})\geq\ell(t_\eta)-1$, and let $\tld{w}_2^{-1}w_0t_\nu \tld{w}_1$ be the unique up to $X^0(T)$ factorization of $\tld{w}$ with $\tld{w}_1,\tld{w}_2\in \tld{W}^+_1$ and $\nu \in X^+(T)$ $($see \cite[Proposition 2.1.5]{MLM}$)$.
Suppose that $\tld{w}=\tld{x}_2^{-1}s\tld{x}_1$ with $\tld{x}_1,\tld{x}_2\in \tld{W}^+$, $s\in W$, and that $\tld{x}_1\uparrow\tld{x}\uparrow\tld{w}_h^{-1}\tld{x}_2$ with $\tld{x} \in \tld{W}^+_1$.
Then $s=w_0$ and either $\tld{x} = \tld{x}_1 \in \tld{w}_1X^0(T)$ or $\tld{x} = \tld{w}_h^{-1} \tld{x}_2 \in \tld{w}_h^{-1}\tld{w}_2X^0(T)$.
If moreover $\ell(\tld{w})=\ell(t_\eta)$ we further have $\tld{x} = \tld{x}_1=\tld{w}_h^{-1} \tld{x}_2$ and $\nu\in X^0(T)$.
\end{lemma}
\begin{proof}
Let $\tld{w},\tld{x}_1,\tld{x}_2,s,$ and $\tld{x}$ be as above.
Since $\tld{x}_1\uparrow\tld{x}\uparrow\tld{w}_h^{-1}\tld{x}_2$, we have that $\tld{x}_1 \leq \tld{x}$ and $\tld{x}_2 \leq \tld{w}_h \tld{x}$ by \cite[Theorem 4.1.1, Proposition 4.1.2]{LLL} as $\tld{w}_h \tld{x} \in \tld{W}^+$.
Then we have
\begin{align*}
1&\geq\ell(t_\eta)-\ell(\tld{w})\\
&\geq \ell((\tld{w}_h\tld{x})^{-1}w_0\tld{x}) -[\ell(\tld{x}_2^{-1})+\ell(s)+\ell(\tld{x}_1)]\\
&= [\ell(\tld{w}_h\tld{x})-\ell(\tld{x}_2)]+ [\ell(w_0)-\ell(s)]+[\ell(\tld{x})-\ell(\tld{x}_1)]
\end{align*}
where the last equality follows from \cite[Lemma 4.1.9]{LLL}.
Since each term in the last expression is nonnegative, we immediately see that $\tld{x} = \tld{x}_1$ or $\tld{w}_h^{-1} \tld{x}_2$.
(If moreover $\ell(\tld{w})=\ell(t_\eta)$ we further have $s=w_0$ and $\tld{x}_1=\tld{x}=\tld{w}_h^{-1}\tld{x}_2$.)
It suffices to show that $\tld{x}_1 \in \tld{w}_1X^0(T)$ in the former case and $\tld{x}_2 \in \tld{w}_2X^0(T)$ in the latter case.

We next show that $s=w_0$.
Supposing otherwise, $\tld{x}_1 = \tld{x} = \tld{w}_h^{-1}\tld{x}_2$ and $s = w_0s_{\alpha}$ with $\alpha\in \Delta^+$.
By Proposition \ref{prop:classification, colength 1 dual shapes} and the paragraph following it, we conclude that $\tld{w}$ is not regular which is a contradiction.

Now suppose that $\tld{x} = \tld{x}_1$ so that we have $\tld{x}_2^{-1} w_0 \tld{x}_1 = \tld{w}_2^{-1} w_0 t_\nu \tld{w}_1$.
The uniqueness in \cite[Proposition 2.1.5]{MLM} guarantees that $\tld{x}_1 \in \tld{w}_1X^0(T)$ (and, if moreover $\ell(t_\eta)=\ell(\tld{w})$, that $\nu\in X^0(T)$).
Similarly, if $\tld{x} = \tld{w}_h^{-1} \tld{x}_2$, then $\tld{x}_2 \in \tld{W}_1^+$ and uniqueness again guarantees that $\tld{x}_2 \in \tld{w}_2X^0(T)$.
\end{proof}

For the following lemma, which explicitly describes the unique up to $X^0(T)$-decomposition of colength one regular admissible elements mentioned in Lemma \ref{lem:fact:reg}, we specialize to the case $G=\GL_n$.
Recall that in this case we have the elements $\eps_i\in X^*(T)$ for $i=1,\dots,n$ and the fundamental weights $\omega'_{\eps_i-\eps_{i+1}}\defeq \sum_{k=1}^{i}\eps_i$ for $1\leq i\leq n-1$.
We have the inclusion
\begin{equation}
\label{eq:bij:W1}
W\into\tld{W}^+_1,\quad
w\longmapsto t_{\eta_w}w
\end{equation}
where $t_{\eta_w}$ is the dominant weight defined by
\begin{equation}
\label{eq:def:tw}
t_{\eta_w}\defeq \sum_{\substack{\beta\in\Delta,\, w^{-1}(\beta)<0}}\omega'_\beta
\end{equation}
(see \cite[equation (5.1)]{herzig-duke}).
Given $w\in W$ we write $\tld{w}$ for the image of $w$ via the inclusion \eqref{eq:bij:W1}.
\begin{lemma}
\label{lem:can:dec:easy}
Let $\tld{w}\in \Adm^{\textnormal{reg}}(\eta)$ satisfy $\ell(\tld{w})=\ell(t_\eta)-1$.
Then in the decomposition $\tld{w}=(\tld{w}_2)^{-1}w_0t_{\nu}\tld{w}_1$ of Lemma \ref{lem:fact:reg}, we can take $\tld{w}_1=\tld{(s_{\alpha}w^{-1})}$, $\tld{w}_2=\tld{w}_h\tld{w^{-1}}$ and $\nu=\eta_{w^{-1}}-\eta_{s_\alpha w^{-1}}-\eps\alpha$, where $\eps$ equals $1$ $($resp.~$0)$ if we are in case \eqref{it:colength:1:2} $($resp. case~\eqref{it:colength:1:1}$)$ of Proposition \ref{prop:classification, colength 1 dual shapes}.
\end{lemma}
\begin{proof}
Letting $\tld{w}=wt_{\eta-\eps\alpha}  s_\alpha w^{-1}$, where $\eps\in\{0,1\}$ equals $1$ (resp.~$0$) if we are in case \eqref{it:colength:1:2} (resp.~\eqref{it:colength:1:1}) of Proposition \ref{prop:classification, colength 1 dual shapes}, we have
\[
\tld{w}=(\tld{w}_h\tld{w^{-1}})^{-1}w_0t_\nu\tld{s_\alpha w^{-1}}
\]
by \eqref{eq:bij:W1} and the definition of $\tld{w}_h$.
It thus suffices to prove that $\nu$ is dominant, i.e.~using equation \eqref{eq:def:tw}, that
\begin{equation*}
\label{eq:dom:nu}
\delta_{w(\beta)>0}-\delta_{ws_\alpha(\beta)>0}-\eps\langle\alpha,\beta^\vee\rangle\geq 0
\end{equation*}
for all $\beta\in\Delta$.
This is an elementary casewise check, based on the value of $\langle\alpha,\beta^\vee\rangle\in\{-1,0,1\}$ and Remark \ref{rmk:expl:w}, and which we leave to the reader.
\end{proof}

\subsection{Local models, affine charts and monodromy conditions}
For any Noetherian $\cO$-algebra $R$ define
\begin{align*}
L \cG(R)&\defeq \{A \in \GL_n(R(\!(v+p)\!))\mid \text{ $A$ is upper triangular modulo $v$}\};\\
L^+\cM(R)&\defeq \{A \in \Mat_n(R[\![v+p]\!])\mid \text{ $A$ is upper triangular modulo $v$}\}.
\end{align*}

\subsubsection{Affine charts}
\label{subsubsec:AC}
Let $\tld{z}=zt_\nu\in \tld{W}^{\vee}$.
Given an integer $h\geq 0$ define the subfunctor ${\cU}(\tld{z})^{\det,\leq h}\subseteq  L\cG$ defined on Noetherian $\cO$-algebras $R$ as follows: ${\cU}(\tld{z})^{\det,\leq h}(R)$ is the set of matrices $A\in L\cG(R)$ such that for all $1\leq i,k\leq n$ the following holds:
\begin{itemize}
\item $(v+p)^{h}A_{ik}\in v^{\delta_{i>k}}R[v]$;
\item $\deg\big((v+p)^{h}A_{ik}\big)\leq h+\nu_{k}+\delta_{i>k}-\delta_{i<z(k)}$;
\item
$(v+p)^{h}A_{z(k)k}$ is a monic polynomial; and
\item
$\det (A)=\det(z)(v+p)^{|\!|\nu|\!|}$.
\end{itemize}
where we have written $\nu=(\nu_{\ell})_{1\leq\ell\leq n}\in X^*(T)$ and $|\!|\nu|\!|\defeq\sum_{\ell=1}^n\nu_\ell$.

We now introduce the subfunctors and $T^\vee$-torsors which are relevant for our analysis of deformation rings.
In the above setting, define the following subfunctors of $ {\cU}(\tld{z})^{\det,\leq 0}$:
\begin{itemize}
\item $U^{[0,n-1]}(\tld{z})\subseteq {\cU}(\tld{z})^{\det,\leq 0}$ as the subfunctor whose $R$-valued points are those $A\in {\cU}(\tld{z})^{\det,\leq 0}$ such that both $A$ and $(v+p)^{n-1}A^{-1}$ are in  $L^+\cM(R)$; note that $U^{[0,n-1]}(\tld{z})$ is representable by an affine scheme over $\cO$;
\item
$U(\tld{z},\leq\!\!\eta)\subseteq U^{[0,n-1]}(\tld{z})$ as the closed, $\cO$-flat and reduced subscheme of $U^{[0,n-1]}(\tld{z})$ whose $R$-valued points consist of $A\in U^{[0,n-1]}(\tld{z})(R)$ with elementary divisors bounded by $(v+p)^{\eta}$ (i.e.~each $k$ by $k$ minor of $A$ is divisible by $(v+p)^{\frac{(k-1)k}{2}}$); and
\item
$\tld{U}(\tld{z},\leq\!\!\eta)\defeq T^\vee U(\tld{z},\leq\!\!\eta)$ as the subscheme of
$L\cG$ whose $R$-valued points are of the form $DA$ where $D\in T^\vee(R)$ and $A\in U(\tld{z},\leq\!\!\eta)(R)$; note that $\tld{U}(\tld{z},\leq\!\!\eta)$ is endowed with a $T^\vee_\cO$-action induced by left multiplication of matrices, and we have $\big[T^{\vee,\cJ}\backslash\tld{U}(\tld{z},\leq\!\!\eta)\big]\cong U(\tld{z},\leq\!\!\eta)$.
\end{itemize}
The entries of $A\in U^{[0,n-1]}(\tld{z})(R)$ will typically be written in the form
\begin{equation}
\label{eq:expl:U}
A_{ik}= v^{\delta_{i>k}}\Bigg(\sum_{\ell=0}^{\nu_k-\delta_{i<z(k)} }a_{ik,\ell}(v+p)^\ell\Bigg)\nonumber
\end{equation}
so that $U^{[0,n-1]}(\tld{z})$ is an affine $\cO$-scheme whose global functions is the quotient of the polynomial $\cO$-algebra in the variables $a_{ik,\ell}$ (for the appropriate range of $\ell$ determined by $\tld{z}$ and $(i,k)$) by the ideal determined by the condition $\det (A)=\det(z)(v+p)^{|\!|\nu|\!|}$.
Then $U(\tld{z},\leq\!\!\eta)$ is a closed subscheme of $U^{[0,n-1]}(\tld{z})$, given by the $p$-saturation of the condition that each $k$ by $k$ minor is divisible by $(v+p)^{\frac{(k-1)k}{2}}$. In general, these schemes can be somewhat hard to describe, however in Proposition~\ref{prop: FH 1} and Proposition~\ref{prop: FH 2}, we will give an explicit presentations for them when $\tld{z}$ has colength one.

\subsubsection{Algebraic monodromy condition}

Let $\bf{a}\in\cO^{n}$.
We define an operator $\nabla_{\bf{a}}$ on $L\cG$ by $\nabla_{\bf{a}}(A)\defeq \big(v\frac{d}{dv}(A)-[\Diag(\bf{a}),A]\big)$, and a closed subfunctor $L\cG^{\nabla_{\bf{a}}}\subseteq L\cG$ by
\[
L\cG^{\nabla_{\bf{a}}}(R)\defeq\left\{
A\in L\cG(R)\mid \nabla_{\bf{a}}(A)A^{-1}\in \frac{1}{(v+p)}L^+\cM(R)
\right\}.
\]
(We write $[M,N]$ for the usual Lie bracket on $L\cG$, defined by $[M,N]\defeq MN-NM$.)

We define the following closed subschemes of $U(\tld{z},\leq\!\!\eta)$:
\begin{itemize}
\item $U^{\textnormal{nv}}(\tld{z},\leq\!\!\eta,\nabla_{\bf{a}})\defeq U(\tld{z},\leq\!\!\eta)\cap L\cG^{\nabla_{\bf{a}}}$; and

\item $U(\tld{z},\leq\!\!\eta,\nabla_{\bf{a}})\subseteq U^{\textnormal{nv}}(\tld{z},\leq\!\!\eta,\nabla_{\bf{a}})$ as the $p$-flat closure of $U^{\textnormal{nv}}(\tld{z},\leq\!\!\eta,\nabla_{\bf{a}})_E$ inside $U(\tld{z},\leq\!\!\eta)$.
\end{itemize}
Thus $U(\tld{z},\leq\!\!\eta,\nabla_{\bf{a}})$ is the closed subscheme $U(\tld{z},\leq\!\!\eta)$ whose ring of global functions is the quotient of $\cO(U(\tld{z},\leq\!\!\eta))$ by the $p$-saturation of the ideal $I_{\nabla_{\bf{a}}}$ obtained by imposing that the universal matrix
\[
A^{\textnormal{univ}}\in U(\tld{z},\leq\!\!\eta)\big(\cO(U(\tld{z},\leq\!\!\eta)\big)
\]
satisfies
\[\big(\nabla_{\bf{a}}(A^{\textnormal{univ}})\big)\big(A^{\textnormal{univ}}\big)^{-1}\in \frac{1}{(v+p)}L^+\cM(\cO(U(\tld{z},\leq\!\!\eta)),
\]
cf.~\cite[Definition 7.1.8]{MLM} for a list of generators of this ideal.
Proposition~\ref{prop: naive main, one} and Proposition~\ref{prop: naive main, zero} give results towards an explicit presentation for the affine scheme $U(\tld{z},\leq\!\!\eta,\nabla_{\bf{a}})$ when $\tld{z}$ has colength at most one.

We define analogously the closed subschemes $\tld{U}(\tld{z},\leq\!\!\eta,\nabla_{\bf{a}})\subseteq \tld{U}^{\textnormal{nv}}(\tld{z},\leq\!\!\eta,\nabla_{\bf{a}})\subseteq \tld{U}(\tld{z},\leq\!\!\eta)$, replacing ${U}(\tld{z},\leq\!\!\eta)$ by $\tld{U}(\tld{z},\leq\!\!\eta)$ in the above items. Note that we also have
\begin{equation}\label{eq: left torus action}
\tld{U}(\tld{z},\leq\!\!\eta,\nabla_{\bf{a}})=T^\vee U(\tld{z},\leq\!\!\eta,\nabla_{\bf{a}}),
\end{equation}
compatible with $\tld{U}(\tld{z},\leq\!\!\eta)\defeq T^\vee U(\tld{z},\leq\!\!\eta)$.

\subsubsection{Monodromy condition and Galois representations.}
\label{subsub:true_mon}
We follow the notation and terminology on tame inertial types and their lowest alcove presentations of \cite[\S 2.4]{MLM}.
In particular given
 $(s,\mu)=(s^{(j)},\mu^{(j)})_{j\in\cJ}\in W^{\cJ}\times (X^*(T)\cap C_0)^{\cJ}$,  \cite[Example 2.4.1, equations (5.2), (5.1)]{MLM} produces a tame inertial type $\tau(s,\mu+\eta)$ and $n$-tuples $\bf{a}^{\prime (j')}\in\Z^n$ for $0\leq j'\leq fr-1$ (where $r$ is the order of the element $s^{(0)}s^{(1)}\cdots s^{(f-2)}s^{(f-1)}\in W$).
If $\mu$ is $1$-deep in $\un{C}_0$, for each $0\leq j'\leq fr-1$ we define   $s'_{\orient,j'}$ to be the element of $W$ such that $(s'_{\orient,j'})^{-1}(\bf{a}^{\prime\,(j')})\in\Z^n$ is dominant, and let
\begin{equation}
\label{eq:monodromy_str_constant}
\bf{a}^{(j)}\defeq (s^{\prime\,(j)}_{\orient})^{-1}(\bf{a}^{\prime\,(j)})/(1-p^{fr})
\end{equation}
for $0\leq j\leq f$.

Recall from \cite[\S 7.2]{MLM} the $p$-adic formal algebraic stack $\cX^{\leq \eta,\tau}$ over $\Spf\cO$.
It is $p$-flat, equidimensional of dimension $(1+[K:\Qp]\binom{n}{2})$ and moreover if $\rhobar: G_{K}\ra\GL_n(\F)$ is a continuous Galois representation, then $R^{\leq \eta,\tau}_{\rhobar}$ is a versal ring for $\cX^{\leq \eta,\tau}$ at its point corresponding to $\rhobar$ (see \cite[\S 4.8]{EGstack}).
Let $\tld{z}=(\tld{z}^{(j)})_{j\in\cJ}\in \tld{W}^{\vee,\cJ}$.

Let $N\geq n+1$ and assume that $\mu$ is $N$-deep in alcove $\un{C}_0$.
From \cite[Theorem 7.2.3, Theorem 7.3.2]{MLM} we have an open substack $\cX^{\leq \eta,\tau}(\tld{z})\into \cX^{\leq \eta,\tau}$, an $\cO$-flat closed $p$-adic formal subscheme $\tld{U}(\tld{z},\leq\!\!\eta,\nabla_{\tau,\infty})$ of $\tld{U}(\tld{z},\leq\!\!\eta)^{\wedge_p}$ (where $\tld{U}(\tld{z},\leq\!\!\eta)\defeq\prod_{j\in \cJ}\tld{U}(\tld{z}^{(j)},\leq\!\!\eta)$ and the superscript $\wedge_p$ denotes the $p$-adic completion) and a formally smooth morphism
\begin{equation}
\label{diag:main}
\tld{U}(\tld{z},\leq\!\!\eta,\nabla_{\tau,\infty})\ra \cX^{\leq \eta,\tau}(\tld{z})
\end{equation}
of relative dimension $n\#\cJ$.

Assume further that $N\geq 2n-5$ and $\mu$ is $N$-deep in alcove $\un{C}_0$.
Then \cite[Proposition 7.1.10]{MLM} shows that the formal scheme $\tld{U}(\tld{z},\leq\!\!\eta,\nabla_{\tau,\infty})$ is the $p$-flat closure of a natural deformation of $\prod_{j\in \cJ}\tld{U}(\tld{z}^{(j)},\leq\!\!\eta,\nabla_{\bf{a}^{(j)}})$ in the following sense: Letting $I^{(j)}_{\nabla_{\bf{a}^{(j)}}}$ be the ideal cutting out $\tld{U}(\tld{z}^{(j)},\leq\!\!\eta,\nabla_{\bf{a}^{(j)}})$ in $\tld{U}(\tld{z}^{(j)},\leq\!\!\eta)$,
$\tld{U}(\tld{z},\leq\!\!\eta,\nabla_{\tau,\infty})$ is cut out by the $p$-saturation of an ideal $I_{\nabla_{\tau,\infty}}$ which is obtained from $\sum_{j\in \cJ} I^{(j)}_{\nabla_{\bf{a}^{(j)}}}$ by adding an element divisible by $p^{N-2n+5}$ to each element of its natural set of generators (in particular, $I_{\nabla}\subset \sum_j (I^{(j)}_{\nabla_{\bf{a}^{(j)}}},p^{N-2n+5})$.)
This property can be expressed (abusively) in the following way: Given a $p$-adically complete Noetherian $\cO$-algebra $R$, if $\tld{A}\in \tld{U}(\tld{z},\leq\!\!\eta)^{\wedge_p}(R)$ satisfies the equations in $I_{\nabla_{\tau,\infty}}$ then
\begin{equation}
\label{eq:true:mon:cond:A}
\big(\nabla_{\bf{a}^{(j)}}(\tld{A}^{(j)})\big)\big((v+p)^{n-1}\tld{A}^{(j)}\big)^{-1}\in (v+p)^{n-2}\mathrm{M}_n(R[v])+p^{N-2n+5}\mathrm{M}_n(R[\![v]\!]).
\end{equation}

Finally let us recall the characterization of the Galois representations that contribute to $\cX^{\leq \eta,\tau}(\tld{z})$: If $\rhobar\in \cX^{\leq \eta,\tau}(\tld{z})(\F)$, then its \'{e}tale $\varphi$ module $\cM_{\rhobar}=\mathbb{V}^*_K(\rhobar|_{G_{K_\infty}})$, cf.~\cite[\S 5.5]{MLM}, admits a basis with respect to which the matrix of Frobenius belongs to $\tld{U}(\tld{z},\leq\!\!\eta)(\F) \prod_{j\in \cJ}s_j^{-1}v^{\mu_j+\eta_j}$. We say that $\rhobar$ has shape $\tld{z}$ with respect to $\tau$ if additionally the matrix of Frobenius belongs to $\prod_{j\in \cJ} \cI\tld{z}^{(j)}\cI s_j^{-1}v^{\mu_j+\eta_j}$. (Here $\cI$ denotes the Iwahori subgroup of $\GL_n(\F(\!(v)\!))$ corresponding to the Borel of upper triangular matrices.)

It follows from \cite[Proposition 5.4.7]{MLM} that any $\rhobar\in\cX^{\leq \eta,\tau}(\F)$ has a unique shape in $\Adm(\eta)$, which we denote by $\tld{w}(\rhobar,\tau)^*$.

\section{Finite height conditions}\label{sec:finite height conditions}
In this section, we describe $U(\tld{w},\leq\!\!\eta)$ when $\tld{w}^*\in \Adm(\eta)$ has colength one.
Throughout this subsection $R$ denotes a Noetherian $\cO$-flat $\cO$-algebra, and we use the notation
$u_\beta:\mathbb{G}_a\hookrightarrow w_0Uw_0$ for the embedding to the $\beta$-entry for each $\beta\in\Phi^-$, where $U$ is the subgroup of upper triangular unipotent matrices in $\GL_n$.
Moreover, given $A\in\mathrm{Mat}_n(R)$ and $\beta=\eps_i-\eps_j\in \Phi$, we write $A_{\beta}$ and $A_{\beta,\ell}$ for $A_{ij}$ and $a_{ik,\ell}$ respectively.

\subsection{Finite height conditions: the first form}\label{subsec: height condition, first}
Let $\tld{w}^*\in \Adm(\eta)$ be a colength one shape of the first form.
Thus, by Proposition~\ref{prop:classification, colength 1 dual shapes} and the definition of $\tld{w}\mapsto \tld{w}^*$, we have $\tld{w}=ws_\alpha t_{\eta}w^{-1}$, where $w\in W$ satisfies $w(\alpha)>0$ and $w(k)$ do not belong to the interval $(w(i_0),w(j_0))$ for all $i_0<k<j_0$ if we set $\alpha=\alpha_{i_0j_0}\in \Phi^+\setminus\Delta$. For notational convenience we set $s\defeq s_\alpha$ and $w'\defeq ws$.

With an eye towards the monodromy conditions, we introduce the following:
\begin{defn}
We say that a negative root $\beta\in \Phi^-$ is \emph{bad} (with respect to $\tld{w}$) if either one of the following holds:
\begin{itemize}
\item $\beta$ shares either the row or the column of $\alpha$;
\item $\delta_{w(\beta)<0}\neq\delta_{w'(\beta)<0}$.
\end{itemize}
\end{defn}

\begin{prop}\label{prop: FH 1}
Let $A\in U(\tld{w},\leq\!\!\eta)(R)$.

Then $A_{w(j_0)w(j_0)}=(w^{-1}Aw)_{j_0j_0}=c(v+p)^{n-j_0}$ for some $c\in R$ and
\begin{equation}\label{eq: elementary operation 1}
s\cdot u_{-\alpha}(-c)\cdot (w^{-1}Aw)=(v+p)^\eta\cdot V
\end{equation}
where $V$ is a lower triangular unipotent matrix whose entries are in $R[v]$.

Moreover, for each $\beta\in\Phi^{-}$ the entry $V_\beta$ can be written as $v^{\kappa_\beta}\cdot f_{\beta}(v)$ for a polynomial $f_{\beta}(v)\in R[v]$ of the form $\sum_{k}c_{\beta,k}(v+p)^k$ where
\begin{enumerate}
\item if $\beta$ is bad then $\kappa_\beta=\delta_{w(\beta)<0}=0$ and  $\deg(f_{\beta})<|\langle\eta,\beta^\vee\rangle|+1$;
\item if $\beta$ shares the row or the column of $\alpha$ but is not bad then $\kappa_\beta=\delta_{w(\beta)<0}$ and $\deg(f_{\beta})<|\langle\eta,\beta^\vee\rangle|$;
\item if $\beta$ shares the row of $-\alpha$ then $\kappa_\beta=\delta_{w'(\beta)<0}$ and $\deg(f_{\beta})<|\langle\eta,\beta^\vee\rangle|$;
\item if $\beta$ shares the column of $-\alpha$ with $\beta\neq-\alpha$ then $\kappa_\beta=\delta_{w(\beta)<0}$ and $\deg(f_{\beta})<|\langle\eta,\beta^\vee\rangle|-\delta_{w(-\alpha)<w(\beta)<0}$;
\item for all other roots $\beta\in\Phi^-$, $\kappa_\beta=\delta_{w(\beta)<0}$ and $\deg(f_{\beta})<|\langle\eta,\beta^\vee\rangle|$.
\end{enumerate}
\end{prop}

\begin{proof}
It is easy to see that $(w^{-1}Aw)_{j_0j_0}$ is a polynomial of degree less than $n-j_0+1$. Let $c\in R$ be the coefficient of $v^{n-j_0}$ in $(w^{-1}Aw)_{j_0j_0}$, and set $C\defeq su_{-\alpha}(-c)(w^{-1}Aw)$ as in \eqref{eq: elementary operation 1}. One can readily observe the following degree bounds:
\begin{itemize}
\item $\deg (C_{ll})\leq n-l$ for all $1\leq l\leq n$;
\item for $\beta=\alpha_{lm}\in\Phi$ with $l,m\not\in\{i_0,j_0\}$, $C_\beta=v^{\delta_{w(\beta)<0}}f_\beta(v)$ with $\deg (f_\beta)\leq n-m-1$;
\item for $\beta=\alpha_{i_0 m}\in\Phi$, $C_\beta=v^{\delta_{w(\beta)<0}}f_{\beta}(v)$ with $\deg(f_\beta)\leq n-m-1+\delta_{w'(\beta)<0<w(\beta)}$;
\item for $\beta=\alpha_{j_0 m}\in\Phi$, $C_\beta=v^{\delta_{w'(\beta)<0}}f_{\beta}(v)$ with $\deg(f_\beta)\leq n-m-1$;
\item for $\beta=\alpha_{l i_0}\in\Phi$ with $l\neq j_0$, $C_\beta=v^{\delta_{w(\beta)<0}}f_{\beta}(v)$ with $\deg(f_\beta)\leq n-i_0-\delta_{w(\beta)<0}-\delta_{w(\beta)>w(-\alpha)}$;
\item for $\beta=\alpha_{l j_0}\in\Phi$, $C_\beta=v^{\delta_{w(\beta)<0}}f_{\beta}(v)$ with $\deg(f_\beta)\leq n-j_0-\delta_{w(\beta)>w(\alpha)}-\delta_{w(\beta)<0}$.
\end{itemize}

We will repeatedly use the finite height conditions in the following way:
\begin{enumerate}
\item\label{induction1} Assume that the right-bottom block of $C$ of size $k\times k$ is of the following form
$$
\small{\left(
  \begin{array}{ccc}
    C_{(n+1-k)(n+1-k)} & \cdots & 0 \\
    \vdots & \ddots & 0 \\
    C_{n(n+1-k)} & \cdots & C_{nn} \\
  \end{array}
\right)},
$$
and that $(v+p)^{n-l}\mid C_{lm}$ if $C_{lm}$ is sitting on the $(k\times k)$-block above;
\item\label{induction2} For $l$ and $m$ with $1\leq l,m\leq n-k$, assume that the $(k+1)\times(k+1)$-submatrix determined by $C_{lm}$ and the $k\times k$-block above is also lower-triangular;
\item\label{induction3} By applying the finite height condition $(v+p)^{\frac{k(k+1)}{2}}\mid ((k+1)\times (k+1)\mbox{-minors})$ to the $(k+1)\times(k+1)$-submatrix above, we conclude that $(v+p)^{n-k}\mid C_{lm}$.
\end{enumerate}

We first show that the finite height condition on $C$ implies that the right-bottom $(n-i_0+1)\times(n-i_0+1)$-block of $C$ is as described in the statements. This can be proved by induction together with the degree bounds as follows. It is clear that $C_{\beta}=0$ for $\beta=\alpha_{ln}$ if $l<n$, by the degree bounds. Now by inductively using items \eqref{induction1}, \eqref{induction2}, \eqref{induction3} above we conclude that the right-bottom $(n-i_0+1)\times(n-i_0+1)$-block of $C$ is as described in the statements.
Moreover, it is also immediate that $C_{lm}=0$ if $1\leq l< m$ and $j_0<m\leq n$ and that $(v+p)^{n-l}\mid C_{lm}$ if $i_0\leq l\leq n$ and $1\leq m\leq l$, by inductively using items \eqref{induction1}, \eqref{induction2}, \eqref{induction3} above.

We now check that $C_\beta=0$ if $\beta=\alpha_{lm}$ with $1\leq l<i_0$ and $l< m\leq j_0$. Note that we can not, in general, conclude $C_{lj_0}=0$ from the fact $(v+p)^{n-j_0}|C_{lj_0}$ induced from items \eqref{induction1}, \eqref{induction2}, \eqref{induction3} above for $1\leq l<i_0$, since the degree bound of the polynomial $C_{lj_0}$ is one higher than usual if $w(i_0)<w(l)<w(j_0)$. Let $\beta=\alpha_{lj_0}$ be such a root, and assume that the right-bottom $(n-l)\times(n-l)$-block of $C$ is as described in the statements. By items \eqref{induction1}, \eqref{induction2}, \eqref{induction3} and the degree bound, we may let $C_\beta=x(v+p)^{n-j_0}$ for $x\in R$. Consider $C'\defeq u_{\beta}(-x) C$. Then we have $C'_\beta=0$, and, moreover, the finite height condition together with the degree bounds implies that $C'_{lm}=0$ for all $i_0\leq m< j_0$. In particular, we have $$C'_{li_0}=C_{li_0}-xC_{-\alpha}=0.$$
By looking at the constant part of $C'_{li_0}$, which is the same as the constant part of $-xC_{-\alpha}$ as $v\mid C_{li_0}$, we have $p^{|\langle\eta,-\alpha^\vee\rangle|-1}xc_{-\alpha}=0$ in $R$ where we let $p^{|\langle\eta,-\alpha^\vee\rangle|-1}c_{-\alpha}$ be the constant part of $C_{-\alpha}$. On the other hand, by Lemma~\ref{lemma: equation from height conditions 1} $c_{-\alpha}$ is a unit in $R[\frac{1}{p}]$, so that we conclude that $x=0$. (Note that to prove Lemma~\ref{lemma: equation from height conditions 1} we only used the fact that the right-bottom $(n-i_0+1)\times(n-i_0+1)$-block of $C$ is as described in the statements, which is already proved in the previous paragraph.) Now by repeatedly using items \eqref{induction1}, \eqref{induction2}, \eqref{induction3} above together with the degree bounds, we conclude that $C_{lm}=0$ if $l< m\leq j_0$. Repeating this argument, we conclude that $C_\beta=0$ if $\beta=\alpha_{lm}$ with $1\leq l<i_0$ and $l< m\leq j_0$ and that $(v+p)^{n-l}\mid C_{lm}$ if $1\leq m\leq l\leq n$.

Finally, we point out that we have $(w^{-1}Aw)_{j_0j_0}=c(v+p)^{n-j_0}$, which can be readily seen during the induction steps due to the degree bound. This completes the proof.
\end{proof}

For the rest of this subsection, we observe some necessary properties of $\kappa_\beta\in\{0,1\}$, defined in Proposition~\ref{prop: FH 1}, as well as some identities of the coefficients of $V$ from the finite height conditions. We first fix some notation:
\begin{itemize}
\item For $\beta\in\Phi^-$, we write $\mathfrak{D}_\beta$ for the set of the decompositions $(\beta_1,\beta_2)$ of $\beta$ into two negative roots with $\beta_1$ sharing the column of $\beta$ (and $\beta_2$ sharing the row of $\beta$).
\item For a bad root $\beta\in\Phi^{-}$, we write $\mathfrak{A}_\beta$ for the set $\{(\beta_1,\beta_2)\in\mathfrak{D}_\beta\mid \kappa_\beta=\kappa_{\beta_1}+\kappa_{\beta_2}\}$.
\end{itemize}

The following are immediate consequences of Proposition~\ref{prop: FH 1}, which will be frequently used:
\begin{itemize}
\item if $\beta$ shares the row of $\alpha$ then we have $\kappa_\beta=\kappa_{s(\beta)}$;
\item if $\beta$ is a bad root sharing the column of $\alpha$ then $\deg(v^{-\kappa_{s(\beta)}}V_{s(\beta)})<|\langle\eta,s(\beta)^\vee\rangle|-1$ and $\kappa_{s(\beta)}=1+\kappa_\beta$;
\item if $(\beta_1,\beta_2)\in\mathfrak{D}_{-\alpha}$ then $\kappa_{\beta_1}+\kappa_{\beta_2}=1$;
\item $\kappa_{-\alpha}=0$.
\end{itemize}

\begin{lemma}\label{lemma: subadditive 1}
Let $\beta=\alpha_{lm}\in\Phi^-$. Then if $l\leq j_0$ or $m>i_0$ then $\kappa_\beta$ is subadditive, i.e., $\kappa_{\beta}\leq \kappa_{\beta_1}+\kappa_{\beta_2}$ for all $(\beta_1,\beta_2)\in\mathfrak{D}_{\beta}$.

Moreover, if $\beta$ has a decomposition $(\beta_1,\beta_2)\in\mathfrak{D}_\beta$ with $\kappa_\beta>\kappa_{\beta_1}+\kappa_{\beta_2}$ then
\begin{itemize}
\item $l>j_0$ and $m\leq i_0$;
\item $\kappa_{\beta}=1$ and $\kappa_{\beta_1}=\kappa_{\beta_2}=0$;
\item $\beta_1$ shares the row of $-\alpha$ such that $\beta_2$ is bad and either $s(\beta_1)$ is bad or $\beta_1=-\alpha$.
\end{itemize}
\end{lemma}

\begin{proof}
One can readily observe that
\begin{itemize}
\item if $\beta=\alpha_{lm}$ with $l\not\in\{i_0,j_0\}$ then $\kappa_\beta=\delta_{w(\beta)<0}$;
\item if $\beta=\alpha_{i_0m}$ then $\kappa_\beta=\delta_{w(\beta)<0}$;
\item if $\beta=\alpha_{j_0m}$ then $\kappa_\beta=\delta_{w'(\beta)<0}$.
\end{itemize}
The first part follows from combining these three observations together with the fact $\delta_{w(\beta)<0}\leq \delta_{w(\beta_1)<0}+\delta_{w(\beta_2)<0}$ if $(\beta_1,\beta_2)\in\mathfrak{D}_\beta$, case by case.

For the second part, if $\beta$ has a decomposition $(\beta_1,\beta_2)\in\mathfrak{D}_\beta$ with $\kappa_\beta>\kappa_{\beta_1}+\kappa_{\beta_2}$ then it is easy to see that $\beta_1$ shares the row of $-\alpha$, and so if $\beta_1\neq -\alpha$ then we have $w'(\beta_1)>0$ and $w(\beta_1)<0$ and so $s(\beta_1)$ is bad. $\beta_2$ is also bad, as $w'(\beta_2)<0$.
\end{proof}

\begin{lemma}\label{lemma: decomposition 1-1}
Let $\beta\in\Phi^-$ be a bad root sharing the row of $\alpha$. Then the map $(\beta_1,\beta_2)\mapsto (\beta_1, s(\beta_2))$ gives rise to a bijection between the following sets:
\begin{itemize}
\item the set $\mathfrak{A}_\beta$;
\item the set of elements $(\beta_1,\beta_2)$ of $\mathfrak{A}_{s(\beta)}$ with $\beta_2\neq-\alpha$.
\end{itemize}
Moreover, $\beta_2\in\Phi^-$ is bad if $(\beta_1,\beta_2)\in\mathfrak{A}_\beta$.
\end{lemma}

\begin{proof}
For $(\beta_1,\beta_2)\in\mathfrak{D}_\beta$ it is easy to see that $\kappa_\beta=\kappa_{\beta_1}+\kappa_{\beta_2}$, if and only if $\kappa_{\beta_1}=0=\kappa_{\beta_2}$, if and only if $w(\beta_1)>0$, $w(\beta_2)>0$, and $w'(\beta_2)<0$, if and only if $\kappa_{\beta_1}=0=\kappa_{s(\beta_2)}$, if and only if $\kappa_{s(\beta)}=\kappa_{\beta_1}+\kappa_{s(\beta_2)}$. Moreover, if $(\beta_1,\beta_2)\in\mathfrak{A}_\beta$ then $w(\beta_2)>0$ and $w'(\beta_2)<0$, and so $\beta_2$ is bad. This completes the proof.
\end{proof}

\begin{lemma}\label{lemma: decomposition 1-2}
Let $\beta\in\Phi^-$ be a root sharing the column of $\alpha$. If $\beta$ is bad then the map $(\beta_1,\beta_2)\mapsto (s(\beta_1),\beta_2)$ gives rise to a bijection between the following sets:
\begin{itemize}
\item the set $\mathfrak{A}_{\beta}$;
\item the set of elements $(\beta_1,\beta_2)$ of $\mathfrak{A}_{s(\beta)}$ such that $s(\beta_1)$ is bad.
\end{itemize}
Moreover, if $\beta$ is not bad and $(\beta_1,\beta_2)\in\mathfrak{A}_{s(\beta)}$ then $s(\beta_1)$ is not bad.
\end{lemma}

\begin{proof}
Assume $\beta$ is bad. Then we have $\kappa_\beta=0$ and $\kappa_{s(\beta)}=1$. Also, for $(\beta_1,\beta_2)\in\mathfrak{D}_{\beta}$ it is easy to see that $\kappa_{s(\beta)}=\kappa_{\beta_1}+\kappa_{\beta_2}$ and $s(\beta_1)$ is bad, if and only if $\kappa_{\beta_1}=1$, $\kappa_{s(\beta_1)}=0$, and $\kappa_{\beta_2}=0$, if and only if $\kappa_{\beta}=\kappa_{s(\beta_1)}+\kappa_{\beta_2}$. This completes the proof of the first part.

For the second part, it is clear if $\kappa_\beta=0$. If $\kappa_\beta=1$ then $\kappa_{s(\beta)}=1$ as $\beta$ is not bad, and so we have either $\kappa_{\beta_2}=1$ or $\kappa_{\beta_1}=1$. If $\kappa_{\beta_2}=1$ then $\kappa_{\beta_1}=0$ and so $s(\beta_1)$ is not bad. If $\kappa_{\beta_2}=0$ then $\kappa_{\beta_1}=1$ and $\kappa_{s(\beta_1)}=1$, and so $s(\beta_1)$ is not bad. This completes the proof.
\end{proof}

For each $\beta\in\Phi^-$, let $c_\beta$ be the coefficient of $v^{\kappa_{\beta}}$ in $V_{\beta}$. For each $\beta=\alpha_{lm}\in\Phi^-$ with either $l\leq j_0$ or $i_0\leq m$, let $c^{\imath}_\beta$ be the coefficient of $v^{\kappa_\beta-1}$ (resp. $v^{\kappa_\beta}$) in $V^{\imath}_{\beta}$ if $l> j_0$ and $m=i_0$ (resp. otherwise), where $V^{\imath}\defeq V^{-1}$.

\begin{lemma}\label{lemma: equation from height conditions 1}
If $\beta$ is a bad root sharing the row of $\alpha$ then $$c_\beta\cdot p^{|\langle\eta,\alpha^\vee\rangle|}=-c\cdot c_{s(\beta)}.$$ Moreover, we have $-c\cdot c_{-\alpha}=p^{|\langle\eta,\alpha^\vee\rangle|}$.
\end{lemma}

\begin{proof}
This is immediate from the elementary operations in \eqref{eq: elementary operation 1} together with the finite height conditions. For instance, we have $C_{i_0,i_0}=v(v+p)^{n-i_0-1}-cC_{-\alpha}=(v+p)^{n-i_0}$ where $C$ is defined in the proof of Proposition~\ref{prop: FH 1}. Now, extracting the constant term of the identity we get $-c\cdot c_{-\alpha}=p^{|\langle\eta,\alpha^\vee\rangle|}$.
\end{proof}

\begin{lemma}\label{lemma: equation from VV^-1=1 1}
Let $\beta\in\Phi^-$ be a bad root.
\begin{enumerate}
\item If $\beta\in\Phi^-$ is sharing the row of $\alpha$ then we have $$c_{s(\beta)}+\sum_{(\beta_1,\beta_2)\in\mathfrak{A}_\beta}c_{s(\beta_2)}c^{\imath}_{\beta_1} +c_{-\alpha}c^{\imath}_{\beta}+c^{\imath}_{s(\beta)}=0.$$
\item If $\beta\in\Phi^-$ is sharing the column of $\alpha$ then we have
    $$ c_{\beta}c^{\imath}_{-\alpha} + \sum_{(\beta_1,\beta_2)\in\mathfrak{A}_{\beta}}c_{\beta_2}c^{\imath}_{s(\beta_1)}+c^{\imath}_{s(\beta)}=0.$$
\end{enumerate}
\end{lemma}

\begin{proof}
(1) follows immediately from Lemma~\ref{lemma: decomposition 1-1} and Lemma~\ref{lemma: subadditive 1}, by extracting the constant term in the equation $(V\cdot V^{\imath})_{s(\beta)}=0$.

It is immediate that if $\beta'\in\Phi^-$ with $\beta'\neq-\alpha$ shares the column of $-\alpha$ and $s(\beta')$ is not bad, then the coefficient of $v^{\kappa_{\beta'}-1}$ in $V^{\imath}_{\beta'}$ vanishes, as $\kappa_{\beta'}$ is subadditive if $s(\beta')$ is not bad. This result together with Lemma~\ref{lemma: decomposition 1-2} and Lemma~\ref{lemma: subadditive 1} implies (2), by extracting the constant term in the equation $(V\cdot V^{\imath})_{s(\beta)}=0$.
\end{proof}

\subsection{Finite height conditions: the second form}\label{subsec: height condition, second}
Let $\tld{w}^*\in \Adm(\eta)$ be a colength one shape of the second form.
Write $\tld{w}=ws_\alpha t_{\eta-\alpha}w^{-1}$ where $w\in W$ satisfies $w(\alpha)<0$ and $w(j_0)<w(k)<w(i_0)$ for all $i_0<k<j_0$ if we set $\alpha=\alpha_{i_0j_0}\in \Phi^+\setminus\Delta$. For notational convenience we set $w'\defeq w s_\alpha$ and $s\defeq s_\alpha$, so that we may write $\tld{w}=w't_{\eta-\alpha}sw'^{-1}$.

With an eye towards applying the monodromy conditions, we introduce the following:
\begin{defn}
We say that a negative root $\beta\in \Phi^-$ is \emph{bad} (with respect to $\tld{w}$) if
\begin{itemize}
\item $\beta$ shares either the row or the column of $\alpha$;
\item $\beta$ satisfies $\delta_{w(\beta)<0}=\delta_{w'(\beta)<0}$.
\end{itemize}
\end{defn}

\begin{prop}\label{prop: FH 2}
Let $A\in U(\tld{w},\leq\!\!\eta)(R)$.

Then $A_{w'(i_0)w'(i_0)}=(w'^{-1}Aw')_{i_0i_0}=c(v+p)^{n-j_0}$ for some $c\in R$ and
\begin{equation}\label{eq: elementary operation 2}
u_{\alpha}(-c)\cdot v^{\eps_{i_0}}\cdot (w'^{-1}Aw')\cdot s=(v+p)^\eta\cdot V'
\end{equation}
where $V'$ is a lower triangular matrix whose entries are in $R[v]$.

Moreover, $V\defeq v^{-\eps_{j_0}}\cdot V'$ is unipotent, and for each $\beta\in\Phi^{-}$ the entry $V'_\beta$ can be written as $v^{\kappa'_\beta}\cdot f_{\beta}(v)$ for a polynomial $f_{\beta}(v)\in R[v]$ of the form $\sum_{k}c_{\beta,k}(v+p)^k$ where
\begin{enumerate}
\item if $\beta$ is bad then $\kappa'_\beta=\delta_{w'(\beta)<0}$ and  $\deg(f_{\beta})<|\langle\eta,\beta^\vee\rangle|+1$;
\item if $\beta$ shares the row or the column of $\alpha$ but is not bad then $\kappa'_\beta=1+\delta_{w'(\beta)<0}=\delta_{w(\beta)<0}=1$ and $\deg(f_{\beta})<|\langle\eta,\beta^\vee\rangle|$;
\item if $\beta$ shares the column of $-\alpha$ then  $\kappa'_\beta=\delta_{w(\beta)<0}$ and $\deg(f_{\beta})<|\langle\eta,\beta^\vee\rangle|-\delta_{w(\beta)>w'(\alpha)}-\delta_{w(\beta)< 0}$;
\item for all other roots $\beta\in\Phi^-$, $\kappa'_\beta=\delta_{w'(\beta)<0}$ and $\deg(f_{\beta})<|\langle\eta,\beta^\vee\rangle|$.
\end{enumerate}
\end{prop}

\begin{proof}
It is easy to see that $(w'^{-1}Aw')_{i_0i_0}$ is a polynomial of degree less than $n-j_0+1$. Let $c\in R$ be the coefficient of $v^{n-j_0}$ in $(w'^{-1}Aw')_{i_0i_0}$, and set $C\defeq u_{\alpha}(-c)v^{\eps_{i_0}}(w'^{-1}Aw')s$ as in \eqref{eq: elementary operation 2}. One can readily observe the following degree bound:
\begin{itemize}
\item $\deg (C_{ll})\leq n-l$ if $l\neq j_0$, and $\deg (C_{j_0j_0})\leq n-l+1$ and $C_{j_0j_0}\in vR[v]$;
\item for $\beta=\alpha_{lm}\in\Phi$ with $l\neq i_0$ and $m\not\in\{i_0,j_0\}$, $C_\beta=v^{\delta_{w'(\beta)<0}}f_\beta(v)$ with $\deg (f_\beta)\leq n-m-1$;
\item for $\beta=\alpha_{li_0}\in\Phi$, $C_\beta=v^{\delta_{w(\beta)<0}}f_{\beta}(v)$ with $\deg(f_\beta)\leq n-i_0-1-\delta_{w(\beta)>w'(\alpha)}-\delta_{w(\beta)< 0}$;
\item for $\beta=\alpha_{lj_0}\in\Phi$ with $\beta\neq\alpha$, $C_\beta=v^{\delta_{w(\beta)<0}}f_{\beta}(v)$ with $\deg(f_\beta)\leq n-j_0+1-\delta_{w(\beta)>w'(-\alpha)}-\delta_{w(\beta)< 0}$;
\item for $\beta=\alpha_{i_0m}\in\Phi$, $C_\beta=v^{\delta_{w(\beta)<0}}f_\beta(v)$ with $\deg(f_\beta)\leq n-m$ if $\beta\neq\alpha$, and $C_\alpha=vf_\beta(v)$ with $\deg(f_\beta)\leq n-j_0$.
\end{itemize}
Now, the rest of the proof is similar to Proposition~\ref{prop: FH 1}, and we leave the details for the reader.
\end{proof}

For the rest of this subsection, we observe some necessary properties of $\kappa'_\beta\in\{0,1\}$, defined in Proposition~\ref{prop: FH 2}, as well as some identities of the coefficients of $V'$ (and so of $V$ as well) from the finite height conditions. We first fix some notation:
\begin{itemize}
\item For $\beta\in\Phi^-$, we write $\mathfrak{D}_\beta$ for the set of the decompositions $(\beta_1,\beta_2)$ of $\beta$ into two negative roots with $\beta_1$ sharing the column of $\beta$ (and $\beta_2$ sharing the row of $\beta$).
\item For a bad root $\beta\in\Phi^{-}$, we write $\mathfrak{A}_\beta$ for the set $\{(\beta_1,\beta_2)\in\mathfrak{D}_\beta\mid \kappa'_\beta=\kappa'_{\beta_1}+\kappa'_{\beta_2}\}$.
\item To describe the bottom degrees of each entry of $V$, we set $\kappa_\beta\defeq \kappa'_\beta-1$ if $\beta\in\Phi^-$ shares the row of $-\alpha$, and $\kappa_\beta\defeq\kappa'_\beta$ otherwise.
\end{itemize}

The following are immediate consequences of Proposition~\ref{prop: FH 2}, which will be frequently used:
\begin{itemize}
\item if $\beta$ is bad then we have $\kappa'_\beta=\kappa'_{s(\beta)}$;
\item if $\beta$ is a bad root sharing the column of $\alpha$ then $\deg(v^{-\kappa'_{s(\beta)}}V'_{s(\beta)})<|\langle\eta,s(\beta)^\vee\rangle|-1$;
\item if $\beta=\alpha_{j_0m}$ with $i_0<m\leq j_0$ then $\kappa'_\beta=1$;
\item if $\beta=\alpha_{li_0}$ with $i_0<l\leq j_0$ then $\kappa'_{\beta}=0$, and, in particular, $\kappa_{-\alpha}=-1$.
\end{itemize}

\begin{lemma}\label{lemma: subadditive 2}
Let $\beta=\alpha_{lm}\in\Phi^-$. Then $\kappa'_\beta$ is subadditive, i.e. $\kappa'_{\beta}\leq \kappa'_{\beta_1}+\kappa'_{\beta_2}$ for all $(\beta_1,\beta_2)\in\mathfrak{D}_\beta$.

Moreover, if $\beta$ has a decomposition $(\beta_1,\beta_2)\in\mathfrak{D}_\beta$ with $\kappa_\beta>\kappa_{\beta_1}+\kappa_{\beta_2}$ then
\begin{itemize}
\item $l>j_0$ and $m\leq i_0$;
\item $\beta_1$ shares the row of $-\alpha$ such that $\beta_2$ is bad and either $s(\beta_1)$ is bad or $\beta_1=-\alpha$.
\end{itemize}
\end{lemma}

\begin{proof}
The proof is very similar to Lemma~\ref{lemma: subadditive 1}. Observe that
\begin{itemize}
\item if $\beta=\alpha_{lm}$ with $l\neq i_0$ and $m\not\in\{i_0,j_0\}$ then $\kappa'_\beta=\delta_{w'(\beta)<0}$;
\item if $\beta=\alpha_{lm}$ with $m\in\{i_0,j_0\}$ then $\kappa'_\beta=\delta_{w(\beta)<0}$;
\item if $\beta=\alpha_{i_0m}$ then $\kappa'_\beta=\delta_{w(\beta)<0}$.
\end{itemize}
The subadditivity of $\kappa'_\beta$ follows from combining the three observations above together with the fact $\delta_{w'(\beta)<0}\leq \delta_{w'(\beta_1)<0}+\delta_{w'(\beta_2)<0}$ if $(\beta_1,\beta_2)\in\mathfrak{D}_\beta$. The subadditivity of $\kappa_\beta$ when $l\leq j_0$ or $i_0<m$ also easily follows from the subadditivity of $\kappa'_\beta$ together with the colength one property of $w$.

Assume that $\beta$ has a decomposition $(\beta_1,\beta_2)\in\mathfrak{D}_\beta$ with $\kappa_\beta>\kappa_{\beta_1}+\kappa_{\beta_2}$. Then as $\kappa'_\beta$ is subadditive and $\kappa_{\beta}$ is subadditive if $m>i_0$, we see that $\beta_1$ shares the row of $-\alpha$ and so $\beta_2$ shares the column of $\alpha$. It is also easy to see that $\beta_2$ is bad, and $s(\beta_1)$ is bad if $\beta_1\neq-\alpha$.
\end{proof}

\begin{lemma}\label{lemma: decomposition 2-1}
Let $\beta\in\Phi^-$ be a bad root sharing the row of $\alpha$. Then the map $(\beta_1,\beta_2)\mapsto (\beta_1, s(\beta_2))$ gives rise to a bijection between the following sets:
\begin{itemize}
\item the set $\mathfrak{A}_\beta$;
\item the set of elements $(\beta_1,\beta_2)$ of $\mathfrak{A}_{s(\beta)}$ with $\beta_2\neq-\alpha$.
\end{itemize}
Moreover, $\beta_2\in\Phi^-$ is bad if $(\beta_1,\beta_2)\in\mathfrak{A}_\beta$.
\end{lemma}

\begin{proof}
The argument is very similar to Lemma~\ref{lemma: decomposition 1-1}. We leave the details for the reader.
\end{proof}

\begin{lemma}\label{lemma: decomposition 2-2}
Let $\beta\in\Phi^-$ be a root sharing the column of $\alpha$. If $\beta$ is bad then the map $(\beta_1,\beta_2)\mapsto (s(\beta_1),\beta_2)$ gives rise to a bijection between the following sets:
\begin{itemize}
\item the set $\mathfrak{A}_\beta$;
\item the set of elements $(\beta_1,\beta_2)$ of $\mathfrak{A}_{s(\beta)}$ such that $s(\beta_1)$ is bad.
\end{itemize}
Moreover, if $\beta$ is not bad and $(\beta_1,\beta_2)\in\mathfrak{A}_{s(\beta)}$ then $s(\beta_1)$ is not bad.
\end{lemma}

\begin{proof}
The argument is very similar to Lemma~\ref{lemma: decomposition 1-2}. We leave the details for the reader.
\end{proof}

For each $\beta\in\Phi^-$, let $c_\beta$ be the coefficient of $v^{\kappa'_{\beta}}$ in $V'_{\beta}$. For each $\beta=\alpha_{lm}\in\Phi^-$ with either $l\leq j_0$ or $i_0\leq m$, let $c^{\imath}_\beta$ be the coefficient of $v^{\kappa'_\beta-1}$ (resp. $v^{\kappa'_\beta}$) in $V^{\imath}_{\beta}$ if either $l=j_0$ or $l> j_0$ and $m=i_0$ (resp. otherwise), where $V^{\imath}\defeq V^{-1}$.

\begin{lemma}\label{lemma: equation from height conditions 2}
If $\beta$ is a bad root sharing the row of $\alpha$ then $$c_\beta\cdot p^{|\langle\eta,\alpha^\vee\rangle|}=-c\cdot c_{s(\beta)}.$$ Moreover, we have $-c\cdot c_{-\alpha}=p^{|\langle\eta,\alpha^\vee\rangle|}$.
\end{lemma}

\begin{proof}
The same argument as in Lemma~\ref{lemma: equation from height conditions 1} also works, using the elementary operations in \eqref{eq: elementary operation 2} together with the finite height conditions.
\end{proof}

\begin{lemma}\label{lemma: equation from VV^-1=1 2}
Let $\beta\in\Phi^-$ be a bad root
\begin{enumerate}
\item If $\beta\in\Phi^-$ is sharing the row of $\alpha$ then we have $$c_{s(\beta)}+\sum_{(\beta_1,\beta_2)\in\mathfrak{A}_\beta}c_{s(\beta_2)}c^{\imath}_{\beta_1} +c_{-\alpha}c^{\imath}_{\beta}+c^{\imath}_{s(\beta)}=0.$$
\item If $\beta\in\Phi^-$ is sharing the column of $\alpha$ then we have
    $$c_{\beta}c^{\imath}_{-\alpha}+ \sum_{(\beta_1,\beta_2) \in\mathfrak{A}_\beta} c_{\beta_2}c^{\imath}_{s(\beta_1)}+ c^{\imath}_{s(\beta)}=0.$$
\end{enumerate}
\end{lemma}

\begin{proof}
The proof is similar to that of Lemma~\ref{lemma: equation from VV^-1=1 1}. The only difference is that we use Lemma~\ref{lemma: decomposition 2-1}, Lemma~\ref{lemma: subadditive 2}, and Lemma~\ref{lemma: decomposition 2-2} instead of Lemma~\ref{lemma: decomposition 1-1}, Lemma~\ref{lemma: subadditive 1}, and Lemma~\ref{lemma: decomposition 1-2}, respectively.
\end{proof}

\subsection{Description of $U(\tld{w},\leq\!\! \eta)$ in colength one}
Let $V$ be the matrix introduced at Proposition~\ref{prop: FH 1} or at Proposition~\ref{prop: FH 2}. For each $\beta\in\Phi^-$, we let $m_\beta\defeq |\langle\eta,\beta^\vee\rangle|-1$, and recall that the entry $V_{\beta}$ is of the form
$$v^{\kappa_\beta}\sum_{k=0}^{m'_\beta}c_{\beta,k}(v+p)^k$$
where
\begin{equation}\label{equ: description of the entries of V}
m'_\beta=
\left\{
  \begin{array}{ll}
    m_\beta+1 & \hbox{if $\beta$ is bad;} \\
    m_\beta-1 & \hbox{if $\beta$ shares the column of $-\alpha$ such that $s(\beta)\in\Phi^-$ is bad;} \\
    m_\beta & \hbox{otherwise.}
  \end{array}
\right.
\end{equation}

\begin{prop}\label{prop: space of height conditions}
Let $\tld{w}^*\in \Adm(\eta)$ be a colength one shape. Then there is a closed immersion $$U(\tld{w},\leq\!\! \eta)\hookrightarrow \mathrm{Spec}\cO[\{c_{\beta,k}\mid \beta\in\Phi^-\mbox{ and } 0\leq k\leq m_\beta'\}\cup\{c\}].$$
\end{prop}

\begin{proof}
This is immediate from Propositions~\ref{prop: FH 1} and~\ref{prop: FH 2}.
\end{proof}

\section{Monodromy conditions}\label{sec: monodromy}
In this section, we induce certain identities and properties, that will be necessary to describe $U(\tld{w},\leq\!\! \eta,\nabla_{\bf{a}})$ for $\bf{a}\in\cO^n$ when $\tld{w}^*\in \Adm(\eta)$ is of colength one. Throughout this section, by $R$ we mean a Noetherian $\cO$-flat $\cO$-algebra. We keep the notation of \S \ref{sec:finite height conditions}.

We fix some notation. Set $\gamma_\beta\defeq \kappa_\beta+m'_\beta$, where $m'_\beta$ is defined in \eqref{equ: description of the entries of V}. For each $\beta\in\Phi^-$, we write $\mathbf{F_{\geq\beta}}$ (resp. $\mathbf{F}_{>\beta}$) for the free $\Z$-module generated by the monomials $c_{\beta_1,k_1}\cdots c_{\beta_s,k_s}$ for all non-negative integers $k_1,\cdots,k_s$ with $\sum_{i=1}^sk_i\leq\gamma_\beta+1$ and for all negative roots $\beta_1,\cdots,\beta_s$ with $\beta=\beta_1+\cdots+\beta_s$ (resp. with $\beta=\beta_1+\cdots+\beta_s$ and $s>1$). For $\beta'\in\Phi^-$ with $\beta'>\beta$ we write $\mathbf{F}_{>\beta}^{\beta'}$ for the submodule of $\mathbf{F}_{>\beta}$ generated by the monomials $c_{\beta_1,k_1}\cdots c_{\beta_s,k_s}$ with $\beta_i\neq\beta'$ for all $i$. We will consider all of these free $\Z$-modules as submodules of the ring $\cO(U(\tld{w},\leq\!\! \eta))$ via the closed immersion of Proposition~\ref{prop: space of height conditions}.
Finally, if $A\in U(\tld{w},\leq\!\! \eta)(R)$, we abuse notation and write $\mathbf{F}_{\geq\beta}$, $\mathbf{F}_{>\beta}$, $\mathbf{F}^{\beta'}_{>\beta}$ for the image in $R[v]$ (via the map $\cO(U(\tld{w},\leq\!\! \eta))\rightarrow R$ corresponding to $A$) of the free $\Z$-modules above.

Let $A\in U(\tld{w},\leq\!\! \eta)(R)$, and let $V$ be the corresponding matrix obtained by Proposition~\ref{prop: FH 1} (resp.~by Proposition~\ref{prop: FH 2}) if $\tld{w}$ is of the first form (resp.~if $\tld{w}$ is of the second form). Then an elementary computation shows that $A\in U^{\textnormal{nv}}(\tld{w},\leq\!\! \eta,\nabla_{\bf{a}})(R)$ if and only if for each $\beta\in\Phi^-$
\begin{equation}\label{equ: monodromy condition}
V^\sharp_\beta\defeq(\nabla_{\bf{a}_w} V)_\beta+\sum_{(\beta_1,\beta_2)\in\mathfrak{D}_\beta}(\nabla_{\bf{a}_w} V)_{\beta_2}V^{\imath}_{\beta_1}\in (v+p)^{m_\beta}\mathrm{M}_n(R[v])
\end{equation}
where $\bf{a}_w\in\cO^n$ satisfies $w(\bf{a}_w)=\bf{a}$.
Throughout this subsection, we assume \eqref{equ: monodromy condition} holds for all $\beta\in\Phi^-$, and write $\nabla$ for $\nabla_{\bf{a}_w}$ to lighten the notation.

\subsection{Monodromy conditions: the first form}\label{subsec: monodromy cond, first}
Let $\tld{w}^*\in \Adm(\eta)$ be a colength one shape of the first form (described in Proposition~\ref{prop:classification, colength 1 dual shapes}), and keep the notation of \S\ref{subsec: height condition, first}. In particular, we keep the notation of Proposition~\ref{prop: FH 1}.

\begin{lemma}\label{lemma: routine monodromy 1}
Let $\beta=\alpha_{lm}\in\Phi^-$ with $l\leq j_0$ or $i_0\leq m$. Then we have
\begin{enumerate}
\item if $\beta$ is not bad then $$V^\sharp_\beta\in v^{\tld{\kappa}_\beta}(v+p)^{m_\beta}\big(X_\beta+\mathbf{F}_{>\beta}\big)$$
    where $X_\beta=(m_\beta+\kappa_\beta-\langle\bf{a}_w,\beta^\vee\rangle)c_{\beta,m'_\beta}$ and $$\tld{\kappa}_\beta\defeq\left\{
                                 \begin{array}{ll}
                                  \kappa_\beta  & \hbox{if either $l\leq j_0$, $i_0<m$, or $i_0=m$ and $s(\beta)\in\Phi^-$ is not bad;} \\
                                  \kappa_\beta-1  & \hbox{if $i_0=m$ and $s(\beta)\in\Phi^-$ is bad,}
                                 \end{array}
                               \right.
     $$
\item if $\beta$ is bad then $$V^\sharp_\beta\in v^{\kappa_\beta}(v+p)^{m_\beta}\big(Y_\beta+vX_{\beta}+\mathbf{F}_{>\beta}\big)$$ where
    $$\left\{
      \begin{array}{rl}
        X_\beta&=(m'_\beta+\kappa_\beta-\langle\bf{a}_w,\beta^\vee\rangle) c_{\beta,m'_\beta};  \\
        Y_\beta&=(m_\beta+\kappa_\beta-\langle\bf{a}_w,\beta^\vee\rangle+m_\beta+\kappa_\beta)c_{\beta,m_\beta} +p(\kappa_\beta-\langle\bf{a}_w,\beta^\vee\rangle) c_{\beta,m'_\beta}.
      \end{array}
    \right.$$
\end{enumerate}
\end{lemma}

\begin{proof}
Assume that $\beta=\alpha_{lm}\in\Phi^-$ is not bad. If either $l\leq j_0$, $i_0<m$, or $i_0=m$ and $s(\beta)$ is not bad then by Proposition~\ref{prop: FH 1} together with Lemma~\ref{lemma: subadditive 1} it is routine to check the equation in (1). If $l> j_0$, $i_0= m$, and $s(\beta)\in\Phi^-$ is bad then $\kappa_\beta$ is not subadditive and $\deg (V_\beta/v^{\kappa_\beta})<|\langle\eta,\beta^\vee\rangle|-1$ by part (3) of Proposition~\ref{prop: FH 1}, and so $v^{\kappa_\beta}\nmid \sum_{(\beta_1,\beta_2)\in\mathfrak{D}_\beta}(\nabla V)_{\beta_2}V^{\imath}_{\beta_1}$ while $v^{\kappa_\beta}\mid (\nabla V)_\beta$. Hence, in this case we get the equation in (1) with $\tld{\kappa}_\beta=\kappa_\beta-1$ in the equation.

Assume that $\beta$ is bad. By part (1) of Proposition~\ref{prop: FH 1}, $\deg V_\beta<|\langle\eta,\beta^\vee\rangle|+1$ (as $\kappa_\beta=0$) and $\kappa_\beta$ is subadditive, and so we get the equation in (2).
\end{proof}

For $\beta\in\Phi^-$ and for each integer $s\geq1$, we set $I_{\beta,s}$ to be the set of the tuples of negative roots $(\beta_1,\beta_2,\cdots,\beta_s)$ such that
\begin{itemize}
\item $\beta=\beta_1+\beta_2+\cdots+\beta_s$ and $\beta_1$ shares the column of $\beta$;
\item $\beta_i+\beta_{i+1}\in\Phi^-$ for $i\in\{1,\cdots,s-1\}$;
\item $\sum_{i=1}^s\delta_{w(\beta_i)<0}=s-\delta_{w(\beta)>0}$.
\end{itemize}
Moreover, we set $$I_\beta\defeq\bigcup_{s\geq 1}I_{\beta,s}.$$

\begin{lemma}\label{lemma: shadow 1}
Let $\beta\in\Phi^-$ with $\beta\geq -\alpha$.
\begin{enumerate}
\item If $(\beta_1,\beta_2)\in\mathfrak{D}_{\beta}$ then $\gamma_\beta\geq \gamma_{\beta_1}+\gamma_{\beta_2}$;
\item $(\beta_1,\beta_2)\in\mathfrak{D}_{\beta}$ satisfies $\gamma_\beta=\gamma_{\beta_1}+\gamma_{\beta_2}$ if and only if $\kappa_\beta<\kappa_{\beta_1}+\kappa_{\beta_2}$.
\end{enumerate}
In particular, if $(\beta_1,\beta_2)\in\mathfrak{D}_{-\alpha}$ then
$\gamma_{\beta_1}+\gamma_{\beta_2}=\gamma_{-\alpha}$,
$\deg V^{\imath}_{\beta_1}= \gamma_{\beta_1}=m_{\beta_1}+\kappa_{\beta_1}$, and $$\frac{1}{(\gamma_{\beta_1})!}\frac{d^{\gamma_{\beta_1}}V^{\imath}_{\beta_1}}{dv^{\gamma_{\beta_1}}}= \sum_{(\beta'_1,\cdots,\beta'_s)\in I_{\beta_1}} (-1)^{\ell+j_0-s}c_{\beta'_1,m_{\beta'_1}}c_{\beta'_2,m_{\beta'_2}}\cdots c_{\beta'_s,m_{\beta'_s}}$$
where $\beta_1=\alpha_{\ell i_0}$.
\end{lemma}

\begin{proof}
Recall that $\gamma_\beta$ indicates the degree of $V_{\beta}$. By definition of $w$, $\kappa_{\beta}=\kappa_{s(\beta)}$ if either $\beta=\alpha_{li_0}$ with $i_0<l<j_0$ or $\beta=\alpha_{i_0m}$ with $i_0<m<j_0$. Now it is easy to check (1) and~(2) case by case.

For the second part, let $(\beta_1,\beta_2)\in\mathfrak{D}_{-\alpha}$. Then it is clear that $\gamma_{\beta_1}+\gamma_{\beta_2}=\gamma_{-\alpha}$ by (2), and it is also clear that $\deg V^{\imath}_{\beta_1}= \gamma_{\beta_1}$ by (1). Finally, the elements of $I_{\beta}$ correspond to the monomials in $V^{\imath}_{\beta_1}$ which has the highest degree $\gamma_{\beta_1}$. This completes the proof.
\end{proof}

\begin{lemma}\label{lemma: equation for def ring 1}
We have
$$p^{m_{-\alpha}+1}\cdot \langle\bf{a}_w,-\alpha^\vee\rangle =p^{m_{-\alpha}}\cdot c\cdot Z_{-\alpha}$$
where
\begin{multline*}
Z_{-\alpha}\defeq\left(m_{-\alpha}-\langle\bf{a}_w,-\alpha^\vee\rangle\right) c_{-\alpha,m_{-\alpha}}\\
+\sum_{(\beta_1,\beta_2)\in\mathfrak{D}_{-\alpha}} \left(m_{\beta_2}+\kappa_{\beta_2}-\langle\bf{a}_w,\beta_2^\vee\rangle\right) c_{\beta_2,m_{\beta_2}} \sum_{(\beta'_1,\cdots,\beta'_s)\in I_{\beta_1}} (-1)^{\ell+j_0-s}c_{\beta'_1,m_{\beta'_1}}\cdots c_{\beta'_s,m_{\beta'_s}}
\end{multline*}
if we write $\beta_1=\alpha_{\ell i_0}$.
\end{lemma}

\begin{proof}
By Lemma~\ref{lemma: shadow 1}, we see that $\deg((\nabla V)_{\beta_2}V^{\imath}_{\beta_1})=\gamma_{\beta_2}+\gamma_{\beta_1}=\gamma_{-\alpha}=m_{-\alpha}$ for all $(\beta_1,\beta_2)\in\mathfrak{D}_{-\alpha}$. Hence, if we apply the monodromy condition~\eqref{equ: monodromy condition} to $V_{-\alpha}$ then we have
\begin{multline*}
V^\sharp_{-\alpha}= (m_{-\alpha}-\langle\bf{a}_w,-\alpha^\vee\rangle)c_{-\alpha,m_{-\alpha}}(v+p)^{m_{-\alpha}}\\
+\sum_{(\beta_1,\beta_2)\in\mathfrak{D}_{-\alpha}} (m_{\beta_2}+\kappa_{\beta_2}-\langle\bf{a}_w,\beta_2^\vee\rangle)c_{\beta_2,m_{\beta_2}}\sum_{(\beta'_1,\cdots,\beta'_s)\in I_{\beta_1}} (-1)^{\ell+j_0-s}c_{\beta'_1,m_{\beta'_1}}\cdots c_{\beta'_s,m_{\beta'_s}}(v+p)^{m_{-\alpha}}
\end{multline*}
by applying the second part of Lemma~\ref{lemma: shadow 1}. Since the constant term of $V^\sharp_{-\alpha}$ appears only at $(\nabla V)_{-\alpha}$ by Lemma~\ref{lemma: subadditive 1}, by extracting the constant term we have $$-\langle\bf{a}_w,-\alpha^\vee\rangle  c_{-\alpha} =p^{m_{-\alpha}}Z_{-\alpha}.$$ By the second part of Lemma~\ref{lemma: equation from height conditions 1}, we get the desired identity.
\end{proof}

We further eliminate the variables $Y_\beta$ for bad roots $\beta\in\Phi^-$.
\begin{lemma}\label{lemma: bad cases 1}
Let $\beta\in\Phi^-$ be a bad root.
\begin{enumerate}
\item If $\beta$ shares the row of $\alpha$ then we have
$$p^{m_{s(\beta)}}\cdot(Y_\beta+c\cdot X_{s(\beta)})\in p^{m_{s(\beta)}}\cdot \left(c\cdot c_{\beta,m'_{\beta}}\cdot\mathbf{F}_{\geq-\alpha}+ \mathbf{F}_{>\beta}+c\cdot \mathbf{F}_{>s(\beta)}^{\beta}\right).$$
\item If $\beta$ shares the column of $\alpha$ then we have
$$p^{m_{s(\beta)}}\cdot(Y_\beta-c\cdot X_{s(\beta)})\in p^{m_{s(\beta)}}\cdot \left(c\cdot c_{\beta,m'_{\beta}}\cdot\mathbf{F}_{\geq-\alpha}+ \mathbf{F}_{>\beta}+c\cdot \mathbf{F}_{>s(\beta)}^{\beta}\right).$$
\end{enumerate}
\end{lemma}

\begin{proof}
We first treat the case (1). By Lemma~\ref{lemma: decomposition 1-1}, extracting the constant term in the monodromy equation in~part (1) of Lemma~\ref{lemma: routine monodromy 1} for~$s(\beta)$ gives rise to
\begin{equation}\label{eq: row bad 1-1}
\langle\bf{a}_w,s(\beta)^\vee\rangle c_{s(\beta)} +\sum_{(\beta_1,\beta_2)\in\mathfrak{A}_\beta} \langle\bf{a}_w,s(\beta_2)^\vee\rangle  c_{s(\beta_2)}c^\imath_{\beta_1}+\langle\bf{a}_w,-\alpha^\vee\rangle c_{-\alpha}c^{\imath}_\beta
\in -p^{m_{s(\beta)}}(X_{s(\beta)}+\mathbf{F}_{>s(\beta)}).
\end{equation}
As $\deg(V_{\beta}V^{\imath}_{-\alpha})=m'_\beta+m_{-\alpha} =m_\beta+1+m_{-\alpha}=m_{s(\beta)}=\tld{\kappa}_{s(\beta)}+m_{s(\beta)}$, the quantity in \eqref{eq: row bad 1-1}, in fact, belongs to
$$-p^{m_{s(\beta)}}(X_{s(\beta)}+c_{\beta,m'_\beta}\cdot \mathbf{F}_{\geq -\alpha}+\mathbf{F}_{>s(\beta)}^\beta).$$
By multiplying $-c$ and then applying Lemma~\ref{lemma: equation from height conditions 1} together with Lemma~\ref{lemma: decomposition 1-1}, we have
\begin{multline*}
p^{|\langle\eta,\alpha^\vee\rangle|} \langle\bf{a}_w,s(\beta)^\vee\rangle c_{\beta} +\sum_{(\beta_1,\beta_2)\in\mathfrak{A}_\beta}p^{|\langle\eta,\alpha^\vee\rangle|} \langle\bf{a}_w,s(\beta_2)^\vee\rangle c_{\beta_2}c^{\imath}_{\beta_1}+p^{|\langle\eta,\alpha^\vee\rangle|}\langle\bf{a}_w,-\alpha^\vee\rangle c^{\imath}_\beta\\
\in c\cdot p^{m_{s(\beta)}}(X_{s(\beta)}+c_{\beta,m'_\beta}\cdot \mathbf{F}_{\geq -\alpha}+\mathbf{F}_{>s(\beta)}^\beta).
\end{multline*}
Applying the identity $c_\beta+\sum_{(\beta_1,\beta_2)\in\mathfrak{A}_\beta} c_{\beta_2}c^{\imath}_{\beta_1}+c^{\imath}_\beta=0$, induced from $(V\cdot V^{\imath})_\beta=0$, together with $s(\beta)=\beta-\alpha$ and $s(\beta_2)=\beta_2-\alpha$, we have
\begin{equation}\label{eq: row bad 1-2}
p^{|\langle\eta,\alpha^\vee\rangle|} \langle\bf{a}_w,\beta^\vee\rangle c_{\beta} +\sum_{(\beta_1,\beta_2)\in\mathfrak{A}_\beta}p^{|\langle\eta,\alpha^\vee\rangle|} \langle\bf{a}_w,\beta_2^\vee\rangle c_{\beta_2}c^{\imath}_{\beta_1}\\
\in c\cdot p^{m_{s(\beta)}}(X_{s(\beta)}+c_{\beta,m'_\beta}\cdot \mathbf{F}_{\geq -\alpha}+\mathbf{F}_{>s(\beta)}^\beta).
\end{equation}
But by extracting the constant term in the monodromy equation in part (2) of Lemma~\ref{lemma: routine monodromy 1} for~$\beta$, the quantity in \eqref{eq: row bad 1-2} is also belongs to
$$-p^{|\langle\eta,-\alpha^\vee\rangle|+|\langle\eta,\beta^\vee\rangle|-1}(Y_{\beta} +\mathbf{F}_{>\beta})$$
and so we have the desired result.

We now treat the case (2). By extracting the coefficient of $v^{\kappa_\beta}$ in the equation $(V\cdot V^{\imath})_\beta=0$, we have $c_{\beta}+\sum_{(\beta_1,\beta_2)\in\mathfrak{A}_\beta}c_{\beta_2}c^{\imath}_{\beta_1}+c^{\imath}_\beta=0$, which together with the equation in part (2) of Lemma~\ref{lemma: equation from VV^-1=1 2} induces $$\sum_{(\beta_1,\beta_2)\in\mathfrak{A}_\beta}c_{\beta_2}(c^{\imath}_{s(\beta_1)}-c^{\imath}_{-\alpha}c^{\imath}_{\beta_1})+ (c^{\imath}_{s(\beta)}-c^{\imath}_{-\alpha}c^{\imath}_\beta)=0.$$
This equation together with Lemma~\ref{lemma: decomposition 2-2} inductively induces that
\begin{equation}\label{eq: eqation for column bad roots 1}
c^{\imath}_{s(\beta)}-c^{\imath}_{-\alpha}c^{\imath}_\beta=0
\end{equation}
for any bad root $\beta$ sharing the column of $\alpha$.

By extracting the coefficient of $v^{\kappa_\beta}$ in the monodromy equation in part (2) of Lemma~\ref{lemma: routine monodromy 1} for $\beta$, we have $$(\kappa_{\beta}-\langle\bf{a}_w,\beta^\vee\rangle)c_{\beta} +\sum_{(\beta_1,\beta_2)\in\mathfrak{A}_\beta}(\kappa_{\beta_2}-\langle\bf{a}_w,\beta_2^\vee\rangle) c_{\beta_2}c^{\imath}_{\beta_1} \in p^{m_{\beta}}(Y_{\beta}+\mathbf{F}_{>\beta}).$$
Similarly, by extracting the coefficient of $v^{\kappa_{s(\beta)}-1}$ in the monodromy equation in part (1) of Lemma~\ref{lemma: routine monodromy 1} for $s(\beta)$, we have
\begin{equation}\label{eq: column bad 1-1}
(\kappa_{\beta}-\langle\bf{a}_w,\beta^\vee\rangle)c_{\beta}c^{\imath}_{-\alpha}+ \sum_{(\beta_1,\beta_2)\in\mathfrak{A}_\beta}(\kappa_{\beta_2}-\langle\bf{a}_w,\beta_2^\vee\rangle)c_{\beta_2}c^{\imath}_{s(\beta_1)} \in p^{m_{s(\beta)}}(X_{s(\beta)}+\mathbf{F}_{>{s(\beta)}}).
\end{equation}
As $\deg(V_{\beta}V^{\imath}_{-\alpha})=\gamma_\beta+\gamma_{-\alpha} =m'_\beta+m_{-\alpha}=m_\beta+1+m_{-\alpha}=m_{s(\beta)}=\tld{\kappa}_{s(\beta)}+m_{s(\beta)}$, the quantity in \eqref{eq: column bad 1-1}, in fact, belongs to
$$p^{m_{s(\beta)}}(X_{s(\beta)}+c_{\beta,m'_\beta}\cdot \mathbf{F}_{\geq -\alpha}+\mathbf{F}_{>s(\beta)}^\beta).$$
Comparing these two equations via the identity~\eqref{eq: eqation for column bad roots 1}, we have $$p^{m_{s(\beta)}}X_{s(\beta)}-c^{\imath}_{-\alpha}p^{m_{\beta}}Y_{\beta} \in p^{m_{s(\beta)}}(c_{\beta,m'_\beta}\cdot \mathbf{F}_{\geq -\alpha}+\mathbf{F}_{>s(\beta)}^\beta)+p^{m_\beta}\cdot c^{\imath}_{-\alpha}\cdot \mathbf{F}_{>\beta}.$$
As $c_{-\alpha}=-c^{\imath}_{-\alpha}$, by applying the second part of Lemma~\ref{lemma: equation from height conditions 1} we have
$$-cp^{m_{s(\beta)}}X_{s(\beta)}+p^{m_{\beta}+m_{-\alpha}}Y_{\beta} \in c\cdot p^{m_{s(\beta)}}(c_{\beta,m'_\beta}\cdot \mathbf{F}_{\geq -\alpha}+\mathbf{F}_{>s(\beta)}^\beta)+p^{m_\beta+m_{-\alpha}+1}\cdot  \mathbf{F}_{>\beta}.$$
As $s(\beta)=\beta-\alpha$, we get the desired result.
\end{proof}

We finally treat the case $\beta=\alpha_{lm}\in\Phi^-$ with $l>j_0$ and $m< i_0$. In this case, $\kappa_\beta$ is not subadditive in general, so that the following lemma is not trivial.
\begin{lemma}\label{lemma: left-bottom block 1}
For $\beta=\alpha_{lm}\in\Phi^-$ with $l>j_0$ and $m< i_0$,
$$V^\sharp_\beta\in v^{\kappa_\beta}(v+p)^{|\langle\eta,\beta^\vee\rangle|-1}(X_\beta+\mathbf{F}_{>\beta})$$ where $X_\beta=(m_\beta+\kappa_\beta-\langle\bf{a}_w,\beta^\vee\rangle)c_{\beta,m'_\beta}$.
\end{lemma}

\begin{proof}
We first claim that $V^{\imath}_{s(\beta_0)}\in v R[v]$ for a bad root $\beta_0\in\Phi^-$ sharing the row of $\alpha$. From part (1) of Lemma~\ref{lemma: equation from VV^-1=1 1} together with Lemma~\ref{lemma: decomposition 1-1} and Lemma~\ref{lemma: equation from height conditions 1}, we have $$p^{|\langle\eta,\alpha^\vee\rangle|}\left(c_{\beta_0}+\sum_{(\beta_1,\beta_2)\in\mathfrak{A}_{\beta_0}} c_{\beta_2}c^{\imath}_{\beta_1} +c^{\imath}_{\beta_0}\right)-c c^{\imath}_{s(\beta_0)}=0.$$
As $c_{\beta_0}+\sum_{(\beta_1,\beta_2)\in\mathfrak{A}_{\beta_0}} c_{\beta_2}c^{\imath}_{\beta_1} +c^{\imath}_{\beta_0}=0$ induced from extracting the constant term of $(V\cdot V^{\imath})_{\beta_0}=0$, we have $c c^{\imath}_{s(\beta_0)}=0$. Hence, we conclude that $c^{\imath}_{s(\beta_0)}=0$, as $c$ is a unit in $R[\frac{1}{p}]$ by Lemma~\ref{lemma: equation from height conditions 1} and $R$ is $\cO$-flat.

Let $\beta'=\alpha_{l'm'}\in\Phi^-$ with $m'<i_0$. We claim that $V^{\imath}_{\beta'}\in v^{\kappa_{\beta'}}R[v]$ if $l'\neq j_0$. It is clear that $V^{\imath}_{\beta'}\in v^{\kappa_{\beta'}}R[v]$ if $l'<j_0$, by Lemma~\ref{lemma: subadditive 1}. Assume $j_0< l'$. Consider the following identity $$0=(V\cdot V^{\imath})_{\beta'}=V_{\beta'}+\sum_{(\beta_1,\beta_2)\in\mathfrak{D}_{\beta'}}V_{\beta_2}\cdot V^{\imath}_{\beta_1}+V^{\imath}_{\beta'},$$
and write $\beta_1:=\alpha_{km'}$. It is obvious that $V_{\beta'}\in v^{\kappa_{\beta'}}R[v]$, and that $V_{\beta_2}\cdot V^{\imath}_{\beta_1}\in v^{\kappa_{\beta'}}R[v]$ for $m'<k<j_0$ by Lemma~\ref{lemma: subadditive 1}, so that it is enough to check that $V_{\beta_2}\cdot V^{\imath}_{\beta_1}\in v^{\kappa_{\beta'}}R[v]$ for $j_0\leq k<l'$.

Assume that $k=j_0$. If $s(\beta_1)$ is bad then $V^{\imath}_{\beta_1}\in vR[v]$ by the first claim, and so we conclude in this case that $V_{\beta_2}\cdot V^{\imath}_{\beta_1}\in v^{\kappa_{\beta'}}R[v]$. If $s(\beta_1)$ is not bad then it is clear that $V^{\imath}_{\beta_1}\in v^{\kappa_{\beta_1}}R[v]$, and so we also have $V_{\beta_2}\cdot V^{\imath}_{\beta_1}\in v^{\kappa_{\beta'}}R[v]$ by Lemma~\ref{lemma: subadditive 1}. Note that this also implies
\begin{equation}\label{eq: left bottom equation 1}
(\nabla V)_{\beta_2}V^{\imath}_{\beta_1}\in v^{\kappa_{\beta'}}R[v]
\end{equation}
in this case. Assume now that $j_0<k<l'$. By induction hypothesis, we have $V^{\imath}_{\beta_1}\in v^{\kappa_{\beta_1}}R[v]$ and so we conclude that $V_{\beta_2}\cdot V^{\imath}_{\beta_1}\in v^{\kappa_{\beta'}}R[v]$ by Lemma~\ref{lemma: subadditive 1}, which completes the proof of the second claim.

Now, let $\beta=\alpha_{lm}\in\Phi^-$ with $l>j_0$ and $m<i_0$. For $(\beta_1,\beta_2)\in\mathfrak{D}_\beta$, if we write $\beta_1=\alpha_{km}$ then it is clear that $(\nabla V)_{\beta_2}V^{\imath}_{\beta_1}\in v^{\kappa_{\beta}}R[v]$ for $k\neq j_0$, by the claim above together with Lemma~\ref{lemma: subadditive 1}. If $k=j_0$, then we also have $(\nabla V)_{\beta_2}V^{\imath}_{\beta_1}\in v^{\kappa_{\beta}}R[v]$ by \eqref{eq: left bottom equation 1}, which completes the proof.
\end{proof}

\subsection{Monodromy conditions: the second form}\label{subsec: monodromy cond, second}
Let $\tld{w}^*\in \Adm(\eta)$ be a colength one shape of the second form (cf. ~Proposition~\ref{prop:classification, colength 1 dual shapes}), and keep the notation of \S\ref{subsec: height condition, second}. In particular, we keep the notation of Proposition~\ref{prop: FH 2}.

\begin{lemma}
\label{lemma: routine monodromy 2}
Let $\beta=\alpha_{lm}\in\Phi^-$ with $l\leq j_0$ or $i_0\leq m$.
Then we have
\begin{enumerate}
\item if $\beta$ is not bad then $$V^\sharp_\beta \in v^{\tld{\kappa}_\beta}(v+p)^{m_\beta}\big(X_\beta+\mathbf{F}_{>\beta}\big)$$
    where $X_\beta=(m_\beta+\kappa_\beta-\langle\bf{a}_w,\beta^\vee\rangle)c_{\beta,m'_\beta}$ and $$\tld{\kappa}_\beta\defeq\left\{
                                 \begin{array}{ll}
                                  \kappa_\beta  & \hbox{if either $l<j_0$, $i_0<m$, or $i_0=m$ and $s(\beta)\in\Phi^-$ is not bad;} \\
                                  \kappa'_\beta-1  & \hbox{if either $i_0=m$ and $s(\beta)\in\Phi^-$ is bad or $l=j_0$,}
                                 \end{array}
                               \right.
     $$
\item if $\beta$ is bad then
\begin{equation*}
V^\sharp_\beta \in v^{\kappa_\beta}(v+p)^{m_\beta} \big(Y_\beta+vX_\beta+\mathbf{F}_{>\beta}\big)
\end{equation*}
where
$$\left\{
  \begin{array}{rl}
   X_\beta &=(m'_\beta+\kappa_\beta-\langle\bf{a}_w,\beta^\vee\rangle) c_{\beta,m'_\beta}; \\
   Y_\beta &=(m_\beta+\kappa_\beta-\langle\bf{a}_w,\beta^\vee\rangle)c_{\beta,m_\beta} +p(\kappa_\beta-\langle\bf{a}_w,\beta^\vee\rangle) c_{\beta,m'_\beta}.
  \end{array}
\right.$$
\end{enumerate}
\end{lemma}

\begin{proof}
The proof is similar to that of Lemma~\ref{lemma: routine monodromy 1} using Proposition~\ref{prop: FH 2} and Lemma~\ref{lemma: subadditive 2} instead of Proposition~\ref{prop: FH 1} and Lemma~\ref{lemma: subadditive 1}, respectively. We leave the details for the reader.
\end{proof}

For $\beta\in\Phi^-$ and for each integer $s\geq1$, we set $I_{\beta,s}$ to be the set of the tuples of negative roots $(\beta_1,\beta_2,\cdots,\beta_s)$ such that
\begin{itemize}
\item $\beta=\beta_1+\beta_2+\cdots+\beta_s$ and $\beta_1$ shares the column of $\beta$;
\item $\beta_i+\beta_{i+1}\in\Phi^-$ for $i\in\{1,\cdots,s-1\}$;
\item $w'(\beta_i)<0$ for all $i\in\{1,\cdots,s\}$.
\end{itemize}
Moreover, we set $$I_\beta\defeq\bigcup_{s\geq 1}I_{\beta,s}.$$
\begin{lemma}\label{lemma: shadow 2}
Let $\beta\in\Phi^-$ with $\beta\geq -\alpha$.
\begin{enumerate}
\item If $(\beta_1,\beta_2)\in\mathfrak{D}_{\beta}$ then $\gamma_\beta\geq \gamma_{\beta_1}+\gamma_{\beta_2}$;
\item $(\beta_1,\beta_2)\in\mathfrak{D}_{\beta}$ satisfies $\gamma_\beta=\gamma_{\beta_1}+\gamma_{\beta_2}$ if and only if $\kappa_\beta<\kappa_{\beta_1}+\kappa_{\beta_2}$.
\end{enumerate}
In particular, if $(\beta_1,\beta_2)\in\mathfrak{D}_{-\alpha}$ then
$\gamma_{\beta_1}+\gamma_{\beta_2}=\gamma_{-\alpha}$,
$\deg V^{\imath}_{\beta_1}= \gamma_{\beta_1}=m_{\beta_1}$, and $$\frac{1}{(\gamma_{\beta_1})!}\frac{d^{\gamma_{\beta_1}}V^{\imath}_{\beta_1}}{dv^{\gamma_{\beta_1}}}= \sum_{(\beta'_1,\cdots,\beta'_s)\in I_{\beta_1}} (-1)^{\ell+j_0-s}c_{\beta'_1,m_{\beta'_1}}c_{\beta'_2,m_{\beta'_2}}\cdots c_{\beta'_s,m_{\beta'_s}}$$
where $\beta_1=\alpha_{\ell i_0}$.
\end{lemma}

\begin{proof}
The proof is similar to that of Lemma~\ref{lemma: shadow 1}. We leave the details for the reader.
\end{proof}

\begin{lemma}\label{lemma: equation for def ring 2}
We have
\begin{equation*}
p^{m_{-\alpha}+1}\cdot(\langle\bf{a}_w,-\alpha^\vee\rangle+1) =p^{m_{-\alpha}}\cdot c\cdot Z_{-\alpha}
\end{equation*}
where
\begin{multline*}
Z_{-\alpha}\defeq\left(m_{-\alpha}-1-\langle\bf{a}_w,-\alpha^\vee\rangle\right) c_{-\alpha,m_{-\alpha}}\\
+\sum_{(\beta_1,\beta_2)\in\mathfrak{D}_{-\alpha}} \left(m_{\beta_2}-\langle\bf{a}_w,\beta_2^\vee\rangle\right)c_{\beta_2,m_{\beta_2}} \sum_{(\beta'_1,\cdots,\beta'_s)\in I_{\beta_1}} (-1)^{\ell+j_0-s}c_{\beta'_1,m_{\beta'_1}}\cdots c_{\beta'_s,m_{\beta'_s}}
\end{multline*}
if we write $\beta_1=\alpha_{\ell i_0}$.
\end{lemma}

\begin{proof}
By Lemma~\ref{lemma: shadow 2}, we see that $\deg((\nabla V)_{\beta_2}V^{\imath}_{\beta_1})=\gamma_{\beta_2}+\gamma_{\beta_1}=m_{\beta_2}+m_{\beta_1}=m_{-\alpha}-1$ for all $(\beta_1,\beta_2)\in\mathfrak{D}_{-\alpha}$. By the same argument as in Lemma~\ref{lemma: equation for def ring 1}, extracting the constant term of $v\cdot V_{-\alpha}^{\sharp}$ gives rise to the result. We leave the details for the reader.
\end{proof}

We further eliminate the variables $Y_\beta$ for bad roots $\beta\in\Phi^-$.
\begin{lemma}\label{lemma: bad cases 2}
Let $\beta\in\Phi^-$ be a bad root.
\begin{enumerate}
\item If $\beta$ shares the row of $\alpha$ then we have
$$p^{m_{s(\beta)}}\cdot(Y_\beta+c\cdot X_{s(\beta)})\in p^{m_{s(\beta)}}\cdot \left(c\cdot c_{\beta,m'_{\beta}}\cdot\mathbf{F}_{\geq-\alpha}+ \mathbf{F}_{>\beta}+c\cdot \mathbf{F}_{>s(\beta)}^{\beta}\right).$$
\item If $\beta$ shares the column of $\alpha$ then
$$p^{m_{s(\beta)}}\cdot(Y_\beta-c\cdot X_{s(\beta)})\in p^{m_{s(\beta)}}\cdot \left(c\cdot c_{\beta,m'_{\beta}}\cdot\mathbf{F}_{\geq-\alpha}+ \mathbf{F}_{>\beta}+c\cdot \mathbf{F}_{>s(\beta)}^{\beta}\right).$$
\end{enumerate}
\end{lemma}

\begin{proof}
The proof is similar to that of Lemma~\ref{lemma: bad cases 1}. The only difference is that we use Lemmas~\ref{lemma: equation from height conditions 2}, \ref{lemma: decomposition 2-1}, \ref{lemma: routine monodromy 2}, \ref{lemma: equation from VV^-1=1 2}, \ref{lemma: decomposition 2-2}, and \ref{lemma: equation from height conditions 2}  instead of Lemmas~\ref{lemma: equation from height conditions 1}, \ref{lemma: decomposition 1-1}, \ref{lemma: routine monodromy 1}, \ref{lemma: equation from VV^-1=1 1}, \ref{lemma: decomposition 1-2}, and \ref{lemma: equation from height conditions 1}, respectively. We leave the details for the reader.
\end{proof}

Finally, we treat the case $\beta=\alpha_{lm}\in\Phi^-$ with $l>j_0$ and $m<i_0$. In this case, $\kappa_\beta$ is not subadditive in general, so that the following lemma is not trivial.
\begin{lemma}\label{lemma: left-bottom block 2}
Assume that $V^{\imath}_{s(\beta)}\in v^{\kappa_\beta}R[v]$ for a bad root $\beta\in\Phi^-$ sharing the row of $\alpha$. For $\beta=\alpha_{lm}\in\Phi^-$ with $l>j_0$ and $m<i_0$, we have
$$V^\sharp_\beta \in v^{\kappa_\beta}(v+p)^{m_\beta}\big(X_\beta+\mathbf{F}_{>\beta}\big)$$
    where $X_\beta=(\langle\bf{a}_w,\beta^\vee\rangle +m_\beta+\kappa_\beta)c_{\beta,m'_\beta}$.
\end{lemma}

\begin{proof}
We first claim that $V^{\imath}_{s(\beta_0)}\in v^{\kappa_{\beta_0}}R[v]$ for each bad root $\beta_0\in\Phi^-$ sharing the row of $\alpha$, whose proof is almost identical to the first claim in the proof of Lemma~\ref{lemma: left-bottom block 1}. The rest of the proof also is similar to that of Lemma~\ref{lemma: left-bottom block 1}. The only difference is that we use Lemma~\ref{lemma: subadditive 2} instead of Lemma~\ref{lemma: subadditive 1}. We leave the details for the reader.
\end{proof}

\subsection{Description of $U(\tld{w},\leq\!\! \eta,\nabla_{\bf{a}})$ in colength $\leq 1$}
In this subsection we describe $U(\tld{w},\leq\!\! \eta,\nabla_{\bf{a}})$ when $\tld{w}^*\in \tld{W}$ has colength $\leq 1$.

We first give an upper bound of $U(\tld{w},\leq\!\! \eta,\nabla_{\bf{a}})$ for $\tld{w}^*$ of colength zero.
\begin{prop}\label{prop: naive main, zero}
Let $\tld{w}^*\in \Adm(\eta)$ be a colength zero shape, and let $\bf{a}\in\cO^n$.
Assume that $\ovl{\bf{a}}$ is $n$-generic $($\cite[Definition 4.2.2]{MLM}$)$.
Then there is a closed immersion
\[
U(\tld{w},\leq\!\! \eta,\nabla_{\bf{a}})\into \mathrm{Spec}\cO[\{X_\beta\mid \beta\in\Phi^-\}].
\]
\end{prop}

\begin{proof}
This follows from the arguments in \cite[\S 3.4]{LLL}, specifically the proof of \cite[Proposition 3.4.12]{LLL}.
(Note that the argument in \emph{loc.~cit}.~is written for complete local Noetherian $\cO$-algebras but is valid in our setting of $\cO$-flat Noetherian $\cO$-algebras.)
\end{proof}

We now give an upper bound of $U(\tld{w},\leq\!\! \eta,\nabla_{\bf{a}})$ for $\tld{w}^*$ of colength one. Recall that $Z_{-\alpha}$ is constructed in Lemma~\ref{lemma: equation for def ring 1} (resp. in Lemma~\ref{lemma: equation for def ring 2}) if $\tld{w}^*$ is of colength one of the first form (resp.~of the second form).

\begin{prop}\label{prop: naive main, one}
Let $\tld{w}^*\in \Adm(\eta)$ be a colength one shape, and let $\bf{a}\in\cO^n$.
Assume that $\ovl{\bf{a}}$ is $n$-generic $($as defined in \cite[\S 4.2]{MLM}$)$.
Then there is a closed immersion $$U(\tld{w},\leq\!\! \eta,\nabla_{\bf{a}})\hookrightarrow \mathrm{Spec}\frac{\cO[\{X_\beta\mid \beta\in\Phi^-\}\cup\{c\}]}{(c\cdot Z_{-\alpha}-p)}.$$
\end{prop}

\begin{proof}
By Proposition \ref{prop: space of height conditions}, $\cO(U(\tld{w},\leq\!\! \eta,\nabla_{\bf{a}}))$ is generated by $c_{\beta,k}$ (with $\beta\in\Phi^-$ and $0\leq k\leq m'_\beta$) and $c$.
But Lemmas~\ref{lemma: routine monodromy 1} ,~\ref{lemma: left-bottom block 1},~\ref{lemma: routine monodromy 2} and~\ref{lemma: left-bottom block 2} show that we can also generate using $c$, $X_\beta$ ($\beta\in\Phi^-$) and $Y_\beta$ ($\beta$ bad).
In turn, Lemmas~\ref{lemma: bad cases 1} and~\ref{lemma: bad cases 2} show that we can generate using just $c$ and $X_\beta$.
Finally Lemmas~\ref{lemma: equation for def ring 1} and~\ref{lemma: equation for def ring 2} (and $p$-flatness) give the relation $c\cdot Z_{-\alpha}-p$.
\end{proof}

\section{Colength one deformation rings}\label{sec:CL1:def}

In this section we apply the results of \S\ref{sec:finite height conditions} and \S\ref{sec: monodromy} to compute potentially crystalline deformation rings with Hodge--tate weights $\eta$, for sufficiently generic $\rhobar$ and tame inertial types $\tau$ such that $\tld{w}(\rhobar,\tau)$ has colength at most one in each embedding.

\subsection{Product structures and error terms}\label{subsec:error_term}
We first extend the technical results of \S \ref{subsec: monodromy cond, first} and \S \ref{subsec: monodromy cond, second} in a way which can be used to describe the closed immersion $\tld{U}(\tld{z},\leq\!\!\eta,\nabla_{\tau,\infty})\into \tld{U}(\tld{z},\leq\!\!\eta)^{\wedge_p}$, when $\tld{z}=(\tld{z}^{(j)})_{j\in\cJ}$ has colength at most one. This requires the modification of some of the previous formulas by allowing product structures, non-trivial diagonal entries, and an ``error term'' which takes into account the monodromy condition defined in \S \ref{subsub:true_mon}.

Keep the notation of \S\ref{sec:finite height conditions}, and let $R$ be a $p$-adically complete, topologically finite type, Noetherian $\cO$-algebra.
Let $\tld{w}^*\in \Adm(\eta)$ be a colength one shape of the first form (resp. of the second form), and $\tld{A}\in \tld{U}(\tld{w},\leq\!\!\eta)(R)$ with its image $A\in U(\tld{w},\leq\!\!\eta)(R)$ under the natural morphism $\tld{U}(\tld{w},\leq\!\!\eta)\rightarrow U(\tld{w},\leq\!\!\eta)$. We may write $w^{-1}\tld{A}w=sD_0s\cdot w^{-1}A w$ (resp. $w'^{-1}\tld{A}w'=D_0\cdot w'^{-1}A w'$) for some $D_0=\mathrm{Diag}(a_1,\cdots,a_n)\in T^\vee(R)$. Let $V\in \frac{1}{v}\Mat_n(R[v])$ be the matrix obtained from Proposition \ref{prop: FH 1} (resp. from Proposition \ref{prop: FH 2}) applied to $A\in U(\tld{w},\leq\!\!\eta)(R)$, and set $\tld{V}\defeq D_0\cdot V$. Then by the same argument as in Proposition \ref{prop: FH 1} (resp. in Proposition \ref{prop: FH 2}) we may write
$(w^{-1}\tld{A}w)_{j_0j_0}=a_{j_0j_0}(v+p)^{n-j_0}$ (resp. $(w'^{-1}\tld{A}w^{\prime})_{i_0i_0}=a_{i_0i_0}(v+p)^{n-j_0})$
for some $a_{j_0j_0}\in R$ (resp. for some $a_{i_0i_0}\in R$), and we have the following identity:
\begin{multline}\label{eq: elementary operation with diagonal entries}
s\cdot u_{-\alpha}(-\frac{a_{j_0j_0}}{a_{j_0}})\cdot(w^{-1}\tld{A}w)=(v+p)^\eta\cdot \tld{V}\\
\big(\mbox{resp. }u_{\alpha}(-\frac{a_{i_0i_0}}{a_{j_0}})\cdot v^{\varepsilon_{i_0}}\cdot (w'^{-1}\tld{A}w')\cdot s=(v+p)^\eta\cdot v^{\varepsilon_{j_0}} \cdot\tld{V}\big).
\end{multline}
Note that we may identify $c$, defined in Proposition~\ref{prop: FH 1} (resp. in Proposition~\ref{prop: FH 2}), with $a_{j_0j_0}/a_{i_0}$ (resp. with $a_{i_0i_0}/a_{i_0}$). We also note that the degree description of each entry of $\tld{V}$ is exactly the same as that of $V$, as $\tld{V}=D_0\cdot V$.

We fix a tame inertial type $\tau$ with a $N$-generic lowest alcove presentation $(s,\mu)$, together with an element $\tld{z}\in\Adm(\eta)^\vee$ which we write as $\tld{z}=(\tld{z}^{(j)})_{j\in\cJ}$.

Until the end of this subsection, assume that $j\in \cJ$ is such that $\ell(\tld{z}^{(j)})=\ell(t_{\eta})-1$, and let $A^{(j)}\in U(\tld{z}^{(j)},\leq\!\!\eta)(R)$ be the image of $\tld{A}^{(j)}\in \tld{U}(\tld{z}^{(j)},\leq\!\!\eta)(R)$. Let $\tld{V}^{(j)}\in \frac{1}{v}\Mat_n(R[v])$ be the matrix obtained from \eqref{eq: elementary operation with diagonal entries} (according to the two possible forms of $\tld{z}^{(j)}$) applied to $\tld{A}^{(j)}\in \tld{U}(\tld{z}^{(j)},\leq\!\!\eta)(R)$.
(We adapt the notation of Propositions \ref{prop: FH 1} and \ref{prop: FH 2} as well as condition \eqref{eq: elementary operation with diagonal entries} in our context by adding a superscript $(j)$, so that for instance a  colength one shape of the second form has decomposition $w^{(j)}s_{\alpha^{(j)}}t_{\eta-\alpha^{(j)}}{w^{(j)}}^{-1}$.) It is easy to see that condition~\eqref{equ: monodromy condition} shows that condition \eqref{eq:true:mon:cond:A} has the form
\begin{equation}\label{eq:true:mon:cond:V}
\tld{V}^{(j),\sharp}_{\beta^{(j)}}\defeq\left((\nabla_{\bf{a}_w} \tld{V}^{(j)})\cdot\tld{V}^{(j)\,\imath}\right)_{\beta^{(j)}}
\in (v+p)^{m_{\beta^{(j)}}}R[v]+p^{N-2n+5}R[\![v]\!]
\end{equation}
for all $\beta^{(j)}\in\Phi^-$, where $\bf{a}_w\in\Z^n$ is defined by $w^{(j)}(\bf{a}_w)=\bf{a}^{(j)}$.

As $\tld{V}^{(j)}=D_0\cdot V^{(j)}$ for some $D_0\defeq\mathrm{Diag}(a_1,\cdots,a_n)\in T^\vee(R)$, if we let $\beta^{(j)}=\alpha_{lm}\in\Phi^-$ then it is easy to see that $$a_m\cdot \tld{V}^{(j),\sharp}_{\beta^{(j)}}=a_l\cdot V^{(j),\sharp}_{\beta^{(j)}}.$$ Hence, condition \eqref{eq:true:mon:cond:V} applied to $\tld{V}^{(j)}$ induces all the relevant lemmas from \S \ref{subsec: monodromy cond, first} and \S \ref{subsec: monodromy cond, second} keeping track of the ``error term'' $p^{N-2n+5}$, as the diagonal entries $a_k$ of $D_0$ are units in $R$. More precisely,
\begin{itemize}
\item in Lemma \ref{lemma: routine monodromy 1} (resp. in Lemma~\ref{lemma: routine monodromy 2}), we have
\begin{enumerate}
\item $V^\sharp_{\beta^{(j)}}\in v^{\tld{\kappa}_{\beta^{(j)}}}(v+p)^{m_{\beta^{(j)}}}\big(X_{\beta^{(j)}}+\mathbf{F}_{>{\beta^{(j)}}}\big)+p^{N-2n+5}R[\![v]\!]$;
\item $V^\sharp_{\beta^{(j)}}\in v^{\kappa_{\beta^{(j)}}}(v+p)^{m_{\beta^{(j)}}} \big(Y_{\beta^{(j)}}+vX_{\beta^{(j)}}+\mathbf{F}_{>{\beta^{(j)}}}\big)+p^{N-2n+5}R[\![v]\!]$,
\end{enumerate}
\item in Lemma \ref{lemma: equation for def ring 1} (resp. in Lemma~\ref{lemma: equation for def ring 2}), we have
\begin{equation*}
p^{m_{-\alpha^{(j)}}+1}\cdot(\langle\bf{a}_w,-\alpha^{(j)\,\vee}\rangle+\kappa_{-\alpha^{(j)}}) \in p^{m_{-\alpha^{(j)}}}\cdot c\cdot Z_{-\alpha^{(j)}}+p^{N-2n+5}R,
\end{equation*}
\item in Lemma \ref{lemma: bad cases 1} (resp. in Lemma~\ref{lemma: bad cases 2}), we have
\begin{equation*}
p^{m_{s({\beta^{(j)}})}}\cdot(Y_{\beta^{(j)}}\pm c\cdot X_{s({\beta^{(j)}})})\in
p^{m_{s({\beta^{(j)}})}}\cdot \left(c\cdot c_{{\beta^{(j)}},m'_{{\beta^{(j)}}}}\cdot\mathbf{F}_{\geq-{\alpha^{(j)}}}+ \mathbf{F}_{>{\beta^{(j)}}}+c\cdot \mathbf{F}_{>s({\beta^{(j)}})}^{{\beta^{(j)}}}\right)+p^{N-2n+5}R,
\end{equation*}

\item in Lemma \ref{lemma: left-bottom block 1} (resp. in Lemma~\ref{lemma: left-bottom block 2}), we have
\[
V^\sharp_{\beta^{(j)}} \in v^{\kappa_{\beta^{(j)}}}(v+p)^{m_{\beta^{(j)}}}\big(X_{\beta^{(j)}}+\mathbf{F}_{>{\beta^{(j)}}}\big)+p^{N-2n+5}R[\![v]\!].
\]
\end{itemize}

\subsection{Potentially crystalline deformation rings}

Fix $\rhobar:G_K\rightarrow \GL_n(\F)$, and let $\tau$ be a tame inertial type.
We assume that $\tau$ has an $N$-generic lowest alcove presentation (\cite[Definition 2.4.3]{MLM}) with $N>3n-6$.
Assume that $R_{\rhobar}^{\eta,\tau}\neq 0$ so that in particular $\tld{w}(\rhobar,\tau)$ is defined.
If $\tld{w}(\rhobar,\tau)^{(j)}=\ell(t_\eta)-1$ for some $j\in \cJ$, then it determines a positive root $\alpha^{(j)}$ and we write $Z_{-\alpha^{(j)}}$ for the element $Z_{-\alpha}$ constructed using Lemma~\ref{lemma: equation for def ring 1} (resp.~using Lemma~\ref{lemma: equation for def ring 2}) if $\tld{w}^*$ is of colength one of the first form (resp.~of the second form) taking into account the ``error term'' as explained in \S \ref{subsec:error_term}.
(Note that the element $Z_{-\alpha}$ is defined exactly because $N-2n+5>m_{\beta}+1$ for all negative roots $\beta$.)

\begin{lemma}
\label{lemma:explicit:def:ring}
Let $\tau$ be a tame inertial type with an $N$-generic lowest alcove presentation, where $N>3n-6$. Assume that $\tld{w}(\rhobar,\tau)$ satisfies $\ell(\tld{w}(\rhobar,\tau)^{(j)})\geq \ell(t_\eta)-1$ for each $j\in\cJ$. Then there is a closed immersion
$$\tld{U}((\tld{w}(\rhobar,\tau)^{(j)})^*,\leq\!\!\eta,\nabla_{\bf{a}^{(j)}})\hookrightarrow\mathrm{Spec}R^{(j)}$$
where $\bf{a}^{(j)}\in\Zp$ are the constants defined in \eqref{eq:monodromy_str_constant} and
$$R^{(j)}\defeq\left\{
  \begin{array}{ll}
    \frac{\cO[\{X_\beta\mid \beta\in\Phi^-\}\cup\{c\}]}{(c\cdot Z_{-\alpha^{(j)}}-p)}\otimes_\cO\left(\bigotimes_{\cO,\,i=1}^n\frac{\cO[a_i,Y_i]}{(a_i\cdot Y_i-1)}\right) & \hbox{if $\ell(\tld{w}(\rhobar,\tau)^{(j)})=\ell(t_\eta)-1$;} \\
    \cO[\{X_\beta\mid \beta\in\Phi^-\}]\otimes_\cO\left(\bigotimes_{\cO,\,i=1}^n\frac{\cO[a_i,Y_i]}{(a_i\cdot Y_i-1)}\right) & \hbox{otherwise.}
  \end{array}
\right.
$$
\end{lemma}

\begin{proof}
The results follow immediately from Proposition~\ref{prop: naive main, zero} and Proposition~\ref{prop: naive main, one} together with equation~\eqref{eq: left torus action}.
\end{proof}

\begin{prop}
\label{prop:cl:imm:def:ring}
Let $\tau$ be a tame inertial type with an $N$-generic lowest alcove presentation, where $N>3n-6$.
Assume that $\tld{w}(\rhobar,\tau)$ satisfies that $\ell(\tld{w}(\rhobar,\tau)^{(j)})\geq \ell(t_\eta)-1$ for each $j\in\cJ$.
Then there is a closed immersion $$\tld{U}(\tld{w}(\rhobar,\tau)^*,\leq\!\!\eta,\nabla_{\tau,\infty})\hookrightarrow \mathrm{Spf}\,\bigg(\bigotimes_{\cO,\,j\in\cJ}R^{(j)}\bigg)^{\wedge_p}$$
where the rings $R^{(j)}$ have been defined in Lemma~\ref{lemma:explicit:def:ring}.
\end{prop}

\begin{proof}
The proof goes very similar to the ones of Proposition~\ref{prop: naive main, one} and Proposition~\ref{prop: naive main, zero} together with Lemma~\ref{lemma:explicit:def:ring}. The only difference is that we need to take care of the error terms. For instance, if $(\tld{w}(\rhobar,\tau)^{(j)})^*$ is of colength one of the first form, and if $\beta$ is a bad root sharing the row of $\alpha$, then there exist $F_\beta\in \mathbf{F}_{\geq-\alpha}$, $G_\beta\in \mathbf{F}_{>\beta}$, and $H_{s(\beta)}\in \mathbf{F}_{>s(\beta)}^{\beta}$ such that
$$Y_\beta=-c\cdot X_{s(\beta)}+ \left(c\cdot c_{\beta,m'_{\beta}}\cdot F_\beta+ G_\beta+c\cdot H_{s(\beta)}\right)+O\left(p^{N-2n+5-m_{s(\beta)}}\right),$$
by Lemma~\ref{lemma: bad cases 1}. The coefficient $c_{\beta,m'_{\beta}}$ corresponds to $X_\beta$, and so due to our generic assumption, by scaling $Y_\beta$ we can further eliminate the variable $Y_\beta$.

Repeating the same arguments, we conclude that there is a surjection
$$\bigg(\bigotimes_{\cO,\,j\in\cJ}R^{(j)}\bigg)^{\wedge_p}\twoheadrightarrow \cO(\tld{U}(\tld{w}(\rhobar,\tau)^*,\leq\!\!\eta,\nabla_{\tau,\infty})),$$ which completes the proof.
\end{proof}

Set $$\cJ_0\defeq\{j\in\cJ\mid \ell(\tld{w}(\rhobar,\tau)^{(j)})=\ell(t_\eta)-1 \mbox{ and }Z_{-\alpha^{(j)}}\equiv 0\pmod{\varpi}\}.$$
\begin{thm}\label{thm: main deformation ring}
Let $\tau$ be a tame inertial type with an $N$-generic lowest alcove presentation, where $N>3n-6$.
Assume that $R^{\eta,\tau}_{\rhobar}\neq 0$, and the shape $\tld{w}(\rhobar,\tau)$ satisfies $\ell(\tld{w}(\rhobar,\tau)^{(j)})\geq \ell(t_\eta)-1$ for each $j\in\cJ$. Then $R_{\rhobar}^{\eta,\tau}$ is formally smooth over
\begin{equation}
\label{eq:model:dr}
\widehat{\bigotimes}_{\cO,\,j\in \cJ_0}\frac{\cO[[X,Y]]}{(XY-p)}.
\end{equation}
\end{thm}

\begin{proof}
First, notice that we may replace $Z_{-\alpha}$ with $X_{-\alpha}$, due to the equations of $Z_{-\alpha}$ in Lemma~\ref{lemma: equation for def ring 1} and Lemma~\ref{lemma: equation for def ring 2}.
As $\rhobar\in \cX^{\leq\eta,\tau}(\tld{w}(\rhobar,\tau)^*)(\F)$, we can pick $\tld{A}\in \tld{U}(\tld{w}(\rhobar,\tau)^*,\leq\!\!\eta)(\F)$ corresponding to $\rhobar|_{G_{K_\infty}}$ as explained in \S \ref{subsub:true_mon}.
Then the completion of $\cO\big(\tld{U}(\tld{w}(\rhobar,\tau)^*,\leq\!\!\eta,\nabla_{\tau,\infty})\big)$ at $\tld{A}$ is formally smooth over the ring in \eqref{eq:model:dr}, by dimension counting, and by \eqref{diag:main}, the deformation ring $R_{\rhobar}^{\leq\eta\,\tau}$ is also formally smooth over the ring in \eqref{eq:model:dr}.
As the latter is irreducble, so is $R_{\rhobar}^{\leq\eta\,\tau}$, in particular $R_{\rhobar}^{\leq\eta\,\tau}=R_{\rhobar}^{\eta\,\tau}$
 which completes the proof.
\end{proof}

\begin{rmk}
Under stronger genericity assumptions on $\tau$ we have $\tld{U}(\tld{z},\leq\!\!\eta,\nabla_{\tau,\infty})\neq \emptyset$ whenever $\tld{z}\in \Adm(\eta)^\vee$ (\cite[Lemma 7.3.5]{MLM}).
\end{rmk}

\section{Applications}

In this section we elaborate on how the explicit description of the potentially crystalline deformation rings from Theorem \ref{thm: main deformation ring} can provide information on representation theory and automorphic forms through following the philosophy of the mod $p$ local Langlands correspondence.

\subsection{Subextremal weights}
\label{subsec: sub ext weight}
In this section we refine, in Definition \ref{def:def:SW}, the notion of \emph{defect} for Serre weights of $\rhobar:G_K\ra\GL_n(\F)$ introduced in \cite[\S 8.6.1]{MLM} and \cite[\S 3.7]{OBW}.

We fix once and for all a lowest alcove presentation $\tld{w}(\rhobar^{\semis})$ for $\rhobar$.
All tame inertial types will be endowed with the unique lowest alcove presentation compatible with $\tld{w}(\rhobar^{\semis})$. Throughout this subsection, we assume $S_p=\{v\}$ so that $F^+_p=K$.

Recall from \cite[Theorem 3.7]{CEGGPS} that given a tame inertial type $\tau$ for $K$ there exists an irreducible smooth representation $\sigma(\tau)$ of $\GL_n(\cO_K)$ which satisfies  properties towards the inertial local Langlands correspondence.
By \cite[Theorem 2.5.3]{MLM}, if $\tau$ has a $1$-generic lowest alcove presentation $(s,\mu-\eta)$ then $\sigma(\tau)$ can be taken to be $R_s(\mu)$.

Let $\rhobar^{\speci}:I_K\ra \GL_n(\F)$ be a $2n$-generic tame inertial $\F$-type.
The choice of a lowest alcove presentation $\tld{w}(\rhobar^{\speci})=t_{\mu+\eta}s$ for it gives a map
\begin{align*}
\tau_{\rhobar^{\speci}}:\Adm^{\textnormal{reg}}(\eta)&\longrightarrow\left\{\tau:I_K\ra\GL_n(E)\right\}.\\
t_\nu w&\mapsto \tau(sw,\mu+\eta+s(\nu))
\end{align*}
(Note that given $\tld{w}\in \Adm^{\textnormal{reg}}(\eta)$ we have $\tld{w}(\rhobar^{\speci},\tau_{\rhobar^{\speci}}(\tld{w}))=\tld{w}$ by construction.)

Recall from \cite[\S 2.3.1]{MLM} the background on Deligne--Lustig representations and their lowest alcove presentations.
In particular given a Deligne--Lusztig representation $R$ with a $2(n-1)$-generic lowest alcove presentation $\tld{w}(R)$ we have a set $\JH_{\textnormal{out}}(\ovl{R})$ of \emph{outer weights} for $R$.
We also recall that we have a bijection
\begin{align}
\label{eq:bij:pairs:admreg}
\left\{(\tld{w}_1,\tld{w}_2)\in \left(\tld{\un{W}}^+\times\tld{\un{W}}_1^+\right)/X^0(\un{T})\mid \tld{w}_1\uparrow\tld{w}_h^{-1}\tld{w}_2
\right\}&\longrightarrow \Adm^{\textnormal{reg}}(\eta)\\
(\tld{w}_1,\tld{w}_2)&\longmapsto \tld{w}_2
^{-1}w_0\tld{w}_1\nonumber
\end{align}
from \cite[Remark 2.1.8]{MLM}, and a bijection
\begin{align*}
\left\{(\tld{w}_1,\tld{w}_2)\in \left(\tld{\un{W}}^+\times\tld{\un{W}}_1^+\right)/X^0(\un{T})\mid \tld{w}_1\uparrow\tld{w}_2
\right\}&\longrightarrow W^?(\rhobar^{\speci})\\
(\tld{w}_1,\tld{w}_2)&\longmapsto F_{(\tld{w}_2,\tld{w}(\rhobar^{\speci})(\tld{w}_1)^{-1}(0))}\nonumber
\end{align*}
from \cite[Proposition 2.6.2]{MLM}.
As multiplication by $\tld{w}_h$ gives a self bijection on $\tld{\un{W}}_1^+$, we finally obtain a bijection
\begin{align*}
\label{eq:bij:pair&SW}
\sigma_{\rhobar^{\speci}}:\Adm^{\textnormal{reg}}(\eta)&\longrightarrow W^?(\rhobar^{\speci}).\\
\tld{w}_2^{-1}w_0\tld{w}_1&\longmapsto F_{(\tld{w}_h^{-1}\tld{w}_2,\tld{w}(\rhobar^{\speci})(\tld{w}_1)^{-1}(0))}\nonumber
\end{align*}
We write $\tld{w}_{\rhobar^{\speci}}$ for the inverse of $\sigma_{\rhobar^{\speci}}$.

\begin{lemma}
\label{lem:reg:type}
Assume that $\tld{w}(\rhobar^{\speci})$ is a $2n$-generic lowest alcove presentation for $\rhobar^{\speci}$, and let $\tld{w}\in \Adm^{\textnormal{reg}}(\eta)$.
Then $\sigma_{\rhobar^{\speci}}(\tld{w})\in W^?(\rhobar^{\speci})\cap \JH_{\textnormal{out}}(\ovl{\sigma(\tau_{\rhobar^{\speci}}(\tld{w}))})$ and it satisfies the following property: for any $\sigma'\in W^?(\rhobar^{\speci})\cap \JH(\ovl{\sigma(\tau_{\rhobar^{\speci}}(\tld{w}))})$ we have
\begin{equation}
\label{eq:max:def}
\tld{w}_{\rhobar^{\speci}}(\sigma')^{(j)}\geq \tld{w}^{(j)}
\end{equation}
for all $j\in\cJ$ with equality for all $j\in\cJ$ if and only if $\sigma'=\sigma$.
\end{lemma}
\begin{proof}
This is an immediate consequence the proof of \cite[Proposition 8.6.3]{MLM} of which we employ here the notation and convention.
In particular we let $\tau\defeq \tau_{\rhobar^{\speci}}(\tld{w})$ so that  $\tld{w}=\tld{w}(\rhobar,\tau)\in \Adm^{\textnormal{reg}}(\eta)$.
Using the bijection \eqref{eq:bij:pairs:admreg} we decompose $\tld{w}$ as $\tld{w}_2^{-1}w_0\tld{w}_1$ where $(\tld{w}_1,\tld{w}_2)\in \tld{\un{W}}^+\times\tld{\un{W}}_1^+ $ satisfies $\tld{w}_1\uparrow \tld{w}_h^{-1}\tld{w}_2$.

We have $\sigma_{\rhobar^{\speci}}(\tld{w})=F_{(\tld{w}_h^{-1}\tld{w}_2,\tld{w}(\rhobar^{\speci})\tld{w}_1^{-1}(0))}$ by definition, and the latter Serre weight is the element $\kappa\in W^?(\rhobar^{\speci})\cap \JH(\ovl{\sigma(\tau)})$ defined in \cite[Proposition 8.6.3]{MLM}.
As $\tld{w}(\rhobar^{\speci})\tld{w}_1^{-1}(0)=\tld{w}(\tau)\tld{w}_2^{-1}(0)$  by \cite[Proposition 2.6.4]{MLM}, the weight $\kappa$ is in $\JH_{\textnormal{out}}(\ovl{\sigma(\tau)})$, by definition of $\JH_{\textnormal{out}}(\ovl{\sigma(\tau)})$ (see \cite[Proposition 2.3.7]{MLM} and the beginning of \cite[\S 2.3.1]{MLM}).

The fact that for any $\sigma'\in W^?(\rhobar^{\speci})\cap \JH(\ovl{\sigma(\tau)})$ the inequality \eqref{eq:max:def} holds is immediate from the proof of \cite[Proposition 8.6.3]{MLM}.
We provide the details: by \cite[Proposition 2.6.4]{MLM} any $\sigma'\in W^?(\rhobar^{\speci})\cap \JH(\ovl{\sigma(\tau)})$ is of the form $\sigma'=F_{(\tld{w}',\tld{w}(\tau)\tld{s}_2^{-1}(0))}$ for some $\tld{w}'\in \un{\tld{W}}_1^+$ and some pair $(\tld{s}_1,\tld{s}_2)\in \tld{\un{W}}^+\times \tld{\un{W}}^+$ satisfying $\tld{s}_1\uparrow \tld{w}'\uparrow\tld{w}_h^{-1}\tld{s}_2$ and $\tld{s}_2^{-1}w\tld{s}_1=\tld{w}(\rhobar^{\speci},\tau)$ for some $w\in \un{W}$.
By Wang's Theorem (see \cite[Theorem 4.1.1]{LLL}) the condition $\tld{s}_2\uparrow \tld{w}_h\tld{w}'$, which is defined \emph{embeddingwise}, gives $\tld{s}^{(j)}_2\leq \tld{w}^{(j)}_h\tld{w}^{\prime\,(j)}$ for all $j\in\cJ$ and by \cite[Lemma 4.1.9]{LLL} we conclude that
\begin{equation*}
(\tld{w}^{(j)}_h\tld{w}^{\prime\,(j)})^{-1}w_0\tld{s}_1^{(j)}\geq (\tld{s}_2^{(j)})^{-1}w_0\tld{s}_1^{(j)}
\geq (\tld{s}_2^{(j)})^{-1}w\tld{s}_1^{(j)}=\tld{w}^{(j)}
\end{equation*}
since $(\tld{s}_2^{(j)})^{-1}w_0\tld{s}_1^{(j)}$, $(\tld{w}^{(j)}_h\tld{w}^{\prime\,(j)})^{-1}w_0\tld{s}_1^{(j)}$ are reduced expressions for all $j\in\cJ$ and $w_0\geq w$.
As $\tld{s}_1\uparrow \tld{w}'\in\tld{\un{W}}_1^+$ we see that $(\tld{w}_h\tld{w}',\tld{s}_1)$ defines an element in the left hand side of \eqref{eq:bij:pairs:admreg} and hence $\tld{w}_{\rhobar^{\speci}}(\sigma')=(\tld{w}_h\tld{w}')^{-1}w_0\tld{s}_1$, proving \eqref{eq:max:def}.
The fact that the equality holds for all $j\in\cJ$ if and only if $\sigma'=\sigma_{\rhobar^{\speci}}(\tld{w})$ is immediate since $\sigma_{\rhobar^{\speci}}$ is a bijection.
\end{proof}

The usual order on $\N$ induces the product partial order on $\N^\cJ$, and for $\un{h}=(h^{(j)})_{j\in\cJ}\in\N^\cJ$ we define
\begin{equation*}
W^?_{\leq \un{h}}(\rhobar^{\speci})\defeq\left\{\sigma\in W^?(\rhobar^{\speci}) \mid \ell\big(\tld{w}_{\rhobar^{\speci}}(\sigma)^{(j)}\big) \geq \ell(t_{\eta})-h^{(j)} \text{ for all $j\in\cJ$} \right\}.
\end{equation*}

Let $\rhobar:G_{K}\ra\GL_n(\F)$ be a continuous Galois representation such that $\rhobar^{\semis}$ is $0$-generic (so that $\tld{w}(\rhobar^{\semis})\in\tld{\un{W}}$ is defined).
We say that a tame inertial $\F$-parameter $\rhobar^{\speci}:I_K\ra\GL_n(\F)$ is a \emph{specialization} of $\rhobar$, and write $\rhobar\leadsto \rhobar^{\speci}$, if there exists an $n$-generic tame inertial type $\tau$ such that $\rhobar$ is $\tau$-admissible and $\tld{w}(\rhobar,\tau)=\tld{w}(\rhobar^{\speci},\tau)$.
(In this definition, the lowest alcove presentations of $\tau$ and $\rhobar^{\speci}$ are always assumed to be compatible with a fixed lowest alcove presentation of $\rhobar^{\semis}$.)

Let $\cX_{n,K}$ be the Noetherian formal algebraic stack over $\Spf\cO$ defined in \cite[Definition 3.2.1]{EGstack}.
It has the property that $\cX_{n,K}(\F)$ is isomorphic to the groupoid of continuous representations of $G_K$ over rank $n$ vector spaces over $\F$.
Moreover there is a bijection $\sigma\mapsto \cC_\sigma$ between Serre weights of $\rG(\defeq\GL_n(k))$ and irreducible components of the reduced special fiber of $\cX_{n,K}$, described in \cite[\S 7.4]{MLM}.
(We refer the reader to \cite[\S 2.2]{MLM} concerning Serre weights and their lowest alcove presentations.)
This bijection is a renormalization of the bijection $\sigma\mapsto \cX_{n,\textnormal{red}}^\sigma$ of \cite[Theorem 6.5.1]{EGstack}.

In particular, if $\rhobar\in \cX_{n,K}(\F)$ we define the set of \emph{geometric} weights of $\rhobar$ as
\begin{equation*}
W^g(\rhobar)\defeq\left\{\sigma\mid \rhobar\in \cC_\sigma(\F)\right\}.
\end{equation*}

\begin{defn}
\label{def:def:SW}
Let $\rhobar:G_{K}\ra\GL_n(\F)$ be a continuous Galois representation such that $\rhobar^{\semis}$ is $0$-generic.
For $\un{h}=(h^{(j)})_{j\in\cJ}\in \N^{\cJ}$ define $W^g_{\leq \un{h}}(\rhobar)$ to be
\[
W^g(\rhobar)\cap \bigg(\bigcup_{\rhobar\leadsto\rhobar^{\speci}}W^?_{\leq\un{h}}(\rhobar^{\speci}) \bigg).
\]
\end{defn}

In what follows we write $\un{1}$ for the tuple of $\un{h}\in \N^{\cJ}$ satisfying $h^{(j)}=1$ for all $j\in\cJ$, and similarly for $\un{0}$.
Note that $W^g_{\leq \un{0}}(\rhobar)=W_{\textnormal{extr}}(\rhobar)$ is the set of \emph{extremal weights} in \cite[Definition 3.7.1]{OBW}.

\subsection{Application to patching functors}
\label{subsec: App to patching}

We introduce the formalism of patching functors following \cite[\S 5.2]{OBW}, giving applications to the results on the deformation rings in \S \ref{sec:CL1:def}.

\subsubsection{$L$-parameters}
Recall that $F^+_p$ is a finite unramified \'etale $\Q_p$-algebra, which we write as $\prod_{v\in S_p}F_v^+$ for a finite set $S_p$ and finite unramified extensions $F^+_v/\Qp$.
We assume throughout that for any $v\in S_p$ the coefficient field $E$ (resp.~$\F$) contains the image of any homomorphism $F^+_v\into \ovl{\Qp}$ (resp.~$k_v\into\ovl{\F}_p$, where $k_v$ denotes the residue field of $F^+_v$).
We let $\un{G}^{\vee}$ denote the product $\prod_{F^+_p\ra E}{\GL_n}^{\vee}_{/\cO}$ (the dual group of $\Res_{F^+_p/\Qp}{\GL_n}_{/\cO_p}$) and ${}^{L}\un{G}(\F)\defeq \un{G}^\vee \rtimes \Gal(E/\Qp)$, where $\Gal(E/\Qp)$ acts on the set $\{\iota:F^+_p\ra E\}$ by post-composition.
An $L$-homomorphism over $A\in\{E,\F\}$ is a continuous homomorphism $\rbar_p:G_{\Qp}\ra {}^{L}\un{G}(A)$.
An $L$-parameter is a $\un{G}^\vee(E)$-conjugacy class of an $L$-homomorphism.
An tame inertial $L$-parameter is a $\un{G}^\vee(E)$-conjugacy class of an homomorphism $\tau:I_{\Qp}\ra \un{G}^\vee(E)$ which has open kernel and factors through the tame quotient of $I_{\Qp}$, and which admits an extension to an $L$-homomorphism.
By \cite[Lemma 9.4.1, Lemma 9.4.5]{GHS}, the datum of an $L$-parameters $\rbar_p$ (resp.~a tame inertial $L$-parameter $\tau$) is equivalent to the datum of a collection of continuous homomorphisms $\{\rho_v: G_{F_v^+}\ra\GL_n(A)\}_{v\in S_p}$ (resp.~tame inertial types $\{\tau_v: I_{F_v^+}\ra\GL_n(A)\}_{v\in S_p}$).
Via this bijection, we can therefore give the notion of lowest alcove presentations and genericity for $L$-parameters.
Given a tame inertial $L$-parameter $\tau$ with corresponding collection $\{\tau_v: I_{F_v^+}\ra\GL_n(A)\}_{v\in S_p}$ of tame inertial types, we let $\sigma(\tau)$ be the tame smooth irreducible representation of $\GL_n(\cO_p)$ over $E$ given by $\otimes_{v\in S_p,E}\sigma(\tau_v)$, where for each $v\in S_p$ we let $\sigma(\tau_v)$ be the tame smooth irreducible representation of $\GL_n(\cO_{F^+_v})$ over $E$ attached to $\tau_v$ via the inertial local Langlands correspondence of \cite[Proposition 2.5.5]{MLM}.

\subsubsection{Patching functors and Serre weights}
Let now $\rbar_p:G_{\Qp}\ra {}^{L}\un{G}(\F)$ be an $L$-homomorphism, with corresponding collection $\{\rhobar_v: G_{F_v^+}\ra \GL_{n}(\F)\}_{v\in S_p}$.

We let $R^p$ be a nonzero complete local Noetherian equidimensional flat $\cO$-algebra with residue field $\F$ such that each irreducible component of $\Spec R^p$ and of $\Spec \overline{R}^p$ is geometrically irreducible, and define
\[
R_{\rbar_p}\defeq\widehat{\bigotimes}_{v\in S_p,\cO}R^{\Box}_{\rhobar_v} ,\qquad\qquad R_\infty\defeq R^p\widehat{\otimes}_{\cO} R_{\rbar_p}
\]
(we suppress the dependence on $R^p$ in the notation of $R_\infty$).

Given a tame inertial $L$-homomorphism $\tau: I_{\Qp}\ra \un{G}^\vee(E)$, with corresponding collection $\{\tau_v\}_{v\in S_p}$, we define
\[
R_{\rbar_p}^{\eta,\tau}\defeq \widehat{\bigotimes}_{v\in S_p}R_{\rhobar_v}^{\eta_v,\tau_v}
\]
and $R_\infty(\tau)\defeq R_\infty \otimes_{R_{_{\rbar_p}}} R_{_{\rbar_p}}^{\eta,\tau}$, and write $X_\infty$, $X_\infty(\tau)$, $\ovl{X}_\infty(\tau)$
for $\Spec R_\infty$, $\Spec R_\infty(\tau)$, and $\Spec \ovl{R}_\infty(\tau)$ respectively.

We write $\Mod(X_\infty)$ be the category of coherent sheaves over $X_\infty$ and $\Rep_{\cO}(\GL_n(\cO_p))$ be the category of topological $\cO[\GL_n(\cO_p)]$-modules which are finitely generated over $\cO$.

\begin{defn}\label{minimalpatching}
A \emph{weak patching functor} for an $L$-homomorphism $\rbar_p: G_{\Qp} \ra{}^L\un{G}(\F)$ is a nonzero covariant exact functor $M_\infty:\Rep_{\cO}(\GL_n(\cO_p))\ra \Coh(X_{\infty})$ such that for any tame inertial $L$-homomorphism $\tau$ and any $\cO$-lattice $\sigma(\tau)^\circ$ in $\sigma(\tau)$ one has:
\begin{enumerate}
\item
\label{support}
$M_\infty(\sigma(\tau)^\circ)$ is a maximal Cohen--Macaulay sheaf on $X_\infty(\tau)$; and
\item
\label{dimd}
for all $\sigma \in \JH(\ovl{\sigma(\tau)})$, $M_\infty(\sigma)$ is either zero or a maximal Cohen--Macaulay sheaf on $\ovl{X}_\infty(\tau)$.
\end{enumerate}
\end{defn}

Given a weak patching functor for an $L$-homomorphism $\rbar_p$ we thus define
\begin{equation}
\label{eq:Minfty}
W_{M_\infty}(\rbar_p)\defeq \left\{\sigma\mid \sigma \text{ is a $3(n-1)$-deep Serre weight of $\rG$ and } M_\infty(\sigma)\neq 0\right\}.
\end{equation}

By \cite[Proposition 2.4.5]{MLM} and \cite[Theorem 5.1.1, Proposition 5.4.1]{OBW} we see that if $\ovl{r}_p$ is $6(n-1)$-generic then condition that $\sigma$ is $3(n-1)$-deep in the right hand side of \eqref{eq:Minfty} is automatically satisfied.

\subsubsection{Modularity of defect one weights}
We now assume $S_p=\{v\}$ and $F^+_v=K$.
In particular a continuous homomorphism $\rhobar: G_{K} \ra{} \GL_n(\F)$ can be seen as an $L$-parameter, and have weak patching functors associated to it.

\begin{prop}
\label{prop:mod:def:1:abstract}
Let $\rhobar: G_{K} \ra{} \GL_n(\F)$ be $6(n-1)$-generic and let $M_\infty$ be a weak patching functor for $\rhobar$.
Then the following are equivalent:
\begin{enumerate}
\item $W^g_{\leq \un{0}}(\rhobar)\cap W_{M_\infty}(\rhobar)\neq \emptyset$; and
\item $W^g_{\leq \un{1}}(\rhobar)\subseteq W_{M_\infty}(\rhobar)$.
\end{enumerate}
\end{prop}
Before the proof, we record the following lemma, which will be also used in the proof of \cite[Conjecture 5.3.1]{PQ} in \S \ref{subsec:PQ:conj}.

\begin{lemma}
\label{lemma:nb:SW}
Let $\rhobar^{\speci}: G_{K} \ra{} \GL_n(\F)$ be $6(n-1)$-generic and let $\tld{w}\in\Adm^{\mathrm{reg}}(\eta)$ satisfy $\ell(t_\eta)-\ell(\tld{w}^{(j)})\leq  1$ for all $j\in\cJ$.
Let $\tau\defeq \tau_{\rhobar^{\speci}}(\tld{w})$ and write $\tld{w}=\tld{w}_2^{-1}w_0t_\nu\tld{w}_1$ where $\tld{w}_1,\tld{w}_2\in\tld{\un{W}}^+_1$ and $\nu\in X^+(\un{T})$ are uniquely determined up to $X^0(\un{T})$.
Then:
\begin{equation*}
\label{eq:nb:SW}
W^?(\rhobar^{\speci})\cap \JH(\ovl{\sigma(\tau)})=\left\{\sigma_{\rhobar^{\speci}}(\tld{w}')\mid \tld{w}^{\prime\,(j)}\in S^{(j)} \text{ for all $j\in\cJ$}
\right\}
\end{equation*}
where $S^{(j)}\defeq \left\{(\tld{w}_2^{(j)})^{-1}w_0t_{\nu^{(j)}}\tld{w}^{(j)}_1, (\tld{w}_h^{(j)}\tld{w}_1^{(j)})^{-1}w_0\tld{w}^{(j)}_1\right\}$. Moreover, $\#S^{(j)}=1$ if and only if $\tld{w}^{(j)}_1=\tld{w}_h^{-1}\tld{w}_2^{(j)}$ and $\nu^{(j)}\in X^0(T)$, if and only if $\ell(\tld{w}^{(j)})=\ell(t_\eta)$.
\end{lemma}

Note that the existence and uniqueness of the decomposition of $\tld{w}$ in the lemma is guaranteed by \cite[Proposition 2.1.5]{MLM}.

\begin{proof}
The relation $F_{(\tld{x},\omega)}\in W^?(\rhobar^{\speci})\cap \JH(\ovl{\sigma(\tau)})$ is equivalent, by \cite[Proposition 2.6.4]{MLM}, to the existence of a factorization $\tld{w}=(\tld{x}_2)^{-1}s\tld{x}_1$ with $\tld{x}_1,\tld{x}_2\in \tld{\un{W}}^+$, $\tld{x}_1\uparrow\tld{x}\uparrow\tld{w}_h^{-1}\tld{x}_2$ and $\omega=\tld{w}(\rhobar^{\speci})(\tld{x}_1)^{-1}(0)$.
By Lemma \ref{lem:fact:reg} applied to each $\tld{w}^{(j)}$ we have $s=w_0$ and for all $j\in\cJ$ either $\tld{x}^{(j)}=\tld{x}^{(j)}_1\in \tld{w}_1^{(j)}X^0(T)$ (in which case, by the uniqueness of the factorization in \cite[Proposition 2.1.5]{MLM} we further have $\tld{x}_2^{(j)}=t_{-w_0(\nu^{(j)})}\tld{w}_2^{(j)}$) or $\tld{x}^{(j)}=(\tld{w}^{(j)}_h)^{-1}\tld{x}^{(j)}_2\in (\tld{w}^{(j)}_h)^{-1}\tld{w}_2^{(j)}X^0(T)$ (in which case, by the uniqueness of the factorization in \cite[Proposition 2.1.5]{MLM} we further have $\tld{x}_1^{(j)}=t_{\nu^{(j)}}\tld{w}_1^{(j)}$).
The conclusion follows now from the definition of the map $\sigma_{\rhobar^{\speci}}$.
\end{proof}

\begin{proof}[Proof of Proposition \ref{prop:mod:def:1:abstract}]
The proof is by induction on the following quantity $\delta_{\rhobar}(\sigma)$ attached to a Serre weight $\sigma\in W^g_{\leq \un{1}}(\rhobar)$:
\begin{equation}
\label{eq:defect}
\delta_{\rhobar}(\sigma)\defeq\min \Big\{\sum_{j\in\cJ}\ell(t_\eta)-\ell(\tld{w}_{\rhobar^{\speci}}(\sigma)^{(j)})\mid\rhobar\leadsto \rhobar^{\speci}\Big\}.
\end{equation}
We fix throughout the proof a choice of an algebraic central character $\zeta$; all lowest alcove presentations below will be chosen to be compatible with $\zeta$.

By definition of $W^g_{\leq \un{1}}(\rhobar)$ and Lemma \ref{lem:reg:type}, for each $\sigma\in W^g_{\leq \un{1}}(\rhobar)$ there exist $\rhobar\leadsto\rhobar^{\speci}$ and a tame inertial type $\tau$ (depending on $\rhobar^{\speci}$) such that
\begin{enumerate}
\item\label{en:type:nice:1}
$\sigma\in W^g(\rhobar)\cap W^?(\rhobar^{\speci})\cap \JH_{\textnormal{out}}(\ovl{\sigma(\tau)})$;
\item\label{en:type:nice:2}
$\tld{w}_{\rhobar^{\speci}}(\sigma)=\tld{w}(\rhobar^{\speci},\tau)\in \Adm^{\textnormal{reg}}(\eta)$; and
\item\label{en:type:nice:3} $\tld{w}_{\rhobar^{\speci}}(\sigma')^{(j)}\geq \tld{w}_{\rhobar^{\speci}}(\sigma)^{(j)}$ for all $j\in\cJ$ and all $\sigma'\in  W^?(\rhobar^{\speci})\cap \JH(\ovl{\sigma(\tau)})$.
\end{enumerate}
(Note that condition \eqref{en:type:nice:2} determines $\tau$ uniquely, and we thus let $\tau=\tau_{\rhobar^{\speci}}(\tld{w}_{\rhobar^{\speci}}(\sigma))$ in the notation of Lemma \ref{lem:reg:type}.)
We fix $\rhobar\leadsto\rhobar^{\speci}$ such that $\tld{w}(\rhobar^{\speci},\tau)$ is maximal possible (which is equivalent to $\tld{w}_{\rhobar^{\speci}}(\sigma)$ realizing the minimum \eqref{eq:defect}), where $\tau$ is the type associated to $\rhobar^{\speci}$ satisfying \eqref{en:type:nice:1}, \eqref{en:type:nice:2}, and \eqref{en:type:nice:3} above.

We claim that with this choice of $\rhobar\leadsto \rhobar^{\speci}$ and $\tau$ we have
\begin{equation*}
\label{eq:shape:speci}
\tld{w}(\rhobar,\tau)=\tld{w}(\rhobar^{\speci},\tau).
\end{equation*}
Indeed, we always have $\tld{w}(\rhobar,\tau)\geq\tld{w}(\rhobar^{\speci},\tau)$ by \cite[Theorem 3.5.1]{OBW}.
If $\tld{w}(\rhobar,\tau)>\tld{w}(\rhobar^{\speci},\tau)$, then there exists a specialization $\rhobar\leadsto \rhobar^{\prime,\speci}$ such that $\tld{w}(\rhobar^{\prime,\speci},\tau)=\tld{w}(\rhobar,\tau)$ (in particular, $\tld{w}(\rhobar^{\prime,\speci},\tau)\in \Adm^{\textnormal{reg}}(\eta)$) and therefore $\tld{w}(\rhobar^{\prime,\speci},\tau)>\tld{w}(\rhobar^{\speci},\tau)$.
This contradicts the maximality of $\tld{w}(\rhobar^{\speci},\tau)$.

For each $\sigma\in W^g_{\leq \un{1}}(\rhobar)$ with corresponding $\rhobar\leadsto\rhobar^{\speci}$ and  $\tau$ as above, we now claim that
\begin{equation}
\label{eq:geom:wgt}
W^g(\rhobar)\cap \JH(\ovl{\sigma(\tau)})=W^?(\rhobar^{\speci})\cap \JH(\ovl{\sigma(\tau)})
\end{equation}
or, equivalently by \cite[Theorem 4.2.4]{LLL}, that any $\sigma'\in W^?(\rhobar^{\speci})\cap \JH(\ovl{\sigma(\tau)})$ is in $W^g(\rhobar)$.
Indeed by Lemma \ref{lem:reg:type} there exists $\sigma_0\in W^?(\rhobar^{\speci})\cap \JH(\ovl{\sigma(\tau)})$ such that
\begin{enumerate}
\item
\label{it:geom:1}
$\delta_{\rhobar}(\sigma_0)=0$; and
\item
\label{it:geom:2}
for all $\sigma'\in W^?(\rhobar^{\speci})\cap \JH(\ovl{\sigma(\tau)})$ we have $\tld{w}_{\rhobar^{\speci}}(\sigma')^{(j)}\in\{\tld{w}_{\rhobar^{\speci}}(\sigma_0)^{(j)},\, \tld{w}_{\rhobar^{\speci}}(\sigma)^{(j)}\}$.
\end{enumerate}
By item \eqref{it:geom:1} and \cite[Theorem 7.4.2]{MLM} (see also  \cite[Remark 3.9.1]{OBW}) we have $\sigma_0\subset W^g(\rhobar)$ and thus $\rhobar\in \cC_{\sigma}\cap \cC_{\sigma_0}$.
Recall from \cite[equation (4.11) and Definition 4.6.1]{MLM} the variety $\tld{C}^{\zeta}_{\sigma'}$ with its decomposition $\prod_{j\in\cJ}\tld{C}^{\zeta,(j)}_{\sigma'}$.
Each $\tld{C}^{\zeta,(j)}_{\sigma'}$ is determined explicitly by $\tld{w}(\rhobar^{\speci})^{(j)}$ and $\tld{w}_{\rhobar^{\speci}}(\sigma')^{(j)}$ (see \emph{loc.~cit}.~Definition 4.3.2, Theorem 4.3.9) so that $\tld{C}^{\zeta,(j)}_{\sigma'}=\tld{C}^{\zeta,(j)}_{\sigma}$ if $\tld{w}_{\rhobar^{\speci}}(\sigma')^{(j)}=\tld{w}_{\rhobar^{\speci}}(\sigma)^{(j)}$, and $\tld{C}^{\zeta,(j)}_{\sigma'}=\tld{C}^{\zeta,(j)}_{\sigma_0}$ if $\tld{w}_{\rhobar^{\speci}}(\sigma')^{(j)}=\tld{w}_{\rhobar^{\speci}}(\sigma_0)^{(j)}$.
Thus, by item \eqref{it:geom:2}, for all $\sigma'\in W^?(\rhobar^{\speci})\cap \JH(\ovl{\sigma(\tau)})$ we have
\[
\tld{C}^{\zeta}_{\sigma}\cap \tld{C}^{\zeta}_{\sigma_0}=\prod_{j\in\cJ} \tld{C}^{\zeta,(j)}_{\sigma}\cap \tld{C}^{\zeta,(j)}_{\sigma_0}
\subseteq\prod_{j\in\cJ}\tld{C}^{\zeta,(j)}_{\sigma'} =\tld{C}^{\zeta}_{\sigma'}
\]
which together with \cite[Theorem 7.4.2]{MLM} implies $\cC_{\sigma}\cap \cC_{\sigma_0}\subseteq \cC_{\sigma'}$.
This proves \eqref{eq:geom:wgt}.

We now proceed to the inductive argument.
We freely use the notation for cycles from patching functors introduced in \cite[\S 5.3]{OBW}, in particular we write $\overline{\textnormal{pr}}$ to indicate the projection map from cycles over the reduced union $\cup_{\tau\in \cT} \ovl{X}_\infty(\tau)$ (for a set of generic tame inertial $L$-parameters $\cT$) to cycles over the special fiber of the multi-type deformation ring associated to the set $\cT$.
The set $\cT$ can be fixed to satisfy condition (ii) in \cite[\S 5.3]{OBW} since all the tame inertial types $\tau$ involved in this proof satisfy $\ell(\tld{w}(\rhobar,\tau)^{(j)})\geq \ell(t_\eta)-1$ for all $j\in\cJ$, hence Theorem \ref{thm: main deformation ring} applies.
Furthermore, given $\sigma\in W^g(\rhobar)$ we will write $\cC_{\sigma}(\rhobar)$ to denote the pullback of the component $\cC_{\sigma}$ to the versal ring of $\cX_{n,K}$ at $\rhobar\in \cX_{n,K}(\F)$.

Let $\sigma\in W^g_{\leq \un{1}}(\rhobar)$.
We prove by induction on  $\delta_{\rhobar}(\sigma)$ that the support of $\ovl{\textnormal{pr}}\circ Z(M_\infty(\ovl{\sigma}(\tau)^\circ))$ contains $\cC_{\sigma}$, and that any \emph{other} component in the support of $\ovl{\textnormal{pr}}\circ Z(M_\infty(\ovl{\sigma}(\tau)^\circ))$ is of the form $\cC_{\kappa}$ with $\kappa \in W^g(\rhobar)$, $\delta_{\rhobar}(\kappa)< \delta_{\rhobar}(\sigma)$.

If $\delta_{\rhobar}(\sigma)=0$ then $\sigma \in W^g_{\leq \un{0}}(\rhobar)$, and  $W^g_{\leq \un{0}}(\rhobar)\cap W_{M_\infty}(\rhobar)\neq \emptyset$ implies $W_{\leq \un{0}}^g(\rhobar)\subseteq  W_{M_\infty}(\rhobar)$ by the main result of \cite{OBW}.
As $W^g(\rhobar)\cap \JH(\ovl{\sigma(\tau)})= W^?(\rhobar^{\speci})\cap \JH(\ovl{\sigma(\tau)})=\{\sigma\}$    by Lemma~\ref{lemma:nb:SW}) and $M_\infty(\ovl{\sigma}(\tau)^\circ)$ has full support over (a formally smooth modification of) $\ovl{R}_{\rhobar}^{\eta,\tau}$, we conclude that the cycle $\ovl{\textnormal{pr}}\circ Z(M_\infty(\ovl{\sigma}(\tau)^\circ))$ is supported on $\cC_\sigma(\rhobar)$.
(Here and below we write $\ovl{\sigma}(\tau)^\circ$ to denote the mod $\varpi$-reduction of any $\cO$-lattice $\sigma(\tau)^\circ$ in $\sigma(\tau)$.)

Assume now that $\delta_{\rhobar}(\sigma)>0$ and that for any $\sigma'\in W^g_{\leq \un{1}}(\rhobar)$ with $\delta_{\rhobar}(\sigma')<\delta_{\rhobar}(\sigma)$ the cycle $\ovl{\textnormal{pr}}\circ Z(M_\infty(\sigma'))$ is supported on $\cC_{\sigma'}(\rhobar)$ and possibly other components of the form $\cC_\kappa(\rhobar)$ where $\kappa \in W^g(\rhobar)$ and $\delta_{\rhobar}(\kappa)< \delta_{\rhobar}(\sigma')$.
We have
\begin{eqnarray}
\label{eq:eq:support}
&\textnormal{Supp}\big(Z(\ovl{R}_{\rhobar}^{\eta,\tau})\big)&=
\textnormal{Supp}\big(\textnormal{red}(Z(R_{\rhobar}^{\eta,\tau})[1/p])\big)
\\
\nonumber
&&=\textnormal{Supp}\big(
\textnormal{red}\circ \textnormal{pr} \big(Z(M_\infty(\sigma(\tau)^\circ)[1/p])\big)\big)
\\
\nonumber
&&=\textnormal{Supp}\big(
\ovl{\textnormal{pr}}\circ \textnormal{red}\big( Z(M_\infty(\sigma(\tau)^\circ)[1/p])\big)\big)
\\
\nonumber
&&=
\textnormal{Supp}\big(\ovl{\textnormal{pr}}\circ Z(M_\infty(\ovl{\sigma}(\tau)^\circ))\big)
\end{eqnarray}
where the first and last equality follow from \cite[Proposition 5.3.1]{OBW}, the second from the fact that $R_{\rhobar}^{\eta,\tau}$ is geometrically integral (Theorem \ref{thm: main deformation ring}), the third from $\textnormal{red}\circ \textnormal{pr}=\ovl{\textnormal{pr}}\circ \textnormal{red}$ (see \cite[\S 5.3]{OBW}).
Since $R_{\rhobar}^{\eta,\tau}$ is geometrically integral by Theorem \ref{thm: main deformation ring}, by exactness of $M_\infty$ and \cite[Theorem 5.1.1 and Proposition 5.4.1]{OBW} we have
\begin{align}
\label{eq:final:modularity}
\ovl{\textnormal{pr}}\circ Z(M_\infty(\ovl{\sigma}(\tau)^\circ))
&= \sum_{\sigma'\in W^?(\rhobar^{\speci})\cap \JH(\ovl{\sigma(\tau)})} \ovl{\textnormal{pr}}\circ Z(M_\infty(\sigma'))
\nonumber
\\
\nonumber
&=\ovl{\textnormal{pr}}\circ Z(M_\infty(\sigma))+\sum_{\sigma'\in W^?(\rhobar^{\speci})\cap \JH(\ovl{\sigma(\tau)})\setminus\{\sigma\}} \ovl{\textnormal{pr}}\circ Z(M_\infty(\sigma'))\nonumber
\end{align}
From the definition of $\delta_{\rhobar}$ together with Lemma~\ref{lem:reg:type}, we easily check that $\delta_{\rhobar}(\sigma')<\delta_{\rhobar}(\sigma)$ for all $\sigma'\in W^?(\rhobar^{\speci})\cap \JH(\ovl{\sigma(\tau)})\setminus \{\sigma\}$.
We thus deduce, using the inductive hypothesis and \eqref{eq:geom:wgt}, that
\[
\bigcup_{\sigma'\in W^?(\rhobar^{\speci})\cap \JH(\ovl{\sigma(\tau)})\setminus\{\sigma\}} \textnormal{Supp}\big(\ovl{\textnormal{pr}}\circ Z(M_\infty(\sigma'))\big)= \bigcup_{\sigma'\in W^?(\rhobar^{\speci})\cap \JH(\ovl{\sigma(\tau)})\setminus\{\sigma\}} \cC_{\sigma'}(\rhobar).
\]

On the other hand, we have
\[
\textnormal{Supp}\big(Z(\ovl{R}_{\rhobar}^{\eta,\tau})\big)=
\bigcup_{\sigma'\in W^g(\rhobar)\cap \JH(\ovl{\sigma(\tau)})} \cC_{\sigma'}(\rhobar)
\]
by \cite[Theorem 7.4.2]{MLM} (and  \cite[\href{https://stacks.math.columbia.edu/tag/0DRB}{Lemma 0DRB}, \href{https://stacks.math.columbia.edu/tag/0DRD}{Lemma 0DRD} and \href{https://stacks.math.columbia.edu/tag/0DRA}{Definition 0DRA}]{stacks-project}) and hence $\ovl{\textnormal{pr}}\circ Z(M_\infty(\sigma))$ is necessarily supported on $\cC_{\sigma}(\rhobar)$, and possibly on other components $\cC_{\sigma'}(\rhobar)$ with $\sigma' \in W^g(\rhobar)$ and $\delta_{\rhobar}(\sigma')< \delta_{\rhobar}(\sigma)$.
\end{proof}

\begin{rmk}
If, in the statement of Proposition \ref{prop:mod:def:1:abstract}, we furthermore assume that weak patching functor $M_\infty$ is \emph{minimal} (\cite[Definition 5.2.1]{OBW}), then equation \eqref{eq:eq:support} can be replaced with the stronger statement
\[
e(R_{\rhobar}^{\eta,\tau})=e(M_\infty(\ovl{\sigma}(\tau)^\circ))
\]
(where $e(\cdot)$ denotes the Hilbert--Samuel multiplicity), which forces $e(M_\infty(\sigma))=1$ for all $\sigma\in W^?(\rhobar^{\speci})\cap \JH(\ovl{\sigma(\tau)})$.
\end{rmk}

\begin{rmk}
\label{rmk:Lpar}
Let $\rbar_p: G_{\Qp}\ra{}^{L}\un{G}(\F)$ be an $L$-parameter, with corresponding collection $\{\rhobar_v: G_{F^+_v}\ra\GL_n(\F)\}_{v\in S_p}$.
For each $v\in S_p$ let $\cJ_v$ denote the set of ring homomorphisms $\{k_v\into \F\}$, so that $\cJ_p=\prod_{v\in S_p}\cJ_v$.
Given $\un{h}_v\in \N^{\cJ_v}$ for each $v\in S_p$ we then have a collection $\{W^g_{\leq \un{h}_v}(\rhobar_v)\}_{v\in S_p}$ whose elements are Serre weights for $\rG$ by taking tensor products over $v\in S_p$.
In particular, given $\un{h}\in \N^{\cJ_p}$ we can define the set $W^g_{\leq \un{h}}(\rbar_p)$ for an $L$-parameter $\rbar_p: G_{\Qp}\ra{}^{L}\un{G}(\F)$.
The proofs of Proposition \ref{prop:mod:def:1:abstract} and Lemma \ref{lemma:nb:SW} go through, \emph{mutatis mutandis}, replacing $\rhobar:G_K\ra\GL_n(\F)$ with an $L$-homomorphism $\rbar_p: G_{\Qp}\ra{}^{L}\un{G}(\F)$ (for these kind of passages from $\# S_p=1$ to $\# S_p>1$ see also  \cite[Remark 7.3.4]{MLM}).
\end{rmk}

\subsection{Global applications and the conjecture of \cite{PQ}} \label{subsec:PQ:conj}
We apply the results above to obtain our main global applications.
We follow the setup and notation of \cite[\S 5.5.1, 5.5.2, 5.5.3]{OBW}.
In particular, we have a totally real field $F^+/\Q$ not equal to $\Q$, and $F/F^+$ a CM extension.
We assume from now on that all places of $F^+$ above $p$ are unramified over $\Qp$, and that they are further totally split in $F$.
Given a reductive group $G_{/F^+}$ which is an outer form for $\GL_n$, split over~$F$, and such that $G(F^+\otimes_{\Q}\R)=U_n(F^+\otimes_{\Q}\R)$, we define for a continuous Galois representation $\rbar: G_{F^+}\ra \cG_n(\F)$ the notion of being \emph{automorphic} (relatively to $G_{/F^+}$) as in \cite[Definition 5.5.1]{OBW}, as well as the set $W(\rbar)$ of modular Serre weights of $\rbar$ (\emph{loc.~cit.}~Definition 5.5.2).
Here, $\cG_n$ is the group scheme defined in \cite[\S 2]{CHT}.
Finally, let $\rbar_p$ be the $L$-homomorhism induced from the collection of continuous representations $\rbar|_{G_{F^+_{v}}}:{G_{F^+_{v}}}\ra\GL_n(\F)$ for $v|p$. In particular, we denote by $S_p$ the set of the finite places of $F^+$ above $p$.

Let $\rbar: G_{F^+}\ra \cG_n(\F)$ be automorphic and such that $\rbar(G_{F(\zeta_p)})\subseteq \GL_n(\F)$ is adequate.
Then by \cite[Lemma 5.5.4]{OBW} we can and do fix a weak patching functor $M_\infty$ for the $L$-homomorphism $\rbar_p$ such that for any Serre weight $\sigma$ of $\rG=\prod_{v\in S_p}\GL_n(k_v)$ we have
\begin{equation}
\label{eq:Minfty:sigma}
M_\infty(\sigma)\neq 0\Longleftrightarrow \sigma\in W(\rbar).
\end{equation}

\begin{thm}[Modularity of defect one weights] \label{thm:globalobv}
Let $\rbar: G_{F^+} \ra \cG(\F)$ be an automorphic representation such that
\begin{itemize}
\item $\rbar(G_{F(\zeta_p)})\subseteq \GL_n(\F)$ is adequate; and
\item $\rbar_p$ is $6(n-1)$-generic.
\end{itemize}
Then the following are equivalent:
\begin{enumerate}
\item $W_{\leq \un{0}}^g(\rbar_p) \cap W(\rbar) \neq \emptyset$; and
\item $W_{\leq \un{1}}^g(\rbar_p) \subset W(\rbar)$.
\end{enumerate}
\end{thm}
\begin{proof}
This follows from Proposition  \ref{prop:mod:def:1:abstract} and Remark \ref{rmk:Lpar} using \eqref{eq:Minfty:sigma}.
\end{proof}

\begin{rmk}
Keep the setup and notation of Theorem \ref{thm:globalobv} and Remark \ref{rmk:Lpar}.
If $\sigma\in W^?_{\leq 1}(\rbar_p^{\speci})$ for some $\rbar_p \leadsto \rbar_p^{\speci}$, then $\sigma \in W_{M_{\infty}}(\rbar)$ if and only if $\sigma \in W^g(\rbar_p)$ if and only if $\sigma \in W^g_{\leq 1}(\rbar_p)$.
For each such $\sigma$, there exists $\tau$ such that
\begin{equation}
\label{eq:cond:WE}
\tld{w}(\rbar_p, \tau)=\tld{w}(\rbar_p^{\speci},\tau)\quad\text{and}\quad \ell(\tld{w}(\rbar_p^{\speci},\tau)^{(j)})\geq \ell(t_\eta)-1
\end{equation}
for each $j\in \cJ_p$ and $\sigma\in W^?(\rbar_p^{\speci})\cap \JH(\overline{\sigma}(\tau))$, with $\sigma\in W^g(\rbar_p)$ if and only if  $\rbar_p$ satisfies $Z_{-\alpha^{(j)}}=0$ for each $j\in\cJ_p$ for which the inequality in \eqref{eq:cond:WE} is an equality.
\end{rmk}

We now recall the setup of the local--global compatibility result of \cite{PQ}.
Assume that $p$ is totally split in $F$ and fix a place $w|p$ of $F$.
Assume that $\rbar_w\defeq \rbar|_{G_{F_w}}$ is Fontaine--Laffaille of niveau one, and that satisfies a geometric genericity condition dictated by its position in the moduli of Fontaine--Laffaille modules (see \cite[Definition 3.2.5]{PQ}).
In particular we have a lowest alcove presentation $(1,\mu)$ for $\rbar_w^{\semis}$ (with $\mu=(c_{n-1},c_{n-2},\dots,c_1,c_0)$ in the notation of \S 1 in \emph{loc.~cit}.) and a niveau one tame inertial type $\tau$ with lowest alcove presentation $(1,s_\alpha\big((\mu^{\Box,i_1,j_1})^\vee\big))$ (with $\mu^{\Box,i_1,j_1}\defeq \mu^{\vee}+(\langle\eta,\alpha^\vee\rangle+1)\alpha$) where $0\leq i_1,j_1\leq n-1$ corresponds to a positive root $\alpha=\alpha_{i_1+1,j_1+1}$.
Then by the proof of \cite[Lemma 3.4.1]{PQ} (namely, from the expression of $\Mat_{\mathfrak{e}''}(\phi)$  in \emph{loc.~cit}.) we see that $\tld{w}(\rbar_w,\tau)=t_{\eta-\alpha}s_\alpha$, which is a regular colength one shape.
In particular $\rbar_w\leadsto \rhobar^{\speci}$  with $\tld{w}(\rhobar^{\speci},\tau)=t_{\eta-\alpha}s_\alpha$.
By Lemma \ref{lemma:nb:SW} we obtain:
\begin{thm}[Conjecture 5.3.1 \cite{PQ}]
Let $w|p$. Assume that $\rbar_w$ is Fontaine--Laffaille of niveau one, that $(\rbar_w)_{n-i_0,n-j_0}$ is Fontaine--Laffaille generic in the sense of \cite[Definition 3.2.5]{PQ}, and moreover that $\rbar_w^{\semis}$ is $3(n-1)$-generic.
Then
\[
W_w(\rbar)\cap \JH(\ovl{\sigma}(\tau))\subseteq \{F(\mu)^\vee,\ F(\mu^{\Box,i_1,j_1})^\vee\}
\]
where $W_w(\rbar)$ denotes the set of modular weights for $\rbar$ at $w$ $($as defined in the paragraph just below \cite[Definition 5.2.2]{PQ}$)$.
\end{thm}

\begin{proof}
By \cite[Theorem~5.1.1]{OBW}, $W_w(\rbar)\cap \JH(\ovl{\sigma}(\tau))\subseteq W^?(\rhobar^\speci)\cap \JH(\ovl{\sigma}(\tau))$. Now, it follows immediately from Lemma~\ref{lemma:nb:SW}, as $\tld{w}(\rhobar^{\speci},\tau)=s_\alpha t_{\eta-\alpha}$.
\end{proof}

The above result removes the weight elimination condition of \cite[Theorem 5.6.2]{PQ}.
However the  fact that the deformation ring $R_{\rhobar}^{\eta,\tau}$ is formally smooth over $\cO$ when $Z_{-\alpha}\not\equiv 0$ modulo $\varpi$ makes the argument of \cite[Theorem 5.6.2]{PQ} more direct, since the modules of algebraic automorphic forms are in this case free over the Hecke algebra by patching arguments (cf.~\cite[Remark 5.4.6]{PQ}).

\bibliography{Biblio}
\bibliographystyle{amsalpha}

\end{document}